%% file: Manuscript.tex
\documentclass[11pt,letterpaper,twoside,reqno,nosumlimits]{amsart}

\input{macro-3}

\allowdisplaybreaks

\usepackage[numbers]{natbib}

\newcommand{\om}{\omega}
\newcommand{\la}{\langle}
\newcommand{\ra}{\rangle}
 
\newcommand{\vu}{\vct{u}}
\newcommand{\vx}{\vct{x}}
\newcommand{\vy}{\vct{y}}   
\newcommand{\vz}{\vct{z}}   

\mathtoolsset{showonlyrefs}

\begin{document}

\title[Contiguous Vortex-Patch Dipole of 2D Euler]{Steady Contiguous Vortex-Patch Dipole Solutions of \\the 2D Incompressible Euler Equation}
\author[D. Huang, J. Tong]{De Huang$^1$, Jiajun Tong$^2$}
\thanks{$^1$School of Mathematical Sciences, Peking University. E-mail: dhuang@math.pku.edu.cn}
\thanks{$^2$Beijing International Center for Mathematical Research, Peking University. E-mail: tongj@bicmr.pku.edu.cn}

\begin{abstract}
We rigorously construct the first steady traveling wave solutions of the 2D incompressible Euler equation that take the form of a contiguous vortex-patch dipole, which can be viewed as the vortex-patch counterpart of the well-known Lamb--Chaplygin dipole. Our construction is based on a novel fixed-point approach that determines the patch boundary as the fixed point of a certain nonlinear map. Smoothness and other properties of the patch boundary are also obtained.
\end{abstract}

\maketitle

\section{Introduction}
We are interested in vortex-patch traveling wave solutions of the vorticity formulation of the 2D incompressible Euler equation in $\R^2$:
\begin{equation}\label{eqt:2D_Euler}
\begin{split}
&\om_t + \vu\cdot \nabla\om = 0,\\
&\vu = -\nabla^{\perp}(-\Delta)^{-1}\om.
\end{split}
\end{equation}
Here, $\nabla^{\perp} = (-\partial_{x_2},\partial_{x_1})$. A traveling wave solution of \eqref{eqt:2D_Euler} takes the form $\om(\vx,t) = \om_0(\vx-\vct{c}t)$ for some constant (vector) speed $\vct{c}$ and some profile function $\om_0$. Without loss of generality, we may assume that the solution travels in the $x_1$-direction, i.e. $\om(x_1,x_2,t) = \om_0(x_1-ct,x_2)$. In a co-moving frame at the same speed $c$, the vorticity profile $\om_0$ solves the stationary 2D incompressible Euler equation:
\begin{equation}\label{eqt:steady_2D_Euler}
\begin{split}
&\nabla^{\perp}\psi\cdot \nabla\om = 0,\\
&-\Delta \psi = \om,\quad \lim_{|\vx|\rightarrow +\infty}-\partial_{x_2}\psi = c.
\end{split}
\end{equation}
As usual, $\psi$ is called the stream function. Hence, a traveling wave solution is also viewed as a steady solution, as its shape is unchanged over time.

While the global wellposedness of the 3D incompressible Euler equations remains a challenging open problem in fluid dynamics, the study of steady solutions of the 2D and 3D Euler equations has also been a focus of this field over the past century. In particular, an abundance of effort has been dedicated to the construction and stability analysis of nontrivial steady solutions with compactly supported vorticity, mostly motivated by intriguing observations from real physics experiments or numerical simulations that model rigid motions of vortices in fluids. 

In the three-dimensional case, the study of steady vortex ring solutions has been active for a long time since the work of Helmholtz \cite{von2006integrals}. Numerous results have been achieved regarding existence \cite{fraenkel1970steady,fraenkel1972examples,norbury1972steady,fraenkel1974global,ni1980,friedman1981vortex,ambrosetti1989existence,yang1995531,de2013desingularization,cao2021desingularization,abe2022existence,cao2022asymptotic} and uniqueness and stability \cite{amick1988uniqueness,wan1988variational,cao2019local,cao2022uniqueness,cao2023remarks} of steady vortex rings. Notably, in the family of steady vortex rings for the 3D axisymmetric Euler equations, the first explicit vorticity solution whose support is actually a ball was constructed by Hill \cite{hill1894vi}, known as Hill's spherical vortex, whose uniqueness and stability were established in followup works \cite{amick1986uniqueness,choi2024stability}.

In the past decades, there has also been a growing interest in the study of steady solutions of the 2D Euler equation with compactly supported vorticity. An important nontrivial explicit solution of this type is the Lamb--Chaplygin dipole, introduced by Chaplygin \cite{chaplygin2007one} (an English reprint of Chaplygin's original article) and Lamb \cite{lamb1924hydrodynamics} in the early 20th century. Its vorticity is supported in a disc and continuously vanishes at the boundary of the disc. This circular vortex dipole can be viewed as a two-dimensional analogue of Hill's spherical vortex. Here, a vortex dipole means a counter-rotating vortex pair supported in a symmetric bipartite domain. The uniqueness of the Lamb--Chaplygin circular dipole was established in \cite{burton1996uniqueness,burton2005isoperimetric}, and its orbital stability was recently proved in \cite{abe2022stability}. Moreover, existence \cite{norbury1975steady,turkington1983steady,burton1988steady,badiani1998existence,jianfu1991existence,smets2010desingularization,cao2021traveling,coti2023stationary,castro2023traveling} and stability \cite{wan1985nonlinear,burton2005global,burton2021compactness,cao2021nonlinear,choi2022stability,choi2022stability} of general nontrivial planar steady flows have been extensively studied as well. Among these solutions, the vortex-patch type solutions (with vorticity being compactly supported patch-wise constants) are of particular interest \cite{burbea1982motions,turkington1983steady,hmidi2013boundary,castro2016existence,castro2016uniformly,hmidi2017existence,hassainia2020global,gomez2021symmetry}, as they model localized strong rotations in planar flows. Some of these works not only consider the 2D Euler equation but also the family of the generalized Surface Quasi-Geostrophic equations. Most of these nontrivial steady vortex patches constructed so far are either a single rotating patch with a constant angular velocity (also called a V-state), or co-rotating multi-patches, or a counter-rotating traveling vortex-patch dipole, with the vorticity being constant in each connected component of its support. 

However, to the best of our knowledge, there has not been found theoretically a nontrivial (non-radial) example of compactly supported steady vortex-patch solution whose two or more patches with different (constant) values of vorticity touch each other. On the other hand, the existence of a traveling vortex-patch dipole with touching patches was suggested by an early numerical work of Sadovskii \cite{sadovskii1971vortex}. Essentially the same problem was also considered earlier by Gol'dshtik \cite{goldshtik1962mathematical} and Childress \cite{childress1966solutions}. Ever since Sadovskii's finding, there have been plenty of numerical studies \cite{pierrehumbert1980family,saffman1980equilibrium,saffman1982touching,wu1984steady,moore1988calculation,chernyshenko1993density,freilich2017sadovskii,childress2018eroding} that attempted to investigate boundary regularity, uniqueness, and other properties of this special vortex-patch solution, yet its existence has not been rigorously proved. Besides, Sadovskii's scenario was also observed in \cite{shariff2008contour,childress2016eroding,childress2018eroding} as an accurate approximation of the large-time asymptotic profile in the head-on collision of two anti-symmetric vortex rings.

In the present work, we address the long standing open problem regarding Sadovskii's scenario by proving the existence of traveling wave solutions of \eqref{eqt:2D_Euler} that take the form of a contiguous vortex-patch dipole, where the two patches with opposite signs of vorticity share a common boundary. 

\begin{theorem}\label{thm:main_informal}
There exists a simply connected domain $D \subset \R^2$ with positive area and $C^1$ boundary, centered at the origin and symmetric respect with to both axes, such that 
\[\om_D(\vx) = \mathrm{sgn}(x_2)\cdot \mathbf{1}_D\]
is a solution to \eqref{eqt:steady_2D_Euler} with $\psi$ satisfying
\[\psi(x_1,0)=0,\ x_1\in \R,\quad\text{and}\quad \lim_{|\vx|\rightarrow \infty}-\psi_{x_2}(\vx) = c_D,\] 
for some constant $c_D>0$. Moreover, the boundary of $D$ is analytic away from the $x_1$-axis.
\end{theorem}

A more detailed version of Theorem \ref{thm:main_informal} with finer characterizations of $D$ will be given in the next section (see Theorem \ref{thm:main_formal}). In particular, the boundary of $D$ in the first quadrant will be described as the graph of some monotone and locally analytic function $f$. It is also proved that the boundary of $D$ touches the $x_1$-axis vertically at the two endpoints (the left-most and the right-most points in $D$). Note that, by the scaling invariance of the Euler equation, the rescaled vortex-patch dipole
\begin{equation}\label{eqt:scale_invariance}
\om_{a,b}(\vx) = a\,\om_D(b\vx)
\end{equation}
is also a solution to \eqref{eqt:steady_2D_Euler} for any $a\in \R, b> 0$, with the speed constant changed to $c_{a,b} = (a/b)c_D$ correspondingly. Formally, our construction can be viewed as a vortex-patch counterpart of the Lamb--Chaplygin dipole, for the vorticity is also supported in one simply connected, symmetric domain and has opposite signs on the upper and lower halves of the domain. As a comparison, we illustrate a contiguous vortex-patch dipole and a Lamb--Chaplygin dipole in Figure \ref{fig:dipole} (a) and (b), respectively. Unlike the Lamb--Chaplygin dipole, the support of our vortex patch solution is not a perfect disc but an olive-shaped domain. We remark that the vorticity profile plotted in Figure \ref{fig:dipole}(a) is an accurate approximate solution obtained by numerical computation that will be explained in the final section of this paper.

\begin{figure}[!ht]
\centering
    \begin{subfigure}[b]{0.49\textwidth}
        \includegraphics[width=1\textwidth]{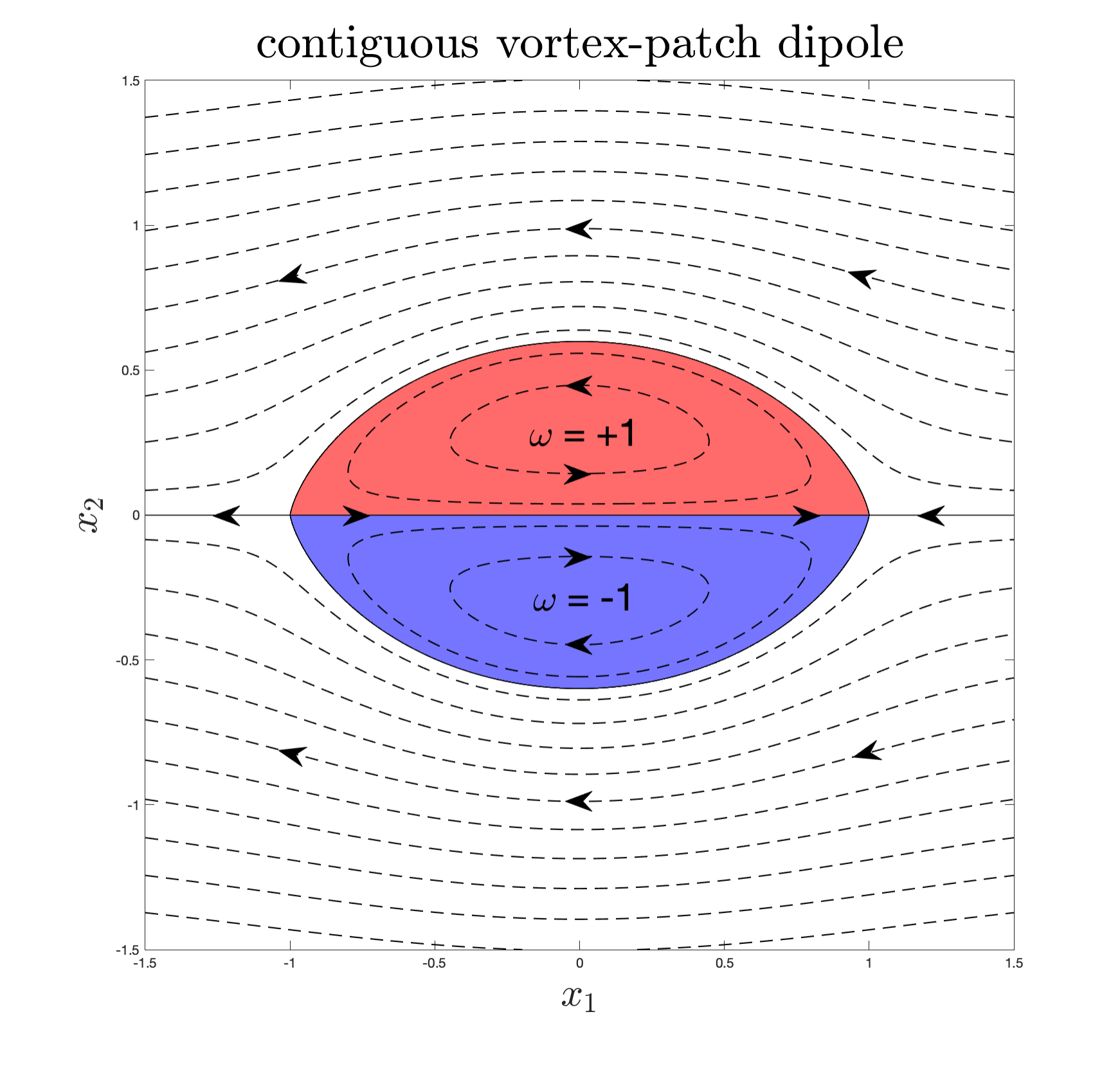}
        \caption{\small Contiguous vortex-patch dipole}
    \end{subfigure}
    \begin{subfigure}[b]{0.49\textwidth}
        \includegraphics[width=1\textwidth]{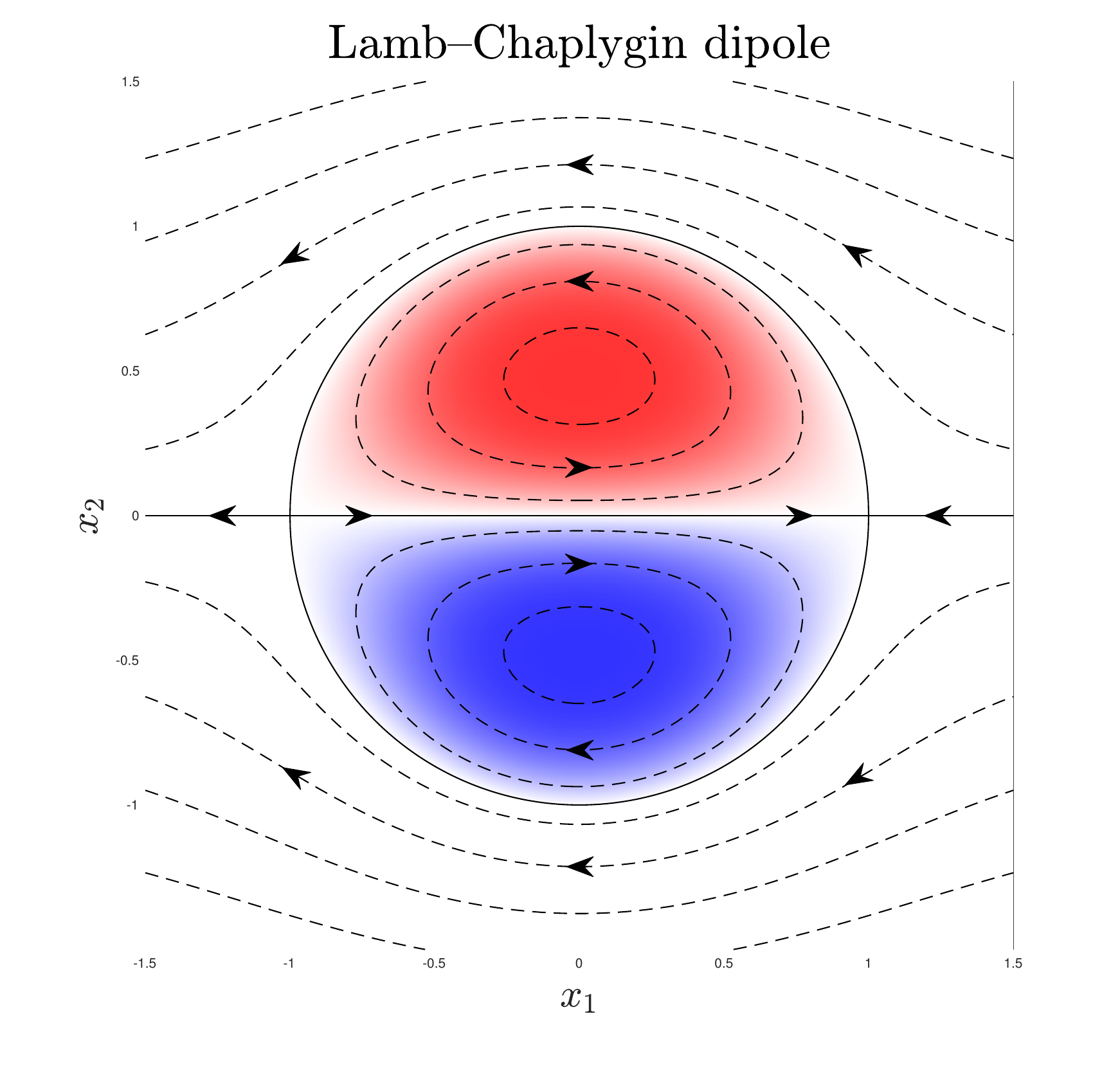}
        \caption{\small Lamb--Chaplygin dipole}
    \end{subfigure}
    \caption[Dipole]{(a) A contiguous vortex-patch dipole. (b) A Lamb--Chaplygin dipole. In both figures, the dashed lines are level sets of the corresponding stream function $\psi$, and the arrows represent the directions of the flow field.}
    \label{fig:dipole}
\end{figure}

In general, there are two major approaches of constructing nontrivial steady solutions to the 2D Euler equation. One is the variational approach that traces back to the general theory of steady vortex rings in 3D flows by Fraenkel--Berger \cite{fraenkel1974global}, and the other is the bifurcation approach following the idea of Burbea \cite{burbea1982motions}. In the variational framework, there are also two different perspectives, namely the stream function method and the vorticity method. The stream function method constructs a solution to the stationary 2D Euler equation \eqref{eqt:steady_2D_Euler} by solving the semilinear elliptic problem 
\begin{equation}\label{eqt:semilinear_elliptic}
-\Delta \psi = h(\psi),
\end{equation}
where $h$ is some non-decreasing function such that $h(t)=0$ for $t<0$; see e.g. \cite{norbury1975steady,smets2010desingularization,ambrosetti1990asymptotic,jianfu1991existence}. If one can find a solution $\psi$ to \eqref{eqt:semilinear_elliptic} for some prescribed function $h$, then $\om=h(\psi)$ gives a solution to \eqref{eqt:steady_2D_Euler}. In particular, such a solution corresponds to a vortex patch solution when $h$ is the Heaviside function. The existence of solutions to \eqref{eqt:semilinear_elliptic} for a given $h$ is usually proved by constructing $\psi$ as the maximizer of a variational problem related to the kinetic energy of the flow. As a variant of the stream function method, the vorticity method formulates the energy variational problem in terms of the vorticity $\om$ and focuses on finding a vorticity solution of \eqref{eqt:steady_2D_Euler} as the energy maximizer; see e.g. \cite{turkington1983steady,burton1988steady,badiani1998existence}.

The bifurcation approach constructs nontrivial solutions of \eqref{eqt:steady_2D_Euler} by perturbing simpler steady solutions, based on perturbation or bifurcation theories such as the Crandall--Rabinowitz theorem. For instance, Burbea \cite{burbea1982motions} constructed $m$-fold V-states using bifurcation from radial solutions, Hassainia--Masmoudi--Wheeler \cite{hassainia2020global} recently extended Burbea's result to global bifurcation of rotating vortex patches, Hmidi--Mateu \cite{hmidi2017existence} constructed counter-rotating vortex pairs as desingularization of a pair of point vortices with opposite circulations, and Castro--Lear \cite{castro2023traveling} constructed non-trivial planar traveling waves as bifurcation from the Couette flow. 

However, neither of the approaches above seems to be readily applicable in constructing a contiguous vortex-patch dipole solution like the one introduced in Theorem \ref{thm:main_informal}. On the one hand, although by symmetry we can formulate our problem as a free boundary problem of the form \eqref{eqt:semilinear_elliptic} (with $h$ being the Heaviside function) in the upper half-space, we do not get to prescribe the property that the support of $\om$ should touch the $x_1$-axis. On the other hand, a contiguous vortex-patch dipole does not seem to be the perturbation of any closed-form solution, and thus perturbation or bifurcation theories do not help with our problem.

Instead, we will construct our solution using a novel fixed-point method. Let us sketch the idea of our proof. Firstly, we reformulate the problem of finding the domain $D$ in Theorem \ref{thm:main_informal} into an equation of its upper boundary $\partial D\cap \{x_2\geq 0\}$ defined by a profile function $x_2=f(x_1)$, and we further transform the equation of $f$ into an equivalent fixed-point problem of a nonlinear map $\mtx{R}$. Secondly, we establish some useful estimates of $\mtx{R}$ and construct a proper function set $\mathbb{D}$ in which we can well control the behavior of $\mtx{R}$. Thirdly, we study the behavior of a continuous-time dynamical system in $\mathbb{D}$ induced by $\mtx{R}$ and prove the existence of time-periodic solutions with any positive period by the Schauder fixed-point theorem. Finally, we obtain a fixed point of $\mtx{R}$ as the limit of a sequence of time-periodic solutions with the period tending to $0$. Let us remark that our numerical experiments strongly suggest the vortex-patch solution $\om_{D}$ given in Theorem \ref{thm:main_informal} is unique up to rescaling and its support $D$ is convex. Unfortunately, we have not been able to prove these extra claims. We shall leave it to future works.

Upon the releasing of our paper, Choi--Jeong--Sim also posted an independent work \cite{choi2024existence} proving the existence of a steady contiguous vortex-patch dipole (they call it the Sadovskii vortex patch) using a totally different approach based on the variational framework. Their variational construction of the Sadovskii vortex patch as the kinetic maximizer naturally relates to its potential dynamical stability, while our fixed-point construction gives a better description of the patch boundary and in particular confirms the numerical observation that the touching angle is $90$ degrees (see \cite{sadovskii1971vortex}). Note that since neither of these works has proved the uniqueness of the Sadovskii vortex patch, it is still unclear whether their construction coincides with ours. Nevertheless, our numerical simulation strongly suggests uniqueness up to rescaling. The reader can also find a more comprehensive introduction to various topics related to Sadovskii's scenario in their paper.

The rest of this paper is organized as follows. In Section \ref{sec:main_result}, we set up the problem under suitable assumptions and present our main result with full details. We reformulate the problem into a well-defined fixed-point problem in Section \ref{sec:fixed_point}. Sections \ref{sec:existence} and \ref{sec:regularity} are devoted to establishing the existence and the regularity of a fixed-point solution, respectively. Finally, Section \ref{sec:numerical} explains how to numerically compute a fixed-point solution.

\section{Contiguous vortex-patch dipoles}\label{sec:main_result}
Our goal is to construct a nontrivial traveling wave solution $\om(x_1,x_2,t) = \om_D(x_1-ct,x_2)$ of the 2D incompressible Euler equation \eqref{eqt:2D_Euler} in $\R^2$, with the vorticity profile $\om_D$ given by a contiguous vortex-patch dipole,
\[\om_D(\vx) = \mathrm{sgn}(x_2)\cdot \mathbf{1}_D,\]
where $D\subset \R^2$ is a simply connected closed domain that is symmetric with respect to the $x_1$-axis. In particular, $\om_D(x_1,x_2)$ is an odd function of $x_2$. As a result, we may restrict our problem to the upper half-space $\mathbb{H} := \{(x_1,x_2): x_2\geq 0\}$.

Let $\Omega := D\cap \mathbb{H}$ denote the closed upper half of $D$ in $\mathbb{H}$. Then, our task is to find an appropriate $\Omega\subset \mathbb{H}$ such that
\begin{equation}\label{eqt:solution_form}
\om_\Omega = \mathbf{1}_\Omega
\end{equation}
solves (weakly) the stationary 2D incompressible Euler equation in the half-space $\mathbb{H}$,
\begin{equation}\label{eqt:steady_2D_Euler_half}
\begin{split}
&\nabla^{\perp}\psi\cdot \nabla\om_\Omega = 0,\\
&\psi = \phi - cx_2, \quad \phi = (-\Delta_{\mathbb{H}})^{-1}\om_\Omega,
\end{split}
\end{equation}
with some constant (speed) $c$ that depends on $\Omega$. Here, $\nabla^{\perp} = (-\partial_{x_2},\partial_{x_1})$, and $(-\Delta_{\mathbb{H}})^{-1}$ denotes the inverse of negative Laplacian on $\mathbb{H}$ subject to zero Dirichlet boundary condition on the $x_1$-axis and vanishing condition in the far field. Generally, $\psi$ is called the overall stream function, while $\phi$ is called the vorticity-induced stream function. More precisely, we have
\begin{equation}\label{eqt:phi_definition}
\phi(\vx) := (-\Delta_{\mathbb{H}})^{-1}\om_\Omega(\vx) = \frac{1}{2\pi}\int_{\mathbb{H}}\om_\Omega(\vy)\ln\frac{|\bar{\vx}-\vy|}{|\vx-\vy|}\idiff \vy = \frac{1}{2\pi}\int_\Omega\ln\frac{|\bar{\vx}-\vy|}{|\vx-\vy|}\idiff \vy,
\end{equation}
where $\vx = (x_1,x_2)$ and $\bar{\vx} = (x_1,-x_2)$, and hence, 
\[\phi_{x_1}(\vx) = \frac{1}{2\pi}\int_\Omega\left(\frac{x_1-y_1}{|\bar{\vx}-\vy|^2} - \frac{x_1-y_1}{|\vx-\vy|^2} \right)\idiff \vy,\]
\[\phi_{x_2}(\vx) = \frac{1}{2\pi}\int_\Omega\left(\frac{x_2+y_2}{|\bar{\vx}-\vy|^2}- \frac{x_2-y_2}{|\vx-\vy|^2} \right)\idiff \vy.\]

\subsection{Assumptions on $\Omega$}\label{sec:Omega_assumption} To make our problem more tractable, we make a few more assumptions on the desired domain $\Omega\subset \mathbb{H}$.

We assume that $\Omega$ is simply connected with positive finite area and that $\Omega$ is also symmetric with respect to the $x_2$-axis. This further symmetry assumption implies that the complete support $D$ is centered at the origin $(0,0)$ and symmetric with respect to both axes. More specifically, we assume that $\Omega$ is given by 
\[\Omega = \{(x_1,x_2):\, x_1\in [-1,1],\, x_2\in [0,f(x_1)]\},\]
where $f$ is some continuous even function on $[-1,1]$ satisfying $f(-1)=f(1)=0$ and $f(x)>0$ for $x\in(-1,1)$. In general, one can assume that $\Omega \cap \{(x_1,0): x_1\in\R\} = \{(x_1,0):x_1\in [-r,r]\}$ for some $r>0$. However, in view of the scaling property \eqref{eqt:scale_invariance}, it suffices to only consider the case $r=1$. In what follows, we will always denote by $\Gamma$ the upper boundary of $\Omega$, that is, 
\begin{equation}\label{eqt:Gamma_para}
\Gamma := \{(x_1,x_2):\, x_1\in[-1,1],\ x_2=f(x_1)\}.
\end{equation}

Now, constructing a (weak) solution to \eqref{eqt:steady_2D_Euler_half} of the form \eqref{eqt:solution_form} for some $\Omega$ satisfying the assumptions above boils down to finding an suitable function $f$ that defines the patch.

\subsection{Reformulation of the problem} It is well known that, in order for $\om_\Omega = \mathbf{1}_\Omega$ in \eqref{eqt:solution_form} to be a steady vortex-patch solution to \eqref{eqt:steady_2D_Euler_half}, its support boundary $\partial \Omega$ must be (part of) a level set of the stream function $\psi(\vx) = \phi(\vx) - cx_2$, with $\phi$ given in \eqref{eqt:phi_definition}. Since $\partial \Omega$ touches the $x_1$-axis where $\psi = 0$ by definition, we can conclude that $\psi=0$ along $\Gamma$. That is, for $\om_\Omega$ to be a (weak) solution of \eqref{eqt:steady_2D_Euler_half}, $\phi$ must satisfy
\[
\phi(x_1,x_2) - cx_2 = 0,\quad (x_1,x_2)\in \Gamma.
\]
By \eqref{eqt:Gamma_para}, this is equivalent to 
\begin{equation}\label{eqt:solution_condition}
\phi(x,f(x)) - cf(x) = 0,\quad x\in[-1,1].
\end{equation}
Here, $\phi$ and its derivatives are given by
\[\phi(\vx) = \frac{1}{2\pi}\int_{-1}^1\idiff y_1\int_0^{f(y_1)}\ln\frac{|\bar{\vx}-\vy|}{|\vx-\vy|}\idiff y_2,\]
\[\phi_{x_1}(\vx) = \frac{1}{2\pi}\int_{-1}^1\idiff y_1\int_0^{f(y_1)}\left(\frac{x_1-y_1}{|\bar{\vx}-\vy|^2} - \frac{x_1-y_1}{|\vx-\vy|^2}\right)\idiff y_2,\]
\[\phi_{x_2}(\vx) = \frac{1}{2\pi}\int_{-1}^1\idiff y_1\int_0^{f(y_1)}\left(\frac{x_2+y_2}{|\bar{\vx}-\vy|^2} - \frac{x_2-y_2}{|\vx-\vy|^2} \right)\idiff y_2.\]

From \eqref{eqt:solution_condition} and our assumptions on $f$, we immediately obtain a compatible condition that determines the speed constant $c$ from $\Omega$ or from the profile function $f$.

\begin{lemma}\label{lemma:Gamma_c}
Suppose that $\Omega\subset \R^2$ satisfies the assumptions in Section \ref{sec:Omega_assumption}. If $\om_\Omega = \mathbf{1}_\Omega$ is a solution of \eqref{eqt:steady_2D_Euler_half} with some constant $c=c_\Omega$, then
\begin{equation}\label{eqt:c_definition}
c = \phi_{x_2}(1,0) = \frac{1}{\pi}\int_{\Omega}\frac{y_2}{(1-y_1)^2+y_2^2}\idiff \vy,
\end{equation}
where $\phi$ is given by \eqref{eqt:phi_definition}. Moreover, in view of \eqref{eqt:Gamma_para}, 
\[c = c(f) = \frac{1}{\pi}\int_{-1}^1\idiff y_1\int_0^{f(y_1)}\frac{y_2}{(1-y_1)^2+y_2^2}\idiff y_2 = \frac{1}{2\pi}\int_{-1}^1 \ln \frac{(1-y_1)^2 + f(y_1)^2}{(1-y_1)^2}\idiff y_1.\]
\end{lemma}

\begin{proof}
By assumption, $f$ is continuous on $[-1,1]$, strictly positive on $(-1,1)$, and vanishes at $x=\pm1$. We can thus use \eqref{eqt:solution_condition} to compute that
\[c= \lim_{x\rightarrow 1-0}\frac{\phi(x,f(x))}{f(x)} = \phi_{x_2}(1,0).\]
We have used the fact that $\phi\in C^1(\mathbb{H})$ by the elliptic regularity theory applied to \eqref{eqt:phi_definition}. The second claim then follows immediately from \eqref{eqt:Gamma_para}.
\end{proof}

\subsection{Main result} We can now state our main result under the preceding setups.

\begin{theorem}\label{thm:main_formal}
There exists a simply connected domain $D \subset \R^2$ with positive area and $C^1$ boundary, centered at the origin and symmetric with respect to both axes, such that the followings hold. 

Let $\Omega = D\cap \mathbb{H}$ and let $c=c_\Omega$ be given in \eqref{eqt:c_definition}. Then $\om_\Omega = \mathbf{1}_\Omega$ solves \eqref{eqt:steady_2D_Euler_half} with $c$. As a result, $\om_D = \mathrm{sgn}(x_2)\cdot \mathbf{1}_D$ solves \eqref{eqt:steady_2D_Euler} with the same $c$.

Let $\Gamma = \partial D\cap \mathbb{H}$. Then $\Gamma$ is given by \eqref{eqt:Gamma_para} with some even function $f\in C[-1,1]\cap C^1(-1,1)$ satisfying the followings:
\begin{itemize}
\item $f(x)=f(-x)>0$ for $x\in(-1,1)$, and $f(1)=f(-1)=0$.
\item $f$ is locally analytic in $(-1,1)$.
\item $f'(0)=0$, $f'(x)<0$ for $x\in(0,1)$, and $\lim_{x\rightarrow 1^-}f'(x)=-\infty$. Moreover, it holds that
\[ - f'(x)\simeq x,\quad \text{for $x\in[0,1/2]$}, \]
and 
\[- f'(x)\simeq \log\left(1+\frac{1}{1-x}\right),\quad \text{for $x\in[1/2,1)$}.\]
Here, $a\simeq b$ means there are some absolute constants $C_1,C_2>0$ such that $C_1b\leq a\leq C_2b$.
\end{itemize}
\end{theorem}

Theorem \ref{thm:main_formal} is a complete version of Theorem \ref{thm:main_informal} that provides more detailed characterizations of the patch boundary. From the estimates of $f'$ we can see that the shape of $f$ is smooth and simple, and it behaves asymptotically like 
\[f(x) - f(0) \simeq -x^2\quad \text{for $x\sim 0$},\quad \text{and}\quad f(x)\simeq (1-x)\log\left(1+\frac{1}{1-x}\right)\quad \text{for $x\sim 1$}.\]  
In particular, the fact that $\lim_{x\rightarrow 1^-}f'(x)=-\infty$ implies the curve $\Gamma$ is perpendicular to the $x_1$-axis at the two endpoints $(\pm 1,0)$. However, the vertical behavior of $(1-x)\log(1+1/(1-x))$ at $x=1$ is not visually obvious as demonstrated in Figure \ref{fig:dipole}(a). We want to emphasize again that the contour of the contiguous vortex dipole plotted in Figure \ref{fig:dipole}(a) is obtained by an accurate numerical computation. Our numerical method will be introduced in Section \ref{sec:numerical}. 

The rest of the paper is dedicated to proving Theorem \ref{thm:main_formal} based on a fixed-point method. We will first reformulate \eqref{eqt:solution_condition} into a fixed-point problem of some nonlinear map. After establishing continuity, compactness, and other necessary properties of this nonlinear map, we will use the Schauder fixed-point theorem to prove the existence of a function $f$ that solves \eqref{eqt:solution_condition}. Finally, the regularity of $f$ will be studied using some estimates of the nonlinear map and classic elliptic regularity theory.

As a remark, it is strongly suggested by our numerical experiments (described in Section \ref{sec:numerical}) that this contiguous vortex-patch solution is unique under the normalization condition $D\cap \{(x_1,0): x_1\in\R\} = \{(x_1,0):x_1\in[-1,1]\}$, and that the support $D$ (or $\Omega$) is in fact a convex domain. However, we have not been able to prove these conjectures.

\section{A fixed-point formulation}\label{sec:fixed_point}
In this section, we reformulate our problem into an equivalent fixed-point problem of some nonlinear map. This map will be constructed by the implicit function theorem. In the next section, we will prove the existence of a fixed point using the Schauder fixed-point theorem.

Recall that our goal is to find a suitable continuous even function $f$ on $[-1,1]$ such that 
\begin{equation}\label{eqt:solution_condition_2}
\phi(x,f(x);f) - c(f)f(x) \equiv 0,\quad x \in[-1,1],
\end{equation}
where 
\begin{equation}\label{eqt:phi_definition_2}
\phi(\vx;f) := \frac{1}{2\pi}\int_{-1}^1\idiff y_1\int_0^{f(y_1)}\ln\frac{|\bar{\vx}-\vy|}{|\vx-\vy|}\idiff y_2,
\end{equation}
and
\begin{equation}\label{eqt:c_definition_2}
c(f) := \phi_{x_2}(1,0;f) = \frac{1}{\pi}\int_{-1}^1\idiff y_1\int_0^{f(y_1)}\frac{y_2}{(1-y_1)^2+y_2^2}\idiff y_2.
\end{equation}
We provide in Appendix \ref{apx:formulas} some useful alternative expressions of $\phi(\vx;f)$ and its derivatives that will be constantly used in the following.

\subsection{Basic function spaces} 
We shall restrict our search for the solution within some proper function spaces. For the underlying Banach space, we define
\[
\mathbb{V} := \{f\in C[-1,1]:\ f(x) = f(-x),\ f(1)=f(-1)=0\},
\]
endowed with the $L^\infty$-norm. In particular, we will study our fixed-point problem in some subsets of $\mathbb{V}$. Define 
\begin{equation}\label{eqt:M0_definition}
\mathbb{M}_0 := \{f\in \mathbb{V}:\ f(x)> 0\ \text{for}\ x\in(-1,1),\ \text{$f$ is non-increasing on $[0,1]$} \}.
\end{equation}
Sometimes we will work with functions in
\begin{equation}\label{eqt:M1_definition}
\mathbb{M}_1 := \{f\in \mathbb{M}_0\cap C^1(-1,1):\ f'(0)=0,\ f'(x)<0\ \text{for}\ x\in(0,1)\},
\end{equation}
which is dense in $\mathbb{M}_0$ under the $L^\infty$-norm. 

Note that neither $\mathbb{M}_0$ nor $\mathbb{M}_1$ is closed or compact in the $L^\infty$-topology. To employ the Schauder fixed-point theorem, eventually we will need to consider a smaller function set $\mathbb{D}$ that is convex, closed, and compact in $\mathbb{V}$. We delay the definition of this $\mathbb{D}$ until we finish proving some crucial properties of the nonlinear map to be constructed below.

\subsection{An implicit nonlinear map} 
Let us introduce a key quantity that will be the focus of our study. For a function $f\in \mathbb{M}_0$, define 
\begin{equation}\label{eqt:F_definition}
F(x_1,x_2;f) := \frac{\phi(x_1,x_2;f)}{x_2} - c(f),
\end{equation}
where $\phi(\vx;f)$ and $c(f)$ are given in \eqref{eqt:phi_definition_2} and \eqref{eqt:c_definition_2}, respectively. If $f$ is a solution to \eqref{eqt:solution_condition_2}, then it should also solve 
\[F(x,f(x);f)=0,\quad x\in[-1,1].\]
This inspires us to consider the following implicit iteration: given a function $f\in \mathbb{M}_0$, compute $F(\vx;f)$, and then find a new function $\tilde{f}$ on $[-1,1]$ such that
\begin{equation}\label{eqt:implicit_relation}
F(x,\tilde f(x);f)=0,\quad x\in[-1,1].
\end{equation}
If such $\tilde f$ is well-defined and unique, then this procedure defines an implicit update $f\mapsto\tilde f$, denoted by $\tilde f = \mtx{R}(f)$. We shall prove that this implicit map $\mtx{R}$ is indeed well-defined for all $f\in \mathbb{M}_0$. We remark that, this idea of finding the $0$-level set was used, for example, in \cite{childress2018eroding} as an implicit iteration scheme to numerically compute the profile of a contiguous vortex-patch dipole.

We first characterize the behavior of $F$ on the $x_1$-axis.

\begin{lemma}\label{lem:F_onR}
For any $f\in \mathbb{M}_0$, $F(\pm1,0;f)=0$ and $F_{x_1}(x,0;f)< 0$ for $x> 0$, $x\neq1$. As a result, $F(x,0;f)>0$ for $x\in(-1,1)$, and $F(x,0;f)<0$ for $x\notin[-1,1]$.
\end{lemma}

\begin{proof} 
Define $h(x) = F(x,0;f)=\phi_{x_2}(x,0;f)-c(f)$. Since $h(x)$ is an even function of $x$, we only need to consider $x\geq 0$. By the formula \eqref{eqt:phi_x2}, we have
\[
h(x) = \frac{1}{\pi}\int_{-1}^1\idiff y_1\int_0^{f(y_1)}\frac{y_2}{(x-y_1)^2+y_2^2}\idiff y_2 -c(f) =\frac{1}{2\pi}\int_{-1}^1\ln\frac{(x-y)^2+f(y)^2}{(x-y)^2}\idiff y -c(f).
\]
It is clear that $h(1) =c(f) -c(f) = 0$. Moreover, for any $x\in(0,1)$, we have
\begin{align*}
h'(x) &= \frac{1}{\pi}\mathrm{P.V.}\int_{-1}^1 \frac{1}{y-x}\cdot \frac{f(y)^2}{(x-y)^2+f(y)^2}\idiff y \\
&\leq \frac{1}{\pi}\mathrm{P.V.}\int_{2x-1}^{1}\frac{1}{y-x}\cdot \frac{f(y)^2}{(x-y)^2+f(y)^2}\idiff y\\
&= \frac{1}{\pi}\mathrm{P.V.}\int_{x-1}^{1-x}\frac{1}{\tau}\cdot \frac{f(\tau+x)^2}{\tau^2+f(\tau+x)^2}\idiff \tau\\
&= \frac{1}{\pi}\int_0^{1-x}\frac{1}{\tau}\cdot \left(\frac{f(x+\tau)^2}{\tau^2+f(x+\tau)^2} - \frac{f(x-\tau)^2}{\tau^2+f(x-\tau)^2}\right)\idiff \tau\leq 0.
\end{align*}
We have used the assumption that $f$ is non-increasing on $[0,1]$, which implies
\[\frac{f(x+\tau)^2}{\tau^2+f(x+\tau)^2} - \frac{f(x-\tau)^2}{\tau^2+f(x-\tau)^2}\leq 0,\quad \tau\in[0,1-x].\]
Note that all inequalities above are simultaneously equality only if $f\equiv 0$, which cannot happen for $f\in\mathbb{M}_0$. This means $h'(x)<0$ for $x\in(0,1)$. One can similarly show that $h'(x)<0$ for $x> 1$. Hence, $h(x) >0$ for $x\in(-1,1)$ and $h(x)<0$ for $x\notin[-1,1]$.
\end{proof}

Next, we study the behavior of $F$ away from the $x_1$-axis.

\begin{lemma}\label{lem:F_partial}
For any $f\in \mathbb{M}_0$, $F_{x_1}(\vx;f)$ and $F_{x_2}(\vx;f)$ are continuous in $\vx=(x_1,x_2)$ for $x_1\in\R, x_2>0$. Moreover, $F_{x_1}(0,x_2;f)=0$ for all $x_2> 0$, and $F_{x_1}(x_1,x_2;f)<0$, $F_{x_2}(x_1,x_2;f)<0$ for all $x_1,x_2>0$.
\end{lemma}

\begin{proof}
Write $F(\vx)=F(\vx;f)$ and $\phi(\vx)=\phi(\vx;f)$. It is clear that $\phi_{x_1}(\vx)$ and $\phi_{x_2}(\vx)$ are continuous in $\vx$ by the elliptic regularity theory applied to \eqref{eqt:phi_definition} (or simply by their integral expressions in \eqref{eqt:phi_x1} and \eqref{eqt:phi_x2}). The continuity of $F_{x_1}(\vx)$ and $F_{x_2}(\vx)$ in $\vx$ for $x_1\in\R, x_2>0$ thus follows. 

Since the set $\mathbb{M}_1$ \eqref{eqt:M1_definition} is dense in $\mathbb{M}_0$ \eqref{eqt:M0_definition}, we only need to prove the theorem for $f\in \mathbb{M}_1$. For any $f\in \mathbb{M}_1$, since $f'(x)<0$ for $x\in (0,1]$, its inverse function $f^{-1}(s)$ is well-defined for $s\in [0,f(0)]$.

Recall the definition \eqref{eqt:F_definition} of $F$. To prove $F_{x_1}(x_1,x_2)\leq 0$, we only need to show that $\phi_{x_1}(x_1,x_2)\leq 0$. In fact, for $x_1\geq 0$, $x_2>0$, we can invoke \eqref{eqt:phi_x1_inverse} to obtain
\begin{align*}
\phi_{x_1}(x_1,x_2) &= -\frac{1}{2\pi}\int_0^{f(0)}\frac{1}{2}\ln\frac{\big((x_1-f^{-1}(y))^2+(x_2+y)^2\big)\big((x_1+f^{-1}(y))^2+(x_2-y)^2\big)}{\big((x_1-f^{-1}(y))^2+(x_2-y)^2\big)\big((x_1+f^{-1}(y))^2+(x_2+y)^2\big)}\idiff y\\
&= -\frac{1}{4\pi}\int_0^{f(0)}\ln\left(1+\frac{16x_1x_2f^{-1}(y)y}{\big((x_1-f^{-1}(y))^2+(x_2-y)^2\big)\big((x_1+f^{-1}(y))^2+(x_2+y)^2\big)}\right)\idiff y\\
&\leq 0.
\end{align*}
The last inequality above is an equality if and only if $x_1=0$. Hence, $F_{x_1}(x_1,x_2)<0$ for all $x_1,x_2>0$. For general $f\in \mathbb{M}_0$, the same claim follows by approximation or by generalization of $f^{-1}$ and \eqref{eqt:phi_x1_inverse} for monotone functions. We omit the details.

As for $F_{x_2}$, we first compute that 
\begin{align*}
\frac{\phi(x_1,x_2)}{x_2} &= \frac{1}{2\pi}\int_{-1}^1\idiff y_1\int_0^{f(y_1)}\frac{1}{y_2}\cdot \frac{y_2}{x_2}\ln\frac{|\bar{\vx}-\vy|}{|\vx-\vy|}\idiff y_2\\
&= \frac{1}{2\pi}\int_{-1}^1\idiff y_1\int_0^{f(y_1)}\frac{1}{y_2}\cdot \left(-\partial_{y_2}K(\vx;\vy)\right) \idiff y_2\\
&= -\frac{1}{2\pi}\int_{-1}^1 \frac{K(\vx;y_1,f(y_1))}{f(y_1)}\idiff y_1 - \frac{1}{2\pi}\int_{-1}^1\idiff y_1\int_0^{f(y_1)}\frac{1}{y_2^2}\cdot K(\vx;\vy)\idiff y_2,
\end{align*}
where 
\begin{align*}
K(\vx;\vy) &:= \frac{x_2^2 - y_2^2 - (x_1-y_1)^2}{2x_2}\ln\frac{|\bar{\vx}-\vy|}{|\vx-\vy|} - y_2 \\
&\qquad + (x_1-y_1)\arctan\left(\frac{x_2+y_2}{x_1-y_1}\right) - (x_1-y_1)\arctan\left(\frac{x_2-y_2}{x_1-y_1}\right).
\end{align*}
Note that 
\[\frac{K(\vx;\vy)}{y_2^3} = -\frac{2}{3}\cdot\frac{1}{(x_1-y_1)^2+x_2^2} + O(y_2),\]
and 
\begin{align*}
\partial_{x_2}K(\vx;\vy) &= \frac{(x_1-y_1)^2+x_2^2+y_2^2}{2x_2^2}\ln\frac{|\bar{\vx}-\vy|}{|\vx-\vy|} -\frac{y_2}{x_2}\\
&= \frac{y_2}{x_2}\left(\frac{(x_1-y_1)^2+x_2^2+y_2^2}{4x_2y_2}\ln\frac{1+\frac{2x_2y_2}{(x_1-y_1)^2+x_2^2+y_2^2}}{1-\frac{2x_2y_2}{(x_1-y_1)^2+x_2^2+y_2^2}}-1\right).
\end{align*}
Define 
\[\eta(a) = \frac{1}{2a}\ln\frac{1+a}{1-a} - 1,\quad a\in[0,1).\]
It holds that $\eta(a)>0$ for $a\in (0,1)$, since
\[
\frac{1}{2a}\ln\frac{1+a}{1-a} = \frac{1}{2a}\int_{1-a}^{1+a} \frac{1}{s}\idiff s > 1.
\]
Hence, for any $x_2,y_2>0$,
\[\partial_{x_2}K(\vx;\vy) = \frac{y_2}{x_2}\cdot \eta\left(\frac{2x_2y_2}{(x_1-y_1)^2+x_2^2+y_2^2}\right)>0\]
Then, for any $x_2>0$, we have
\[F_{x_2}(x_1,x_2) = -\frac{1}{2\pi}\int_{-1}^1 \frac{\partial_{x_2}K(\vx;y_1,f(y_1))}{f(y_1)}\idiff y_1 - \frac{1}{2\pi}\int_{-1}^1\idiff y_1\int_0^{f(y_1)}\frac{\partial_{x_2}K(\vx;\vy)}{y_2^2}\idiff y_2<0.\]
This completes the proof.
\end{proof}

With the monotone behavior of $F$ in hand, we can now justify that the implicit relation \eqref{eqt:implicit_relation} properly defines a map on $\mathbb{M}_0$.

\begin{proposition}\label{prop:implicit_existence}
For any $f\in \mathbb{M}_0$, there exists a unique function $\tilde{f}\in \mathbb{M}_1$ such that $F(x,\tilde f(x);f)=0$ for all $x\in [-1,1]$.
\end{proposition}

\begin{proof}
Write $F(x_1,x_2) = F(x_1,x_2;f)$. By Lemma \ref{lem:F_onR} and Lemma \ref{lem:F_partial}, $F(x_1,x_2) < 0$ for all $|x_1|>1$ and $x_2\geq 0$. Hence, $F(x_1,x_2)=0$ can only happen for $x_1\in[-1,1]$. 

When $x_1=1$, we have $F(1,0) = 0$ and $F_{x_2}(1,x_2)<0$ for $x_2>0$. Hence, $(1,0)$ is the unique solution of $F(x_1,x_2)=0$ on the line $x_1=1$. The same happens for $x_1=-1$.

For any fixed $x_1\in(-1,1)$, we have $F(x_1,0) >0$, $F(x_1,+\infty) = -c(f)<0$, and $F_{x_2}(x_1,x_2) <0$ for $x_2>0$. Hence, there is a unique point $(x_1,x_2)$ with $x_2>0$ such that $F(x_1,x_2)=0$. 

Therefore, we can uniquely determine a function $\tilde f(x_1)$ by the relation $F(x_1,\tilde f(x_1))=0$ for any $x_1\in [-1,1]$. Such $\tilde f$ satisfies $\tilde f(\pm 1)=0$ and $\tilde f(x)>0$ for $x\in(-1,1)$. Moreover, the Implicit Function Theorem and Lemma \ref{lem:F_partial} together guarantee that $\tilde f\in C^1(-1,1)$, 
\[\tilde f'(0) =-\frac{F_{x_1}(0,\tilde f(0))}{F_{x_2}(0,\tilde f(0))}=0,\quad \text{and}\quad \tilde f'(x) = -\frac{F_{x_1}(x,\tilde f(x))}{F_{x_2}(x,\tilde f(x))} <0, \quad x\in(0,1).\]
The lemma is thus proved. 
\end{proof}

From now on, we will denote by $\mtx{R}$ the implicit mapping $f\mapsto \tilde f$ determined by \eqref{eqt:implicit_relation}:  
\[\tilde f = \mtx{R}(f),\quad f\in \mathbb{M}_0.\]
Our task is to show that this nonlinear map $\mtx{R}$ admits a fixed point in some proper subset of $\mathbb{M}_0$, as it gives a \textit{nontrivial} solution to \eqref{eqt:solution_condition_2}. Note that the function $f\equiv0$ is obviously a trivial but undesirable solution to \eqref{eqt:solution_condition_2}. This is why we purposely exclude $f\equiv0$ from $\mathbb{M}_0$.

\section{Existence of a fixed point}\label{sec:existence}
We will use the Schauder fixed-point theorem to prove the existence of a nontrivial fixed point of $\mtx{R}$. Ideally, we should construct a suitable invariant set of $\mtx{R}$ that is convex, closed, and compact in the $L^\infty$-topology of the underlying Banach space $\mathbb{V}$. Then, the existence of a fixed point should follow by the Schauder fixed-point theorem. However, this is not how we approach this problem, as we find it difficult to construct such an invariant set due to the non-local nature of $\mtx{R}$. Instead, we choose to work with a continuous-time dynamical system induced by $\mtx{R}$,
\begin{equation}\label{eqt:dynamic_preview}
\partial_t f = \mtx{R}(f) - f,
\end{equation}
whose equilibria are fixed points of $\mtx{R}$. Firstly, we still construct a suitable function set $\mathbb{D}\subset \mathbb{V}$ that is convex, closed, and compact in the $L^\infty$-topology, based on some crucial estimates of $\mtx{R}$. Secondly, we show that the solution $f(x,t)$ of \eqref{eqt:dynamic_preview} stays in $\mathbb{D}$ for all $t\geq0$ if the initial state $f(x,0)$ lies in $\mathbb{D}$, which relies on the well constrained behavior of $\mtx{R}$ on $\mathbb{D}$. In particular, we construct an upper barrier and a lower barrier for $f(x,t)$ that are preserved by \eqref{eqt:dynamic_preview}. In view of \eqref{eqt:implicit_relation}, the choices of these barrier functions are motivated by derivative estimates for $F$ near the end points. Finally, we use the Schauder fixed-point theorem to show that \eqref{eqt:dynamic_preview} admits time-periodic solutions for any positive period, and we obtain a steady state of \eqref{eqt:dynamic_preview} (and thus a fixed point of $\mtx{R}$ in $\mathbb{D}$) by taking the period to zero. We remark that, since $f\equiv 0$ is a trivial fixed point of $\mtx{R}$, it is important to make $\mathbb{D}$ well separated from $f\equiv 0$ by construction.

\subsection{Derivative estimates of $F$} In this subsection, we work out some useful bounds for the spatial derivatives of $F(\vx;f)$ for $f\in\mathbb{M}_0$, extending the monotonicity results in Lemma \ref{lem:F_partial} to some quantitative estimates that are essential for understanding the behavior of the map $\mtx{R}$.

\begin{lemma}\label{lem:Fx1_estimate}
There is some uniform constant $C>0$ such that, for any $f\in\mathbb{M}_0$ and any $x_1\in\R, x_2>0$,
\[|F_{x_1}(x_1,x_2;f)|\leq C\min\left\{1 + \ln\left(1+\frac{f(0)}{x_2}\right)\,,\,\frac{|x_1|}{x_2}\right\}.\]
\end{lemma}

\begin{proof}
We may assume $f\in\mathbb{M}_1$. We only need to consider $x_1\geq 0$, where $F_{x_1}(\vx;f) = \phi_{x_1}(\vx;f)/x_2 \leq 0$. Using \eqref{eqt:phi_x1_inverse} we obtain
\begin{align*}
-\phi_{x_1}(x_1,x_2;f) &= \frac{1}{2\pi}\int_0^{f(0)}\frac{1}{2}\ln\frac{\big((x_1-f^{-1}(y))^2+(x_2+y)^2\big)\big((x_1+f^{-1}(y))^2+(x_2-y)^2\big)}{\big((x_1-f^{-1}(y))^2+(x_2-y)^2\big)\big((x_1+f^{-1}(y))^2+(x_2+y)^2\big)}\idiff y\\
&= \frac{1}{4\pi}\int_0^{f(0)}\ln\left(1+\frac{16x_1x_2yf^{-1}(y)}{\big((x_1-f^{-1}(y))^2+(x_2-y)^2\big)\big((x_1+f^{-1}(y))^2+(x_2+y)^2\big)}\right)\idiff y.
\end{align*}
To obtain the first upper bound, we compute that 
\begin{align*}
-\frac{\phi_{x_1}(x_1,x_2;f)}{x_2} &\leq \frac{1}{4\pi}\int_0^{f(0)}\frac{1}{x_2}\ln\left(1+\frac{4x_2y}{(x_1-f^{-1}(y))^2+(x_2-y)^2}\right)\idiff y\\
&\leq  \frac{1}{4\pi}\int_0^{f(0)}\frac{1}{x_2}\ln\left(1+\frac{4x_2y}{(x_2-y)^2}\right)\idiff y\\
&=  \frac{1}{4\pi}\int_0^{f(0)/x_2}\ln\left(1+\frac{4s}{(1-s)^2}\right)\idiff s\\
&= \frac{1}{2\pi}\left(\left(1+\frac{f(0)}{x_2}\right)\ln\left(1+\frac{f(0)}{x_2}\right)+\left(1-\frac{f(0)}{x_2}\right)\ln\left|1-\frac{f(0)}{x_2}\right|\right)\\
&\leq C\left(1+ \ln\left(1+\frac{f(0)}{x_2}\right)\right).
\end{align*}
As for the second bound, we use a different estimate for the integrand to get
\begin{align*}
-\frac{\phi_{x_1}(x_1,x_2;f)}{x_2} &\leq \frac{1}{4\pi}\int_0^{f(0)}\frac{1}{x_2}\ln\left(1+\frac{8x_1y}{(x_1-f^{-1}(y))^2+(x_2-y)^2}\right)\idiff y\\
&\leq  \frac{1}{4\pi}\int_0^{f(0)}\frac{1}{x_2}\ln\left(1+\frac{8x_1y}{(x_2-y)^2}\right)\idiff y\\
&\leq  \frac{1}{4\pi}\cdot \frac{x_1}{x_2}\int_{\R}\ln\left(1+\frac{8}{s^2}\right)\idiff s \leq C\frac{x_1}{x_2}.
\end{align*}
The claimed bounds thus hold.
\end{proof}

\begin{lemma}\label{lem:Fx1_estimate_2}
Given any $f\in\mathbb{M}_0$, if there is some constant $d>0$ such that $f(x)\geq d(1-|x|)$ on $[-1,1]$, then for all $x_1\in[-1,1]$, $x_2> 0$, 
\[|F_{x_1}(x_1,x_2;f)|\geq C_d \frac{|x_1|}{1+(x_2+f(0))^2} \ln\left(1 + \frac{1}{(1-|x_1|)^2+x_2^2}\right),\]
where $C_d>0$ is a constant that only depends on $d$.
\end{lemma}

\begin{proof}
We may assume $f\in\mathbb{M}_1$. Again, we use \eqref{eqt:phi_x1_inverse} to obtain, for $x_1\in[0,1]$, $x_2>0$,  
\begin{align*}
-\phi_{x_1}(x_1,x_2;f) &= \frac{1}{4\pi}\int_0^{f(0)}\ln\left(1+\frac{16x_1x_2yf^{-1}(y)}{\big((x_1-f^{-1}(y))^2+(x_2-y)^2\big)\big((x_1+f^{-1}(y))^2+(x_2+y)^2\big)}\right)\idiff y\\
&\geq \frac{1}{4\pi}\int_{0}^{f(0)}\ln\left(1+\frac{4x_1x_2yf^{-1}(y)}{\big((1-x_1)^2+(1-f^{-1}(y))^2+x_2^2+y^2\big)\big(1+(x_2+f(0))^2\big)}\right)\idiff y\\
&\geq \frac{1}{4\pi}\int_{0}^{f(0)}\frac{2x_1x_2yf^{-1}(y)}{\big((1-x_1)^2+(1-f^{-1}(y))^2+x_2^2+y^2\big)\big(1+(x_2+f(0))^2\big)}\idiff y\\
&\geq \frac{1}{2\pi}\int_{0}^{f(0)}\frac{x_1x_2yf^{-1}(y)}{\big((1-x_1)^2+(y/d)^2+x_2^2+y^2\big)\big(1+(x_2+f(0))^2\big)}\idiff y\\
&= \frac{d^2}{1+d^2}\cdot \frac{x_1x_2}{2\pi\big(1+(x_2+f(0))^2\big)}\int_{0}^{f(0)}\frac{yf^{-1}(y)}{y^2 + \frac{d^2}{1+d^2}\big((1-x_1)^2+x_2^2\big)}\idiff y.
\end{align*}
We have used the fact $\ln(1+a)\geq a/2$ for $a\in[0,1]$. Besides, the last inequality above holds because $1\geq f^{-1}(y)\geq 1-y/d$, which follows from the assumption $f(x)\geq d(1-x)$ for $x\in[0,1]$. Then, since $f^{-1}(y)\geq 1/2$ for $y\in[0,f(1/2)]$, we find
\begin{align*}
-F_{x_1}(x_1,x_2;f) &= -\frac{\phi_{x_1}(x_1,x_2;f)}{x_2}\\
&\geq \frac{d^2}{1+d^2}\cdot \frac{x_1}{4\pi\big(1+(x_2+f(0))^2\big)}\int_{0}^{f(1/2)}\frac{y}{y^2 + \frac{d^2}{1+d^2}\big((1-x_1)^2+x_2^2\big)}\idiff y\\
&= \frac{d^2}{1+d^2}\cdot \frac{x_1}{8\pi\big(1+(x_2+f(0))^2\big)} \ln\left(1 + \frac{(1+d^2)f(1/2)^2}{d^2\big((1-x_1)^2+x_2^2\big)}\right)\\
&\geq \frac{d^2}{1+d^2}\cdot \frac{x_1}{8\pi\big(1+(x_2+f(0))^2\big)} \ln\left(1 + \frac{1/4}{(1-x_1)^2+x_2^2}\right)\\
&\geq C_d \frac{x_1}{1+(x_2+f(0))^2} \ln\left(1 + \frac{1}{(1-x_1)^2+x_2^2}\right).
\end{align*}
The lemma is thus proved.
\end{proof}

\begin{lemma}\label{lem:Fx2_estimate}
Given any $f\in\mathbb{M}_0$, if there is some constant $d>0$ such that $f(x)\geq d(1-|x|)$ on $[-1,1]$, then for all $x_1\in[-1,1], x_2>0$, 
\[-F_{x_2}(x_1,x_2;f)\geq \frac{C_d}{\max\{1,x_2^3\}}\, ,\]
where $C_d>0$ is a constant that only depends on $d$.
\end{lemma}

\begin{proof}
Recall in the proof of Lemma \ref{lem:F_partial} we have showed that
\begin{equation}\label{eqt:Fx2_estimate_step}
\begin{split}
- F_{x_2}(x_1,x_2;f) &= \frac{1}{2\pi}\int_{-1}^1 \frac{1}{x_2}\cdot \eta\left(\frac{2x_2f(y_1)}{(x_1-y_1)^2+x_2^2+f(y_1)^2}\right)\idiff y_1 \\
&\quad + \frac{1}{2\pi}\int_{-1}^1\idiff y_1\int_0^{f(y_1)}\frac{1}{x_2y_2}\cdot \eta\left(\frac{2x_2y_2}{(x_1-y_1)^2+x_2^2+y_2^2}\right)\idiff y_2,
\end{split}
\end{equation}
where 
\[\eta(a) = \frac{1}{2a}\ln\frac{1+a}{1-a} - 1,\quad a\in[0,1).\]
Since $\eta(a)\geq a^2/3$, we get 
\begin{equation}\label{eqt:Fx2_estimate_step2}
\begin{split}
- F_{x_2}(x_1,x_2;f) &\geq C\int_{-1}^1\idiff y_1\int_0^{f(y_1)}\frac{x_2y_2}{\big((x_1-y_1)^2+x_2^2+y_2^2\big)^2}\idiff y_2\\
&\geq C\int_{\Omega_f\cap R_{x_1,x_2}}\frac{y_2}{x_2^3}\idiff \vy,
\end{split}
\end{equation}
where $\Omega_f = \{(y_1,y_2): y_1\in[-1,1],\ y_2\in[0,f(y_1)]\}$ and $R_{x_1,x_2}:= [x_1-x_2,x_1+x_2]\times[0,x_2]$. Provided that $f(x)\geq d(1-|x|)$ on $[-1,1]$, we further obtain
\[- F_{x_2}(x_1,x_2;f)\geq C\int_{\Omega_d\cap R_{x_1,x_2}}\frac{y_2}{x_2^3}\idiff \vy,\]
where $\Omega_d = \{(y_1,y_2): y_1\in[-1,1],\ y_2\in[0,d(1-|y_1|)]\}$. In order to prove the lemma, it then suffices to show that
\begin{equation}\label{eqt:Fx2_estimate_step3}
\int_{\Omega_d\cap R_{x_1,x_2}}\frac{y_2}{x_2^3}\idiff \vy \geq \frac{C_d}{\max\{1,x_2^3\}}
\end{equation}
for some constant $C_d>0$ that only depends on $d$. Without loss of generality, we may assume $d\leq 1$. Observe that, for any $x_1\in[-1,1]$, $x_2>0$, it holds 
\[\int_{\Omega_d\cap R_{x_1,x_2}}\frac{y_2}{x_2^3}\idiff \vy\geq \int_{\Omega_d\cap R_{1,x_2}}\frac{y_2}{x_2^3}\idiff \vy.\]
Also note that, since $1-x_2/d\leq 1-x_2$, we always have
\[\{(y_1,y_2): \max\{1-x_2,0\}\leq y_1\leq 1,\ 0\leq y_2\leq d(1-y_1)\} \subset \Omega_d\cap R_{1,x_2}.\]
Therefore, 
\[\int_{\Omega_d\cap R_{1,x_2}}\frac{y_2}{x_2^3}\idiff \vy\geq \int_{\max\{1-x_2,0\}}^1\idiff y_1\int_0^{d(1-y_1)}\frac{y_2}{x_2^3} \idiff y_2 = \frac{d^2}{6\max\{1,x_2^3\}}.\]
This implies that \eqref{eqt:Fx2_estimate_step3} holds with $C_d= d^2/6>0$, and thus the lemma is proved. 
\end{proof}

\begin{lemma}\label{lem:Fx2_estimate_2}
There is some uniform constant $C>0$ such that, for any $f\in \mathbb{M}_0$ and any $x_1\in \R$, $x_2> 0$, 
\[-F_{x_2}(x_1,x_2;f)\leq C.\]
\end{lemma}

\begin{proof}
Recall \eqref{eqt:Fx2_estimate_step} with 
\[\eta(a) = \frac{1}{2a}\ln\frac{1+a}{1-a} - 1,\quad a\in[0,1).\]
It is not difficult to see that $\eta(a)\leq C a|\ln(1-a)|$ for some absolute constant $C>0$. As a result, 
\begin{align*}
- F_{x_2}(x_1,x_2;f) &\leq  C\int_{-1}^1 \frac{f(y_1)}{(x_1-y_1)^2+x_2^2+f(y_1)^2} \ln\frac{(x_1-y_1)^2+x_2^2+f(y_1)^2}{(x_1-y_1)^2+(x_2-f(y_1))^2}\idiff y_1 \\
&\quad + C\int_{-1}^1\idiff y_1\int_0^{f(y_1)}\frac{1}{(x_1-y_1)^2+x_2^2+y_2^2} \ln\frac{(x_1-y_1)^2+x_2^2+y_2^2}{(x_1-y_1)^2+(x_2-y_2)^2}\idiff y_2\\
&=: C(I_1 + I_2).
\end{align*}
For $I_1$, we compute that
\begin{align*}
I_1 &= \int_{-1}^1 \frac{f(y_1)}{(x_1-y_1)^2+x_2^2+f(y_1)^2} \ln\left(1 + \frac{2x_2(f(y_1)-x_2) + 2x_2^2}{(x_1-y_1)^2+(x_2-f(y_1))^2}\right)\idiff y_1\\
&\leq \int_{-1}^1 \frac{1}{\big((x_1-y_1)^2+x_2^2\big)^{1/2}} \ln\left(1 + \frac{x_2}{|x_1-y_1|} + \frac{2x_2^2}{(x_1-y_1)^2}\right)\idiff y_1\\
&\leq \int_{\R} \frac{1}{\big(1+s^2\big)^{1/2}} \ln\left(1 + \frac{1}{|s|} + \frac{2}{s^2}\right)\idiff s\\
&\leq C.
\end{align*}
As for $I_2$, we similarly compute that
\begin{align*}
I_2 &=  \int_{-1}^1\idiff y_1\int_0^{f(y_1)}\frac{1}{(x_1-y_1)^2+x_2^2+y_2^2} \ln\left(1 + \frac{2x_2y_2}{(x_1-y_1)^2+(x_2-y_2)^2}\right)\idiff y_2\\
&\leq \int_{\R\times \R_+} \frac{1}{1 + s_1^2 + s_2^2} \ln\left(1 + \frac{2s_2}{s_1^2+(1-s_2)^2}\right) \idiff s_1 \idiff s_2\\
&\leq C.
\end{align*}
The claimed uniform upper bound for $-F_{x_2}(x_1,x_2;f)$ thus follows.
\end{proof}

\subsection{Bounds of $\mathbf{R}$} Based on the preceding estimates of $F(\vx;f)$, we can establish some quantitative bounds for $\mtx{R}(f)$ under extra assumptions on $f\in \mathbb{M}_0$. These estimates of $\mtx{R}(f)$ will be our guidance for constructing a suitable function set $\mathbb{D}$ such that the dynamical system \eqref{eqt:dynamic_preview} is well constrained for any initial state from $\mathbb{D}$.

\begin{lemma}\label{lem:Constant_upper_bound}
There is some uniform constant $M > 1$ such that, for any $f\in \mathbb{M}_0$, $f(0)\leq M$ implies $\mtx{R}(f)(0)< M$. (For example, we can choose  $M=3\times 2^{8/\pi}$.)
\end{lemma}

\begin{proof}
We may only prove the lemma for $f\in \mathbb{M}_1$ since $\mathbb{M}_1$ is dense in $\mathbb{M}_0$, as we will show later that $\mtx{R}(f)$ is continuous in $f$ with respect to the $L^\infty$ topology (see Proposition \ref{prop:R_continuity} below). 

We first show that there exists some constant $M_1>1$ such that, if $f(0)\geq M_1$, then $\mtx{R}(f)(0)<f(0)$. Recall that $\mtx{R}(f)$ is determined via $F(x,\mtx{R}(f)(x);f)=0$ for $x\in[-1,1]$. Thanks to Lemma \ref{lem:F_partial}, it suffices to show that $f(0)\geq M_1$ implies $F(0, f(0);f)<0$, i.e.
\begin{equation}\label{eqt:M_step}
 \frac{\phi(0,f(0);f)}{f(0)} < c(f).
\end{equation}
Write $\phi = \phi(\cdot,\cdot;f)$. For the left-hand side of \eqref{eqt:M_step}, we use \eqref{eqt:phi_inverse} to compute that
\begin{align*}
\frac{\phi(0,f(0))}{f(0)} &=\frac{1}{2\pi f(0)} \int_0^{f(0)}\idiff y_2\int_0^{f^{-1}(y_2)}\ln\frac{y_1^2 + (f(0)+y_2)^2}{y_1^2 + (f(0)-y_2)^2}\idiff y_1\\
&\leq \frac{1}{2\pi f(0)} \int_0^{f(0)}\idiff y_2\int_0^{f^{-1}(y_2)}\ln\frac{(f(0)+y_2)^2}{(f(0)-y_2)^2}\idiff y_1\\
&= \frac{1}{\pi} \int_0^{f(0)}\frac{f^{-1}(y_2)}{ f(0)}\ln\frac{f(0)+y_2}{f(0)-y_2}\idiff y_2\\
&= \frac{1}{\pi} \int_0^{1}f^{-1}(sf(0))\ln\frac{1+s}{1-s}\idiff s.
\end{align*}
For the right-hand side of \eqref{eqt:M_step}, we use \eqref{eqt:cf_inverse} to obtain
\begin{align*}
c(f) &= \frac{1}{\pi}\int_0^{f(0)}\left\{\arctan\left(\frac{f^{-1}(y_2)-1}{y_2}\right) + \arctan\left(\frac{f^{-1}(y_2)+1}{y_2}\right)\right\}\idiff y_2\\
&= \frac{1}{\pi}\int_0^{f(0)} \arctan\left(\frac{2y_2f^{-1}(y_2)}{1-f^{-1}(y_2)^2 + y_2^2}\right)\idiff y_2\\
&\geq \frac{1}{\pi}\int_0^{f(0)} \arctan\left(\frac{2y_2f^{-1}(y_2)}{1 + y_2^2}\right)\idiff y_2.
\end{align*}
We have used the fact that $f^{-1}(y_2)\in [0,1]$, so the two terms in the first line above have opposite signs and thus can be combined into a single term in the second line. Recall that we assume $f(0)\geq M_1\geq 1$. We then bound $c(f)$ from below as
\begin{align*}
c(f) &\geq \frac{1}{\pi}\int_0^1 \arctan\left(y_2f^{-1}(y_2)\right)\idiff y_2 + \frac{1}{\pi}\int_1^{f(0)} \arctan\left(\frac{f^{-1}(y_2)}{y_2}\right)\idiff y_2\\
&\geq \frac{1}{4}\int_0^1 y_2f^{-1}(y_2)\idiff y_2 + \frac{1}{4}\int_1^{f(0)} \frac{f^{-1}(y_2)}{y_2}\idiff y_2\\
&= \frac{1}{4}\left(f(0)^2\int_0^{1/f(0)}sf^{-1}(sf(0))\idiff s + \int_{1/f(0)}^1\frac{f^{-1}(sf(0))}{s}\idiff s\right)\\
&=  \frac{f(0)}{4}\int_0^1f^{-1}(sf(0))\cdot \min\left\{sf(0)\,,\,\frac{1}{sf(0)}\right\}\idiff s.  
\end{align*}
We have used the fact that $\arctan a\geq (\pi/4)a$ for $a\in[0,1]$.

Denote $g(s) := f^{-1}(sf(0))$, which is positive for $s\in[0,1)$ and strictly decreasing in $s\in[0,1]$. Combining the estimates above, in order to prove \eqref{eqt:M_step}, it suffices to show that
\begin{equation}\label{eqt:M_step2}
\int_0^{1}g(s)\ln\frac{1+s}{1-s}\idiff s < \frac{\pi}{4}\int_0^1g(s)\cdot f(0)\min\left\{sf(0)\,,\,\frac{1}{sf(0)}\right\}\idiff s,
\end{equation}
as long as $f(0)\geq M_1$ for some sufficiently large $M_1$. In fact, we can choose $M_1=2^{8/\pi}$. We observe that
\[\int_0^1\ln\frac{1+s}{1-s}\idiff s =\ln 4,\quad \int_0^1 f(0)\min\left\{sf(0)\,,\,\frac{1}{sf(0)}\right\}\idiff s = \frac{1}{2} + \ln f(0).\]
Therefore, if $f(0)\geq M_1$, then
\[\int_0^1\ln\frac{1+s}{1-s}\idiff s  = \ln 4 < \frac{\pi}{4}\left(\frac{1}{2} + \ln M_1\right)\leq \frac{\pi}{4}\left(\frac{1}{2} + \ln f(0)\right) = \frac{\pi}{4}\int_0^1 f(0)\min\left\{sf(0)\,,\,\frac{1}{sf(0)}\right\}\idiff s.\]
That is,
\[\int_0^1 m(s)\idiff s < 0,\]
where 
\[m(s) := \ln\frac{1+s}{1-s} - \frac{\pi}{4}f(0)\min\left\{sf(0)\,,\,\frac{1}{sf(0)}\right\}.\]
Note that $m(0)=0$, $\lim_{s\rightarrow 1}m(s) = +\infty$, and
\[m'(s)=\left\{\begin{array}{ll}
\displaystyle \frac{2}{1-s^2} - \frac{\pi}{4}f(0)^2 \leq f(0)^2\left(\frac{2}{f(0)^2-1} - \frac{\pi}{4}\right)< 0,& s\in[0,1/f(0)),\\
\displaystyle \frac{2}{1-s^2} +  \frac{\pi}{4s^2} > 0, & s\in(1/f(0),1).
\end{array}\right.\]
Here, we have used the assumption that $f(0)^2\geq M_1^2  = 2^{16/\pi}> 1 + 8/\pi$. Hence, there is some $\xi\in(0,1)$ such that $m(\xi)=0$, $m(s)< 0$ for $s\in(0,\xi)$, and $m(s)> 0$ for $s\in(\xi,1]$. Since $g(s)$ is positive for $s\in[0,1)$ and strictly decreasing in $s\in[0,1]$, it then follows that
\[\int_0^1 g(s)m(s)\idiff s = \int_0^1 (g(s)-g(\xi))m(s)\idiff s + g(\xi)\int_0^1 m(s)\idiff s < 0,\]
which is \eqref{eqt:M_step2}.

Next, we show that there is some sufficiently large constant $M\geq M_1$ such that $f(0)\leq M_1$ implies $\mtx{R}(f(0))< M$. Again, it suffices to show that if $f(0)\leq M_1$, then
\begin{equation}\label{eqt:M_step3}
 \frac{\phi(0,M)}{M} < c(f).
\end{equation}
We can similarly derive that 
\begin{align*}
\frac{\phi(0,M)}{M} &=\frac{1}{2\pi M} \int_0^{f(0)}\idiff y_2\int_0^{f^{-1}(y_2)}\ln\frac{y_1^2 + (M+y_2)^2}{y_1^2 + (M-y_2)^2}\idiff y_1\\
&\leq \frac{1}{\pi} \int_0^{f(0)}\frac{f^{-1}(y_2)}{M}\ln\frac{M+y_2}{M-y_2}\idiff y_2\\
&= \frac{1}{\pi} \int_0^{1}\frac{f^{-1}(sf(0))}{M/f(0)}\ln\frac{M/f(0)+s}{M/f(0)-s}\idiff s\\
&\leq \frac{2}{\pi} \int_0^{1}\frac{f^{-1}(sf(0))f(0)^2s}{M(M-sf(0))}\idiff s\\
&\leq \frac{2f(0)^2}{\pi M(M-M_1)} \int_0^{1}sf^{-1}(sf(0))\idiff s,
\end{align*}
and 
\begin{align*}
c(f) &\geq \frac{1}{\pi}\int_0^{f(0)} \arctan\left(\frac{2y_2f^{-1}(y_2)}{1 + y_2^2}\right)\idiff y_2\\
&\geq \frac{1}{\pi}\int_0^{f(0)} \arctan\left(\frac{2y_2f^{-1}(y_2)}{1 + f(0)^2}\right)\idiff y_2\\
&\geq \frac{1}{4}\int_0^{f(0)} \frac{2y_2f^{-1}(y_2)}{1 + f(0)^2}\idiff y_2 \geq \frac{f(0)^2}{2(1+M_1^2)}\int_0^1sf^{-1}(sf(0))\idiff s.
\end{align*}
We have used above the assumption that $f(0)\leq M_1\leq M$. In order to achieve \eqref{eqt:M_step3}, it suffices to choose $M$ sufficiently large so that
\[\frac{2}{\pi}\frac{1}{M(M-M_1)} <  \frac{1}{2(1+M_1^2)}.\]
This can certainly be achieved if we choose, for example, $M = 3M_1$.

In summary, we have shown that, with $M_1=2^{8/\pi}$ and $M=3M_1 = 3\times2^{8/\pi}$, (1) if $f(0)\geq M_1$, then $\mtx{R}(f)(0)< f(0) $, and (2) if $f(0)\leq M_1$, then $\mtx{R}(f)(0)< M$. Combining these two results, we can conclude that if $f(0)\leq M$, then $\mtx{R}(f)(0)< M$. The lemma is thus proved.
\end{proof}

From now on, we will always denote by $M$ the $\mtx{R}$-preserving constant upper bound determined in Lemma \ref{lem:Constant_upper_bound}. The next lemma provides a non-constant upper barrier that constrains the solution of the dynamical system \eqref{eqt:dynamic_preview} from above, especially for $x$ near $\pm1$ where this upper barrier vanishes like $(1-|x|)\log(1+1/(1-|x|))$.

\begin{lemma}\label{lem:upper_barrier}
Let $v_\Lambda$ be defined as 
\begin{equation}\label{eqt:v_definition}
v_\Lambda(x) := M\cdot W^{-1}\left(\frac{\Lambda(1-|x|)}{M}\right),\quad \text{where}\ W(t):= \int_0^{t}\frac{1}{1+\ln(1+1/s)}\idiff s.
\end{equation}
There is some sufficiently large $\Lambda>0$ such that, if $f\in \mathbb{M}_0$ satisfies $f(x)\leq \min\{v_\Lambda(x)\,,M\}$ on $[-1,1]$ and $f(x_*)=\min\{v_\Lambda(x_*)\,,M\}$ at some $x_*\in (-1,1)$, then $\mtx{R}(f)(x_*)< v_\Lambda(x_*)$.
\end{lemma}

\begin{proof}
There is nothing to prove if $f(x_*)=M\leq v_\Lambda(x_*)$, since Lemma \ref{lem:Constant_upper_bound} guarantees that $\mtx{R}(f)(x_*)\leq \mtx{R}(f)(0)<M$. Hence, we can assume $x_*\in[0,1)$ is such that $f(x_*)=v_\Lambda(x_*)<M$. Observe that $W(t)$ is increasing in $t$ and $W(t)\leq t$ for $t\geq 0$, which means $W^{-1}(s)\geq s$ for $s\geq 0$, and thus
\[
v_\Lambda(x) = M\cdot W^{-1}\left(\frac{\Lambda(1-|x|)}{M}\right)\geq \Lambda(1-|x|),\quad x\in[-1,1].
\]
We thus can make $\Lambda$ large enough, e.g. $\Lambda\geq 5M$, so that for $x\in[0,1)$, $v_\Lambda(x)\geq 4M(1-x)$. Note that $v_\Lambda(x)$ is decreasing in $x\in[0,1]$. This means the point $x_*\in[0,1)$ where $f(x_*)=v_\Lambda(x_*)<M$ can only fall in $(3/4,1)$.

We prove by contradiction. Write $\tilde f=\mtx{R}(f)$. Suppose $\tilde f(x_*)\geq f(x_*)\geq 4M(1-x_*)$. Recall $M>1$. Then by the continuity and monotonicity of $\tilde f$, there exists some unique $a,b\in[x_*,1)$ such that $a<b$, $\tilde f(a) = f(x_*)$, and $\tilde f(b) = 2(1-x_*)$. Note that for any $x\in[a,b]\subset[x_*,1)$, $x-x_*\leq 1-x_*$, and $2(1-x_*)\leq \tilde f(x) \leq f(x_*)$. Since $[0,x_*]\times [0,f(x_*)]\subset\Omega _f:=\{(x_1,x_2): x_1\in[-1,1], x_2\in[0,f(x_1)]\}$, for any $x\in [a,b]$, 
\begin{align*}
& \Omega_f\cap \big([x-\tilde f(x),x+\tilde f(x)]\times [0,\tilde f(x)]\big) \\
&\supset \big([0,x_*]\times [0,f(x_*)]\big)\cap \big([x-\tilde f(x),x+\tilde f(x)]\times [0,\tilde f(x)]\big)\\
&\supset [\max\{0,x-\tilde f(x)\},x_*]\times [0,\tilde f(x)].
\end{align*}
Let us use this to bound $-\tilde f'(x)$ from above for $x\in[a,b]$. Recall \eqref{eqt:Fx2_estimate_step2} in the proof of Lemma \ref{lem:Fx2_estimate}, which now implies that
\begin{align*}
- F_{x_2}(x,\tilde f(x);f) &\geq C\int_{\Omega_f\cap ([x-\tilde f(x),x+\tilde f(x)]\times [0,\tilde f(x)])}\frac{y_2}{\tilde f(x)^3}\idiff \vy\\
&\geq C\int_{\max\{0,x-\tilde f(x)\}}^{x_*}\idiff y_1 \int_0^{\tilde f(x)}\frac{y_2}{\tilde f(x)^3}\idiff y_2\\
&\geq C\cdot \frac{\min\{x_*\,,\,\tilde f(x) +x_*-x\}}{\tilde f(x)}\\
&\geq C\cdot \frac{\min\{x_*\,,\,\tilde f(x)/2\}}{\tilde f(x)}\\
&\geq C\cdot \min\left\{\frac{2x_*}{M}\,,\,1\right\}\geq C,
\end{align*}
for some absolute constant $C>0$ (that only depends on $M$). We have used above the fact that $x_*\geq 3/4$ and $2(x-x_*)\leq \tilde f(x)\leq M$. Besides, we can use the first upper bound in Lemma \ref{lem:Fx1_estimate} to obtain 
\[ - F_{x_1}(x,\tilde f(x);f) \leq C\left(1+\ln\left(1+\frac{f(0)}{\tilde f(x)}\right)\right)\leq C\left(1+\ln\left(1+\frac{M}{\tilde f(x)}\right)\right),\]
for some absolute constant $C>0$. Combining these estimates, we obtain 
\[-\tilde f'(x) = \frac{F_{x_1}(x,\tilde f(x);f)}{F_{x_2}(x,\tilde f(x);f)} \leq C \left(1+\ln\left(1+\frac{M}{\tilde f(x)}\right)\right).\]
Then, by the definition of $W$, 
\[-\frac{\diff\, }{\diff x}W\left(\frac{\tilde f(x)}{M}\right)= W'\left(\frac{\tilde f(x)}{M}\right)\cdot \frac{-\tilde f'(x)}{M} = \frac{1}{1+\ln(1+M/\tilde f(x))}\cdot \frac{-\tilde f'(x)}{M}\leq \frac{C}{M}, \]
and thus
\[W\left(\frac{f(x_*)}{M}\right) - W\left(\frac{2(1-x_*)}{M}\right) = W\left(\frac{\tilde f(a)}{M}\right) - W\left(\frac{\tilde f(b)}{M}\right)\leq \frac{C(b-a)}{M}\leq \frac{C(1-x_*)}{M}.\]
Since $f(x_*)=v_\Lambda(x_*)$, by the definition of $v_\Lambda$, the above can be simplified to 
\[\frac{\Lambda(1-x_*)}{M} - W\left(\frac{2(1-x_*)}{M}\right)\leq \frac{C(1-x_*)}{M},\]
for some constant $C$ that does not depend on $\Lambda$ or $x_*$. Recall $W(t)\leq t$ for $t\geq0$. It follows that
\[\frac{\Lambda(1-x_*)}{M}\leq W\left(\frac{2(1-x_*)}{M}\right) + \frac{C(1-x_*)}{M}\leq (2+C)\frac{(1-x_*)}{M}.\]
However, this cannot happen if we choose $\Lambda$ sufficiently large. The lemma is thus proved.
\end{proof}

Next, we establish a lower barrier property of $\mtx{R}$ in a similar fashion. To this end, we will need to bound $F(x,f(x);f)$ from below as in the following lemma. 

\begin{lemma}\label{lem:F_lower_bound}
Given any $f\in \mathbb{M}_1$, if $f(0)\leq M$, then for any $x\in[0,1)$, 
\begin{align*}
F(x,f(x);f) &\geq \tilde C\int_0^{f(x)}\left(1-\frac{y}{f(x)}\right)\arctan\left(\frac{1-f^{-1}(y)}{y}\right)\idiff y + \tilde C\int_{f(x)}^{f(0)}\frac{(1-x)f^{-1}(y)y}{(1-f^{-1}(y))^2+y^2}\idiff y\\
&\quad - \bar C \frac{f(x)^2}{\max\{1-x\,,f(x)\}} - \bar C \int_0^{f(x)}\frac{f(x)}{x+f(x)+f^{-1}(y)}\idiff y,
\end{align*}
where $\tilde C, \bar C>0$ are some absolute constants.
\end{lemma}

We delay the lengthy proof of Lemma \ref{lem:F_lower_bound} to Appendix \ref{apx:proofs}. We now use this lemma to derive the following lower barrier property for $\mtx{R}$.

\begin{lemma}\label{lem:lower_barrier}
Let $v_\Lambda$ be defined as in \eqref{eqt:v_definition} with $\Lambda$ determined in Lemma \ref{lem:upper_barrier}, and let $u_\lambda$ be defined as
\begin{equation}\label{eqt:u_definition}
u_\lambda(x) := \lambda(1-x^2).
\end{equation}
There is some sufficiently small $\lambda>0$ such that, if $f\in \mathbb{M}_0$ satisfies $u_\lambda(x)\leq f(x)\leq \min\{v_\Lambda(x)\,,M\}$ on $[-1,1]$ and $f(x_*)=u_\lambda(x_*)$ at some $x_*\in (-1,1)$, then $\mtx{R}(f)(x_*) > u_\lambda(x_*)$.
\end{lemma}

\begin{proof}
Without loss of generality, we may assume that $\lambda\leq 1/2$ and $x_*\in[0,1)$. We only need to prove the Lemma for $f\in \mathbb{M}_1$. 

In view of Lemma \ref{lem:F_partial}, it suffices to show that $F(x_*,f(x_*);f)> 0$. By Lemma \ref{lem:F_lower_bound}, this further boils down to verifying, for sufficiently small $\lambda$,  
\begin{equation}\label{eqt:lower_barrier_step}
\begin{split}
&\tilde C\int_0^{f(x_*)}\left(1-\frac{y}{f(x_*)}\right)\arctan\left(\frac{1-f^{-1}(y)}{y}\right)\idiff y + \tilde C\int_{f(x_*)}^{f(0)}\frac{(1-x_*)f^{-1}(y)y}{(1-f^{-1}(y))^2+y^2}\idiff y\\
&>  \bar C \frac{f(x_*)^2}{\max\{1-x_*\,,f(x_*)\}} + \bar C \int_0^{f(x_*)}\frac{f(x_*)}{x_*+f(x_*)+f^{-1}(y)}\idiff y,
\end{split}
\end{equation}
where $\tilde C,\bar C>0$ are the absolute constants given in Lemma \ref{lem:F_lower_bound}. 

We first handle the right-hand side above. By the assumption that $f(x)\geq u_\lambda(x)$ for $x\in[0,1]$, we have
\[f^{-1}(y) \geq u_\lambda^{-1}(y) = \left(1-\frac{y}{\lambda}\right)^{1/2}.\]
Hence, with $f(x_*) = u_\lambda(x_*) = \lambda(1-x_*^2)$, we obtain
\begin{align*}
\int_0^{f(x_*)}\frac{f(x_*)}{x_*+f(x_*)+f^{-1}(y)}\idiff y &\leq f(x_*)\int_0^{f(x_*)}\frac{1}{\left(1-\frac{y}{\lambda}\right)^{1/2}}\idiff y\\
&= \lambda f(x_*) \int_0^{f(x_*)/\lambda }(1-s)^{-1/2}\idiff s \\
&=  2\lambda f(x_*)(1-x_*) \leq C\lambda^2(1-x_*)^2.
\end{align*}
Besides, $f(x_*) = \lambda(1-x_*^2)$ implies $\max\{1-x_*\,,f(x_*)\} = 1-x_*$. Thus, the right-hand side of \eqref{eqt:lower_barrier_step} can be bounded from above as
\begin{align*}
&\bar C \frac{f(x_*)^2}{\max\{1-x_*\,,f(x_*)\}} + \bar C \int_0^{f(x_*)}\frac{f(x_*)}{x_*+f(x_*)+f^{-1}(y)}\idiff y\\
&\leq \frac{Cf(x_*)^2}{1-x_*} + C\lambda^2(1-x_*)^2\leq  C_1\lambda^2(1-x_*).
\end{align*}

Next, we control the left-hand side of \eqref{eqt:lower_barrier_step} from below. By the assumption $f(x)\leq v_\Lambda(x)$, we have
\[1-f^{-1}(y)\geq 1- v_\Lambda^{-1}(y) =\frac{M}{\Lambda}W\left(\frac{y}{M}\right). \]
Since $\Lambda\geq 1$ and 
\[t\geq W(t)=\int_0^t\frac{1}{1+\ln(1+1/s)}\idiff s\geq \frac{t/2}{1+\ln(1+2/t)},\]
we obtain 
\begin{align*}
\tilde C\int_0^{f(x_*)}\left(1-\frac{y}{f(x_*)}\right)\arctan\left(\frac{1-f^{-1}(y)}{y}\right)\idiff y&\geq C\int_0^{f(x_*)}\left(1-\frac{y}{f(x_*)}\right)\arctan\left(\frac{W(y/M)}{\Lambda y/M}\right)\idiff y \\
&\geq C\int_{f(x_*)/3}^{2f(x_*)/3}\frac{W(y/M)}{\Lambda y/M}\idiff y \\
&\geq  \frac{C}{\Lambda}\int_{f(x_*)/3}^{2f(x_*)/3}\frac{1}{1+\ln(1+2M/y)}\idiff y \\
&\geq \frac{C}{\Lambda}\cdot \frac{f(x_*)}{1+\ln(1+6M/f(x_*))}\\
&\geq C_2 \cdot \frac{\lambda(1-x_*)}{\Lambda\ln\left(1+\frac{M}{\lambda(1-x_*)}\right)}.
\end{align*}
Furthermore, under the assumption $f(x)\geq u_\lambda(x)=\lambda (1-x^2)$,
\[1-f^{-1}(y)\leq 1-u_\lambda^{-1}(y)\leq 1-\left(1-\frac{y}{\lambda}\right)^{1/2}\leq \frac{y}{\lambda}.\]  
It follows that
\begin{align*}
\tilde C \int_{f(x_*)}^{f(0)}\frac{(1-x_*)f^{-1}(y)y}{(1-f^{-1}(y))^2+y^2}\idiff y &\geq C (1-x_*)\int_{\lambda(1-x_*^2)}^{\lambda(1-x_*^2/2)}\frac{(1-y/\lambda)y}{y^2/\lambda^2+y^2}\idiff y\\
&= C \frac{\lambda^2}{1+\lambda^2}(1-x_*)\int_{1-x_*^2}^{1-x_*^2/2}\frac{1-s}{s}\idiff s\\
&\geq C \frac{\lambda^2}{1+\lambda^2}(1-x_*)\cdot \frac{x_*^2}{2} \int_{1-x_*^2}^{1-x_*^2/2}\frac{1}{s}\idiff s\\
&= C \frac{\lambda^2}{1+\lambda^2}(1-x_*)\cdot \frac{x_*^2}{2} \ln\left(1+\frac{x_*^2/2}{1-x_*^2}\right)\\
&\geq C_2 \lambda^2(1-x_*) \cdot x_*^2\ln\left(1+\frac{x_*^2}{1-x_*}\right).
\end{align*}

With all the estimates above, in order to prove \eqref{eqt:lower_barrier_step}, it now suffices to show that, when $\lambda$ is sufficiently small,
\[I_\lambda:= \frac{1}{\Lambda\lambda \ln\left(1+\frac{M}{\lambda(1-x_*)}\right)} + x_*^2\ln\left(1+\frac{x_*^2}{1-x_*}\right)- C_* > 0, \]
where $C_*=C_1/C_2$ is some absolute constant. In fact, if $x_*\geq \delta_*:=\max\{1/2\,,1-\econst^{-4C_*}/4\}$, then
\[I_\lambda \geq x_*^2\ln\left(1+\frac{x_*^2}{1-x_*}\right)- C_* \geq \frac{1}{4}\ln\left(1+\econst^{4C_*}\right)- C_* >0.\]
Otherwise, if $x_*\leq \delta_*<1$, then
\[I_\lambda\geq \frac{1}{\Lambda\lambda \ln\left(1+\frac{M}{\lambda(1-x_*)}\right)} - C_*\geq \frac{1}{\Lambda\lambda \ln\left(1+\frac{M}{\lambda(1-\delta_*)}\right)} - C_* >0,\]
provided that $\lambda$ is sufficiently small, with the smallness only depending on $M$, $\Lambda$, and $C_*$. The Lemma is thus proved.
\end{proof}

In what follows, $v_{\Lambda}$ and $u_\lambda$ will always be defined as in \eqref{eqt:v_definition} and \eqref{eqt:u_definition} with the constants $\Lambda$ and $\lambda$ determined in Lemma \ref{lem:upper_barrier} and \ref{lem:lower_barrier}, respectively.

\subsection{$C^{1/2}$-estimate}
We provide a simple $C^{1/2}$-estimate for $\mtx{R}(f)$ provided that $f\in \mathbb{M}_0$ is lower bounded by some appropriate profile. In fact, we can establish a uniform $C^\alpha$-bound for $\mtx{R}(f)$ in a similar way for any $\alpha\in(0,1)$, but we only need, for example, a uniform $C^{1/2}$-estimate to conclude the compactness of $\mtx{R}$ in the $L^\infty$-topology.

\begin{proposition}\label{prop:C_alpha}
Given any $f\in\mathbb{M}_0$, if $d(1-|x|)\leq f(x)\leq M$ on $[-1,1]$ for some constant $d>0$, then $|\mtx{R}(f)|_{C^{1/2}}\leq C_{M,d}$ for some constant $C_{M,d}$ that only depends on $M$ and $d$.
\end{proposition}  

\begin{proof}
Let $f\in \mathbb{M}_0$ be such that $d(1-|x|)\leq f(x)\leq M$ on $[-1,1]$, and denote $\tilde f=\mtx{R}(f)$. By Lemma \ref{lem:Fx1_estimate} and Lemma \ref{lem:Fx2_estimate}, for all $x_1\in[0,1], x_2\in[0,M]$, we have
\[|F_{x_1}(x_1,x_2;f)| \leq \frac{x_1}{x_2} \leq \frac{1}{x_2},\]
and
\[|F_{x_2}(x_1,x_2;f)|\geq \frac{C_d}{\max\{1,x_2^3\}}\geq \frac{C_d}{M^3}.\]
By Lemma \ref{lem:Constant_upper_bound}, $f(0)\leq M$ implies $\tilde f(x)\leq M$ on $[-1,1]$. Also recall that $\tilde f'(x)\leq 0$ on $[0,1]$. Hence, for all $x\in[0,1]$, 
\[-\tilde {f}'(x) =|\tilde {f}'(x)| = \frac{|F_{x_1}(x,\tilde f(x);f)|}{|F_{x_2}(x,\tilde f(x);f)|}\leq \frac{C_{M,d}}{\tilde f(x)},\]
for some uniform constant $C_{M,d}>0$ that only depends on $M$ and $d$. That is, 
\[-\left(\tilde f(x)^2\right)' = -2\tilde f(x) \tilde f(x)'\leq C_{M,d}. \]
Then, for any $x,y\in[0,1]$, $x<y$, it holds that
\[\tilde f(x)^2 - \tilde f(y)^2 = \int_x^y \left(-\tilde f(s)^2\right)'\idiff s \leq C_{M,d} (y-x),\]
which implies 
\[\tilde f(x)-\tilde f(y) \leq \left(C_{M,d} (y-x) + \tilde f(y)^2\right)^{1/2} - \tilde f(y)\leq C_{M,d} (y-x)^{1/2}.\]
The claimed $C^{1/2}$-estimate of $\mtx{R}(f)$ thus follows.
\end{proof}

\subsection{Continuity of $\mathbf{R}$}
To employ the Schauder fixed-point theorem, we also need to show that the map $\mtx{R}$ is continuous in the $L^\infty$-topology. To achieve this, we first demonstrate how $F(x_1,x_2;f)$ continuously depends on $f\in \mathbb{M}_0$.

\begin{lemma}\label{lem:F_continuity}
For any $f_1,f_2\in\mathbb{M}_0$, and for any $x_1\in[-1,1], x_2\geq0$, 
\[|F(x_1,x_2;f_2) -F(x_1,x_2;f_1)|\leq  C \|f_1-f_2\|_{L^\infty}\ln\left(2+\frac{1}{\|f_1-f_2\|_{L^\infty}}\right),\]
for some absolute constant $C$.
\end{lemma}

\begin{proof}
Given any $f_1,f_2\in \mathbb{M}_0$, define $g_1(x) = \min\{f_1(x),f_2(x)\}$, $g_2(x) = \max\{f_1(x),f_2(x)\}$. We also denote $\|f_2-f_1\|_{L^\infty}=\delta$. 

Let us introduce a useful numeric fact that, for any $a,b\in \R$, $a<b$,  
\begin{equation}\label{eqt:F_continuity_step}
\int_a^b\frac{2|s|}{1+s^2}\idiff s\leq C\ln\big(1+(b-a)\big),
\end{equation}
for some absolute constant $C>0$. One can simply verify this in all three cases: $a<b\leq0$, $0\leq a<b$, or $a<0<b$. Now, we can compute that
\begin{align*}
|c(f_2) - c(f_1)| &=  \left|\frac{1}{2\pi}\int_{-1}^1 \ln \frac{(1-y)^2 + f_2(y)^2}{(1-y)^2+f_1(y)^2}\idiff y\right|\\
&\leq \frac{1}{2\pi}\int_{-1}^1 \ln \frac{(1-y)^2 + g_2(y)^2}{(1-y)^2+g_1(y)^2}\idiff y\\
&=  \frac{1}{2\pi}\int_{-1}^1 \idiff y \int_{g_1(y)/(1-y)}^{g_2(y)/(1-y)} \frac{2s}{1+s^2}\idiff s\\
&\leq C \int_{-1}^1 \ln\left(1+\frac{g_2(y)-g_1(y)}{1-y}\right)\idiff y\\
&\leq C \int_{-1}^1 \ln\left(1+\frac{\delta}{1-y}\right)\idiff y\\
&\leq C\left(\delta \ln\left(1+\frac{1}{\delta}\right) + \ln(1+\delta)\right)\leq C\delta\ln\left(2+\frac{1}{\delta}\right).
\end{align*}
We have used \eqref{eqt:F_continuity_step} for the second inequality above. Next, let $\phi_i := (\cdot\,,\cdot\,;f_i)$, $i=1,2$. We then compute that
\begin{align*}
&\left|\frac{\phi_2(x_1,x_2)}{x_2}-\frac{\phi_1(x_1,x_2)}{x_2}\right|\\
&\leq \frac{1}{2\pi }\int_{-1}^1\idiff y_1\int_{g_1(y_1)}^{g_2(y_1)}\frac{1}{2x_2}\ln\frac{(x_1-y_1)^2+(x_2+y_2)^2}{(x_1-y_1)^2+(x_2-y_2)^2}\idiff y_2\\
&= \frac{1}{2\pi }\int_{-1}^1\idiff y_1\int_{g_1(y_1)}^{g_2(y_1)}\frac{1}{2x_2}\ln\left(1+\frac{4x_2(y_2-x_2)}{(x_1-y_1)^2+(x_2-y_2)^2} + \frac{4x_2^2}{(x_1-y_1)^2+(x_2-y_2)^2}\right)\idiff y_2\\
&\leq \frac{1}{2\pi }\int_{-1}^1\idiff y_1\int_{g_1(y_1)}^{g_2(y_1)}\frac{1}{2x_2}\ln\left(1+\frac{4x_2|y_2-x_2|}{(x_1-y_1)^2+(x_2-y_2)^2}\right)\idiff y_2\\
&\quad + \frac{1}{2\pi }\int_{-1}^1\idiff y_1\int_{g_1(y_1)}^{g_2(y_1)}\frac{1}{2x_2}\ln\left(1+\frac{4x_2^2}{(x_1-y_1)^2+(x_2-y_2)^2}\right)\idiff y_2\\
&=: I_1 + I_2.
\end{align*}
We bound $I_1$ as
\begin{align*}
I_1 &\leq \frac{1}{2\pi }\int_{-1}^1\idiff y_1\int_{g_1(y_1)}^{g_2(y_1)}\frac{2|y_2-x_2|}{(x_1-y_1)^2+(x_2-y_2)^2}\idiff y_2\\
&= \frac{1}{2\pi }\int_{-1}^1\idiff y_1\int_{(g_1(y_1)-x_2)/|x_1-y_1|}^{(g_2(y_1)-x_2)/|x_1-y_1|}\frac{2|s|}{1+s^2}\idiff s\\
&\leq C \int_{-1}^1 \ln\left(1+\frac{g_2(y)-g_1(y)}{|x_1-y_1|}\right)\idiff y_1\\
&\leq C \int_{-1}^1 \ln\left(1+\frac{\delta}{|x_1-y_1|}\right)\idiff y_1 \leq C\delta \ln\left(2+\frac{1}{\delta}\right).
\end{align*}
We have again used \eqref{eqt:F_continuity_step} for the second inequality above. As for $I_2$, we find that
\begin{align*}
I_2 &\leq \frac{1}{2\pi }\int_{-1}^1\idiff y_1\int_{g_1(y_1)}^{g_2(y_1)}\frac{1}{2x_2}\ln\left(1+\frac{4x_2^2}{(x_1-y_1)^2}\right)\idiff y_2\\
&\leq \frac{\delta}{2\pi }\int_{-1}^1\frac{1}{2x_2}\ln\left(1+\frac{4x_2^2}{(x_1-y_1)^2}\right)\idiff y_1\\
&\leq \frac{\delta}{2\pi }\int_{-\infty}^\infty\ln\left(1+\frac{1}{s^2}\right)\idiff s \leq C\delta.
\end{align*}
Putting the preceding estimates together, we obtain
\begin{align*}
\left|F(x_1,x_2;f_2) - F(x_1,x_2;f_1)\right| &\leq \left|\frac{\phi_2(x_1,x_2)}{x_2}-\frac{\phi_1(x_1,x_2)}{x_2}\right| + |c(f_2) - c(f_1)| \\
&\leq  I_1 + I_2  + |c(f_2) - c(f_1)| \leq C\delta\ln\left(2+\frac{1}{\delta}\right),
\end{align*}
as desired.
\end{proof}

It is then not difficult to justify the continuity of $\mtx{R}$ by the continuity of $\mtx{F}(\vx;f)$ in $f$. 

\begin{proposition}\label{prop:R_continuity}
Given any $f_1,f_2\in\mathbb{M}_0$, if $d(1-|x|)\leq f_1(x),f_2(x)\leq M$ on $[-1,1]$ for some constant $d>0$, then
\[\|\mtx{R}(f_2)-\mtx{R}(f_1)\|_{L^\infty} \leq C_{M,d} \|f_2-f_1\|_{L^\infty}\ln\left(2+\frac{1}{\|f_2-f_1\|_{L^\infty}}\right),\]
for some constant $C_{M,d}>0$ that only depends on $M$ and $d$.
\end{proposition}

\begin{proof}
Denote $\tilde f_i = \mtx{R}(f_i)$, $i=1,2$. Recall that $\tilde f_i$ is defined via 
\[F(x,\tilde f_i(x);f_i) = 0,\quad x\in[-1,1].\]
By Lemma \ref{lem:Constant_upper_bound}, $f_i(x)\leq M$ on $[-1,1]$ implies $\tilde f_i(x)\leq M$ on $[-1,1]$.

Now fix an $x\in[-1,1]$. On the one hand, by Lemma \ref{lem:F_continuity}, we have
\[|F(x,\tilde f_1(x);f_2)| = |F(x,\tilde f_1(x);f_2) - F(x,\tilde f_1(x);f_1)|\leq \bar C \|f_1-f_2\|_{L^\infty} \ln(2+\|f_1-f_2\|_{L^\infty}^{-1})\]
for some uniform constant $\bar C>0$. On the other hand, by Lemma \ref{lem:Fx2_estimate}, there is some constant $\tilde C_{M,d}>0$ depending only on $M,d$ such that $|F_{x_2}(x_1,x_2;f_2)|\geq \tilde C_{M,d}$ for all $x_1\in[-1,1], x_2\in[0,M]$, which leads to
\[|F(x,\tilde f_1(x);f_2)|=|F(x,\tilde f_1(x);f_2) - F(x,\tilde f_2(x);f_2)| \geq \tilde C_{M,d} |\tilde f_2(x)-\tilde f_1(x)|.\]
Therefore, 
\[|\tilde f_2(x)-\tilde f_1(x)|\leq \frac{\bar{C}}{\tilde C_{M,d}} \|f_1-f_2\|_{L^\infty} \ln\left(2+\frac{1}{\|f_1-f_2\|_{L^\infty}}\right).\]
This proves the claim.
\end{proof}

\subsection{The convex, closed, and compact function set} Now, we are ready to define the smaller function set $\mathbb{D}\subset \mathbb{M}_0$, in which we will prove the existence of a fixed point $f=\mtx{R}(f)$ by studying a continuous-time dynamical system induced by $\mtx{R}$. Guided by the preceding results, we define
\begin{equation}\label{eqt:D_definition}
\mathbb{D} := \{f\in \mathbb{M}_0: u_\lambda(x)\leq f(x)\leq \min\{v_{\Lambda}(x),M\}\ \text{for}\ x\in[-1,1],\ |f|_{C^{1/2}}\leq \gamma\}.
\end{equation}
Here, the constant $M$ is determined in Lemma \ref{lem:Constant_upper_bound}; $v_{\Lambda}$ and $u_\lambda$ are defined in \eqref{eqt:v_definition} and \eqref{eqt:u_definition} with $\Lambda$ and $\lambda$ determined in Lemma \ref{lem:upper_barrier} and \ref{lem:lower_barrier}, respectively; the constant $\gamma=\gamma_{M,\lambda}$ is determined by $M$ and $\lambda$ the same as how $C_{M,d}$ is determined by $M$ and $d$ in Proposition \ref{prop:C_alpha}. Apparently, $\mathbb{D}$ is a convex, closed, and compact subset of $\mathbb{V}$ in the $L^\infty$-topology.

By the definition \eqref{eqt:D_definition} of $\mathbb{D}$, we immediately have the following corollaries.

\begin{corollary}\label{cor:C_alpha_D}
For any $f\in \mathbb{D}$, $|\mtx{R}(f)|_{C^{1/2}}\leq \gamma$. 
\end{corollary}

\begin{proof}
The claim follows from Proposition \ref{prop:C_alpha}, since $\lambda(1-|x|)\leq u_\lambda(x) \leq f(x)\leq M$ for all $f\in \mathbb{D}$.
\end{proof}

\begin{corollary}\label{cor:R_continuity_D}
The map $\mtx{R}$ is uniformly log-Lipschitz continuous on $\mathbb{D}$. More precisely, there is some uniform constant $C>0$ such that, for any $f_1,f_2\in\mathbb{D}$,
\[\|\mtx{R}(f_2)-\mtx{R}(f_1)\|_{L^\infty} \leq C \|f_2-f_1\|_{L^\infty}\ln\left(2+\frac{1}{\|f_2-f_1\|_{L^\infty}}\right).\]
\end{corollary}

\begin{proof}
The claim follows from Proposition \ref{prop:R_continuity}, since $\lambda(1-|x|)\leq u_\lambda(x) \leq f(x)\leq M$ for all $f\in \mathbb{D}$.
\end{proof}

\subsection{An induced dynamical equation}
Consider the following dynamical equation which allows the solution to evolve continuously in time:
\begin{equation}\label{eqt:fixed-point_dynamic}
\partial_t f = \mtx{R}(f) - f,\quad f(x,0) = f_0(x) \in \mathbb{D}.
\end{equation}
Apparently, $f$ is a fixed point of $\mtx{R}$ if and only if it is a steady state of \eqref{eqt:fixed-point_dynamic}. As mentioned at the beginning of the section, we are going to construct a steady state by taking the limit of a sequence of time-periodic solutions of \eqref{eqt:fixed-point_dynamic}. In what follows, for a space-time function $f(x,t)$, we will sometimes use $\mtx{R}(f)(x,t)$ to denote $\mtx{R}(f(\cdot\,, t))(x)$.

We first prove the existence of at least one solution to the initial value problem \eqref{eqt:fixed-point_dynamic} for any time span.

\begin{lemma}\label{lem:dynamic_existence}
For any $f_0\in \mathbb{D}$ and any $T>0$, there exists at least one solution $f(x,t)\in C^1([0,T],\mathbb{M}_0)$ to the initial value problem \eqref{eqt:fixed-point_dynamic} such that $f(x,t)\leq M$ for all $x\in[-1,1]$ and $t\in[0,T]$. 
\end{lemma}

\begin{proof}
We shall construct a solution using the method of Euler polygons. For any $t\in[0,1)$, define the discrete Forward-Euler operator:
\[\mtx{E}_t(f_0) = f_0 + \int_0^t\big(\mtx{R}(f_0)-f_0\big)\idiff s = (1-t)f_0 + t\mtx{R}(f_0),\quad f_0\in \mathbb{M}_0.\]
By Proposition \ref{prop:implicit_existence}, $\mtx{E}_{t}(f_0)$ is non-increasing on $[0,1]$ since $f_0$ and $\mtx{R}(f_0)$ are both non-increasing on $[0,1]$. It also holds that 
\[\mtx{E}_t(f_0)(x)\geq (1-t) f_0(x) >0 ,\quad \text{for $x\in(-1,1)$, $t\in[0,1)$}.\]
Hence, we have $\mtx{E}_t(f_0)\in \mathbb{M}_0$. Moreover, by Lemma \ref{lem:Constant_upper_bound}, $f\in \mathbb{D}$ implies $\mtx{R}(f_0)(x)\leq M$ on $[-1,1]$, which further implies that $\mtx{E}_{t}(f_0)(x)\leq M$ on $[-1,1]$. 

Now, for any $T>0$, let $0=t_0<t_1<\cdots<t_n=T$ be a uniform division of the interval $[0,T]$ with $t_k-t_{k-1} = T/n$ for some sufficiently large integer $n$ such that $T/n<1 $. We then construct the Euler polygon $p_n(x,t)$ on $[-1,1]\times[0,T]$ inductively as $p_n(x,0) = f_0(x)$ and
\[p_n(x,t) = \mtx{E}_{t-t_{k-1}}(p_n(x,t_{k-1})),\quad \text{for $t\in[t_{k-1},t_k]$},\quad k=1,2,\dots,n.\]
It follows that $p_n(\cdot\,,t)\in \mathbb{M}_0$ and $p_n(\cdot\,,t)\leq M$ for all $t\in[0,T]$, and for a sufficiently large $n$,
\[p_n(x,t) \geq \left(1-\frac{T}{n}\right)^k f_0(x)\geq \frac{1}{2}\econst^{-t}f_0(x)>\frac{1}{2}\econst^{-T}f_0(x),\quad \text{for $x\in(-1,1)$ and $t\in[t_{k-1},t_k]$}.\]
Hence, by Lemma \ref{lem:Constant_upper_bound}, we have $\mtx{R}(p_n)(x,t)\leq M$ for all $t\in[0,T]$, and thus
\[|p_n(x,t) - p_n(x,\tilde t)| \leq M|t-\tilde t|,\quad \text{for all $x\in[-1,1]$ and $t,\tilde t\in[0,T]$}\]
This in particular means the functions $\{p_n(x,t)\}_n$ are equicontinuous in $t$ on the entire space-time domain $[-1,1]\times [0,T]$.

Next, we shall prove that the functions $\{p_n(x,t)\}_n$ are also equicontinuous in $x$ on the space-time domain $[-1,1]\times [0,T]$ for all sufficiently large $n$. Firstly, by the definition \eqref{eqt:D_definition} of $\mathbb{D}$, $|f_0|_{C^{1/2}}\leq \gamma$. Secondly, since 
\[M\geq p_n(x,t) \geq \frac{1}{2}\econst^{-T}f_0(x) \geq \frac{1}{2}\econst^{-T}u_\lambda(x)\geq \frac{\lambda}{2}\econst^{-T}(1-|x|),\quad \text{for all $x\in[-1,1]$, $t\in[0,T]$}\},\]
we have by Proposition \ref{prop:C_alpha} that 
\[|\mtx{R}(p_n)(\cdot\,,t)|_{C^{1/2}_x}\leq C_{M,\lambda,T}, \quad \text{for all $t\in[0,T]$},\]
for some constant $C_{M,\lambda,T}>0$ that only depends on $M$, $\lambda$, and $T$. Moreover, by the construction of the Euler polygons, we actually have, for $t\in[t_{k-1},t_k]$,
\begin{align*}
p_n(x,t) &= \big(1-(t-t_{k-1})\big)\left(1-\frac{T}{n}\right)^{k-1}f_0(x) \\
&\quad + \big(1-(t-t_{k-1})\big)\sum_{j=1}^{k-1}\left(1-\frac{T}{n}\right)^{k-j-1}\left(\frac{T}{n}\right)^j\mtx{R}(p_n)(x,t_{j-1}) + (t-t_{k-1})\mtx{R}(p_n)(x,t_{k-1}).
\end{align*}
Therefore, for any $k=1,\dots,n$ and any $t\in[t_{k-1},t_k]$,   
\begin{align*}
|p_n(\cdot\,,t)|_{C^{1/2}} &\leq \big(1-(t-t_{k-1})\big)\left(1-\frac{T}{n}\right)^{k-1} \gamma \\
&\quad + \big(1-(t-t_{k-1})\big)\sum_{j=1}^{k-1}\left(1-\frac{T}{n}\right)^{k-j-1}\left(\frac{T}{n}\right)^jC_{M,\lambda,T} + (t-t_{k-1})C_{M,\lambda,T}\\
&\leq \max\{\gamma\,,\,C_{M,\lambda,T}\}.
\end{align*}
This readily implies that $\{p_n(x,t)\}_n$ are equicontinuous in $x$ on $[-1,1]\times [0,T]$.

In summary, the functions $\{p_n(x,t)\}_{n=1}^\infty$ are equicontinuous in both $x\in[-1,1]$ and $t\in[0,T]$. Therefore, by the classical Arzela--Ascoli theorem, there is a subsequence $\{p_{n_j}(x,t)\}_{j=1}^\infty$ that converges uniformly in $L^\infty_{x,t}$ to some limit function $f(x,t)$ on $[-1,1]\times [0,T]$ such that $f(x,0)=f_0(x)$, and 
\[M\geq f(x,t)\geq \frac{1}{2}\econst^{-t}f_0(x)>0,\quad \text{for all $x\in(-1,1)$ and $t\in[0,T]$}.\]
It is then clear that $f(x,t)\in \mathbb{M}_0$ for each $t\in[0,T]$. Moreover, by the $L^\infty_x$-continuity of $\mtx{R}$ established in Proposition \ref{prop:R_continuity}, it is not hard to show that $f(x,t)$ is a solution to \eqref{eqt:fixed-point_dynamic}, and thus $f(x,t)\in C^1([0,T],\mathbb{M}_0)$. 
\end{proof}

Next, we shall argue that, if $f_0\in \mathbb{D}$, then the solution $f(x,t)$ will remain in $\mathbb{D}$ for all $t\geq 0$, and hence the solution is unique for all time by the uniform log-Lipschitz continuity of $\mtx{R}$ on $\mathbb{D}$.

\begin{lemma}\label{lem:dynamic_uniqueness}
For any $f_0\in \mathbb{D}$, there is a unique solution $f(x,t)\in C^1([0,+\infty),\mathbb{D})$ to the initial value problem \eqref{eqt:fixed-point_dynamic}.
\end{lemma}

\begin{proof}
We have already proved the existence of a solution $f(x,t)\in C^1([0,T],\mathbb{M}_0)$ for any time span $T$ with $f(x,t)\leq M$ for $t\in[0,T]$. We need to further show that $f_0\in \mathbb{D}$ implies $f(\cdot\,,t)\in \mathbb{D}$ for all $t\geq 0$.

We first show that $f(x,t)\leq v_\Lambda (x)$ for all $t\geq 0$. Suppose that this claim is not true. Since $f_0(x)\leq v_\Lambda (x)$ and since the solution $f(x,t)$ is $C^1$ in $t$, there must be some $x_*\in (-1,1)$ and $t_*\geq 0$ such that $f(x,t_*)\leq v_\Lambda(x)$ for all $x\in[-1,1]$, $f(x_*,t_*) = v_\Lambda(x_*)$, and $f(x_*,t)>v_\Lambda(x_*)$ for $t\in [t_*,t_*+\epsilon]$  with $\epsilon$ sufficiently small. It the follows that $\partial_t f(x_*,t_*)\geq 0$. However, by Lemma \ref{lem:upper_barrier} we have 
\[\partial_t f(x_*,t_*) = \mtx{R}(f)(x_*,t_*) - f(x_*,t_*) <0.\]
This contradiction implies that $f(x,t)\leq v_\Lambda(x)$ must hold for all $t\geq 0$. Similarly, we can use Lemma \ref{lem:lower_barrier} to show that $f(x,t)\geq u_\lambda(x)$ for all $t\geq 0$. 

Next, we show that $|f(\cdot\,,t)|_{C^{1/2}_x}\leq \gamma$ for all $t\geq 0$ with $\gamma$ given in the definition \eqref{eqt:D_definition} of $\mathbb{D}$. Now that we know $u_\lambda(x)\leq f(x,t)\leq M$ for all $t\geq 0$, we have by Corollary \ref{cor:C_alpha_D} that $|\mtx{R}(f(\cdot\,,t))|_{C^{1/2}_x}\leq \gamma$ for all $t\geq 0$. For any fixed $x,y$ such that $0\leq x<y \leq1$, define
\[\rho(t) = \frac{f(x,t)-f(y,t)}{(y-x)^{1/2}}.\]
Since $f(x,t)\in \mathbb{M}_0$ is non-increasing on $[0,1]$, we have $\rho(t)\geq 0$. We then find that, for $t\geq 0$, 
\begin{align*}
\rho'(t) = \frac{\partial_t f(x,t)-\partial_t f(y,t)}{(y-x)^{1/2}} &= \frac{\mtx{R}(f)(x,t)-\mtx{R}(f)(y,t)}{(y-x)^{1/2}} - \frac{f(x,t)-f(y,t)}{(y-x)^{1/2}} \\
&\leq |\mtx{R}(f(\cdot,t))|_{C^{1/2}_x} - \rho(t) \leq \gamma - \rho(t),
\end{align*}
that is 
\[\rho'(t)\leq \gamma - \rho(t),\quad t\geq 0.\]
This immediately implies that $\rho(t)\leq \gamma$ for all $t\geq 0$ given that $\rho(0)\leq \gamma$, which is true since $\rho(0)\leq |f_0|_{C^{1/2}} \leq \gamma$ for $f_0\in \mathbb{D}$. Since the pair $x,y\in [0,1]$ is arbitrary, we can conclude that $|f(\cdot\,,t)|_{C^{1/2}_x}\leq \gamma$ for all $t\geq 0$. 

Combining the results above, we can conclude that $f(\cdot\,,t)\in\mathbb{D}$ for all $t\in[0,T]$. In particular, $f(\cdot\,,T)\in \mathbb{D}$. Then, by Lemma \ref{lem:dynamic_existence}, the solution $f(x,t)$ can be continued beyond time $T$. Iterating this argument yields the existence of a solution $f(x,t)\in C^1([0,+\infty),\mathbb{D})$.

Finally, the uniqueness of the solution is guaranteed by the uniform log-Lipschitz continuity of $\mtx{R}$ over $\mathbb{D}$ established in Corollary \ref{cor:R_continuity_D}. In fact, if there were two solutions, $f_1(x,t)$ and $f_2(x,t)$, to \eqref{eqt:fixed-point_dynamic} with the same initial data $f_0\in \mathbb{D}$, then their $L^\infty_x$ difference, denoted by $\delta(t) = \|f_1(\cdot,t)-f_2(\cdot,t)\|_{L^\infty_x}$, satisfies 
\begin{align*}
\delta(t) &\leq \delta(0) + \int_0^t\left\|\big(\mtx{R}(f_1)(\cdot\,,s)-f_1(\cdot\,,s)\big) - \big(\mtx{R}(f_2)(\cdot\,,s)-f_2(\cdot\,,s)\big)\right\|_{L^\infty_x}\idiff s\\
&\leq \delta(0) + \int_0^t\left(\|\mtx{R}(f_1)(\cdot\,,s)- \mtx{R}(f_2)(\cdot\,,s)\|_{L^\infty_x} + \|f_1(\cdot\,,s)- f_2(\cdot\,,s)\|_{L^\infty_x}\right)\idiff s\\
&\leq \delta(0) + C\int_0^t\delta(s)\left(1 + \ln\left(2+\frac{1}{\delta(s)}\right)\right)\idiff s\\
&\leq \delta(0) + C\int_0^t\delta(s)\left(1 + \left|\ln\delta(s)\right|\right)\idiff s.
\end{align*}
This readily implies that
\[\delta(t)\leq \delta(0)^{\econst^{-Ct}}\cdot \econst^{1-\econst^{-Ct}},\]
at least for $t\in[0,t_0]$, where $t_0$ is the first time such that the right-hand side above achieves 1. But since $\delta(0)=0$, we obtain $\delta(t)=0$ for all $t\geq 0$. This completes the proof.
\end{proof}

So far, we have proved the global existence and uniqueness of a solution to \eqref{eqt:fixed-point_dynamic} for any $f_0\in \mathbb{D}$. It is then legit to define the forward solution operator
\[\mtx{S}_t(f_0) := f(\cdot,t),\quad f_0\in \mathbb{D},\ t\geq 0,\]
where $f(x,t)$ is the unique solution to the initial value problem \eqref{eqt:fixed-point_dynamic} with the initial state $f_0$. The next proposition shows that the map $\mtx{S}_t$ admits a fixed point in $\mathbb{D}$ for each $t\geq 0$. 

\begin{proposition}\label{prop:S_fixed-point}
Given any $t\geq 0$, $\mtx{S}_t$ maps $\mathbb{D}$ continuously (in the $L^\infty$-topology) into itself. As a result, $\mtx{S}_t$ has a fixed point $f\in \mathbb{D}$, that is, $\mtx{S}_t(f)=f$.
\end{proposition}

\begin{proof}
Lemma \ref{lem:dynamic_uniqueness} guarantees that $\mtx{S}_t$ maps $\mathbb{D}$ into itself. We only need to prove the $L^\infty$-continuity of $\mtx{S}_t$ on $\mathbb{D}$. Given any $f_1, f_2\in \mathbb{D}$, denote $\delta(t) = \|\mtx{S}_t(f_1)-\mtx{S}_t(f_2)\|_{L^\infty}$. Using a similar argument as in last part of the proof of Lemma \ref{lem:dynamic_uniqueness}, we find that 
\[\delta(t)\leq \delta(0)^{\econst^{-Ct}}\cdot \econst^{1-\econst^{-Ct}},\]
as long as $\delta(0)^{\econst^{-Ct}}\cdot \econst^{1-\econst^{-Ct}}\leq 1$, i.e. $\delta(0)\leq \econst^{1-\econst^{Ct}}$. This immediately implies that $\mtx{S}_t$ is continuous on $\mathbb{D}$ in the $L^\infty$-norm.

Now, since $\mathbb{D}$ is convex, closed, and compact in the $L^\infty$-topology, it follows from the Schauder fixed-point theorem that $\mtx{S}_t$ admits a fixed point $f=\mtx{S}_t(f)$ in $\mathbb{D}$.
\end{proof}

It is not hard to see that, for any $T>0$, a fixed point of $\mtx{S}_T$ actually gives a time-periodic solution of \eqref{eqt:fixed-point_dynamic} with period $T$. Suppose that $f_T$ is a fixed point of $\mtx{S}_T$, i.e. $\mtx{S}_T(f_T) = f_T$, whose existence is guaranteed by Proposition \ref{prop:S_fixed-point}. By the definition of $\mtx{S}_T$, this implies that, if $f(x,t)$ is the solution of \eqref{eqt:fixed-point_dynamic} with initial data $f_0=f_T$, then for all $t\geq 0$, 
\[f(\cdot\,,t+T) = \mtx{S}_t(f(\cdot\,,T))=\mtx{S}_t(\mtx{S}_T(f_0)) = \mtx{S}_t(f_0) = f(\cdot\,,t),\]
which means $f(x,t)$ is $T$-periodic in time. In fact, this further means that $f(\cdot\,,t)$ is a fixed point of $\mtx{S}_T$ for every $t\geq 0$, as we have shown above that $\mtx{S}_T(f(\cdot\,,t)) = f(\cdot\,,t+T) = f(\cdot\,,t)$.

It is possible that this periodic solution $f(x,t)$ with initial data $f_0=f_T$ is in fact a stationary solution in the sense that $f(x,t) \equiv f_0$ for all $t\geq 0$, that is, $\mtx{R}(f_0) -f_0 = \mtx{R}(f(\cdot\,,t)) -f(\cdot\,,t) = \partial_t f(\cdot\,,t)=0$. In this case, $f_0$ is already a fixed point of $\mtx{R}$, and we are done. Though this may not be the case in general, it inspires us to find a fixed point of $\mtx{R}$ by taking the period $T$ to $0$. In fact, since the time derivative $\partial_t f = \mtx{R}(f) - f$ is uniformly bounded for all $t\geq0$ and all $f_0\in \mathbb{D}$, it is conceivable that a time-periodic solution $f(x,t)$ shall behave like a stationary solution as the period tends to $0$.

\subsection{Existence of a fixed point of $\mathbf{R}$}
Based on the idea explained above, we can finally construct a fixed point of $\mtx{R}$ using the fixed points of $S_{t_n}$ for a vanishing sequence of $t_n$. This proves the existence of the desired domain $\Omega$ (and thus the domain $D$) in our main Theorem \ref{thm:main_formal}.

\begin{theorem}\label{thm:R_fixed-point}
The map $\mtx{R}$ has a fixed point $f\in \mathbb{D}$.
\end{theorem}

\begin{proof}
Let $t_n = 1/n$ for $n\geq 1$. By Proposition \ref{prop:S_fixed-point}, for each $n$, $S_{t_n}$ has a fixed point $f_n\in\mathbb{D}$, i.e. $\mtx{S}_{t_n}(f_n) = f_n$. We shall first show that 
\begin{equation}\label{eqt:existence_step1}
\|f_n-\mtx{R}(f_n)\|_{L^\infty}\leq C Mt_{n} \left(1 + \ln \left(1+\frac{1}{Mt_n}\right)\right),
\end{equation}
for some uniform constant $C$. Let $\hat f_n(x,t)$ denote the solution of \eqref{eqt:fixed-point_dynamic} with initial data $\hat f_n(x,0) = f_n(x)$. Recall that 
\[\hat f_n(\cdot,t) = f_n + \int_0^t \left(\mtx{R}(\hat f_n)(\cdot,s) - \hat f_n(\cdot,s)\right)\idiff s.\]
Then, for any $t\in [0,t_n]$, 
\[\|\hat f_n(\cdot,t) - f_n\|_{L^{\infty}_x} \leq  \int_0^t \|\mtx{R}(\hat f_n)(\cdot,s) - \hat f_n(\cdot,s)\|_{L^\infty_x}\idiff s \leq M t\leq Mt_n.\]
By Proposition \ref{prop:R_continuity}, we also have
\[\|\mtx{R}(\hat f_n)(\cdot,t) - \mtx{R}(f_n)\|_{L^{\infty}_x} \leq CMt \ln \left(1+\frac{1}{Mt}\right) \leq CMt_n \ln \left(1+\frac{1}{Mt_n}\right).\]
Note that the fixed-point relation $\mtx{S}_{t_n}(f_n) = f_n$ implies 
\[\int_0^{t_n}\left(\mtx{R}(\hat f_n)(\cdot,s)-\hat f_n(\cdot,s)\right)\idiff s = \mtx{S}_{t_n}(f_n) - f_n = 0.\]
Combining these calculations we find that 
\begin{align*}
\|f_n-\mtx{R}(f_n)\|_{L^{\infty}} &= \left\|f_n - \mtx{R}(f_n) + \frac{1}{t_n}\int_0^{t_n}\left(\mtx{R}(\hat f_n)(\cdot,s)-\hat f_n(\cdot,s)\right)\idiff s\right\|_{L^{\infty}_x}\\
&\leq \frac{1}{t_n}\int_0^{t_n}\|\mtx{R}(\hat f_n(\cdot,s))-\mtx{R}(f_n)\|_{L^{\infty}_x}\idiff s + \frac{1}{t_n}\int_0^{t_n}\|\hat f_n(\cdot,s)-f_n\|_{L^{\infty}_x}\idiff s\\
&\leq C Mt_{n} \left(1 + \ln \left(1+\frac{1}{Mt_n}\right)\right),
\end{align*}
which is \eqref{eqt:existence_step1}.

Next, since $\mathbb{D}$ is compact in the $L^\infty$-topology, the sequence $\{f_n\}_{n=1}^{\infty}$ has a subsequence, still denoted by $\{f_n\}_{n=1}^{\infty}$, that converges to a limit function $f\in \mathbb{D}$ in the $L^\infty$-norm. Hence, taking $n\rightarrow +\infty$ on both sides of \eqref{eqt:existence_step1} yields 
\[\|f-\mtx{R}(f)\|_{L^\infty}=0,\]
which means $f$ is a fixed point of $\mtx{R}$ in $\mathbb{D}$. 
\end{proof}

\section{Regularity of a fixed point}\label{sec:regularity}
We establish in this section some regularity estimates for a fixed point $f$ of $\mtx{R}$, thus proving the regularity results in Theorem \ref{thm:main_formal}. Note that, as we have not proved the uniqueness of the fixed point, the following regularity results apply to any fixed point in $\mathbb{D}$.

We first provide some quantitative estimates of the derivative $f'$ and show that the boundary of the whole domain $D$ described in Theorem \ref{thm:main_formal} is a $C^1$ curve.

\begin{theorem}\label{thm:C1_regularity}
Let $f\in \mathbb{D}$ be a fixed point of $\mtx{R}$, i.e. $\mtx{R}(f) = f$. Then, $f\in \mathbb{M}_1 \subset C^1(-1,1)$, and it satisfies 
\[C_1 x \leq  -f'(x)\leq C_2 x,\quad \text{for $x\in[0,1/2]$},\]
and
\[\tilde C_1\ln\left(1+\frac{1}{1-x}\right)\leq -f'(x)\leq \tilde C_2 \ln\left(1+\frac{1}{1-x}\right),\quad \text{for $x\in[1/2,1)$},\]
where the constants $C_1,C_2,\tilde C_1,\tilde C_2>0$ only depend on the (constant) parameters $M,\Lambda,\lambda$ in the definition \eqref{eqt:D_definition} of $\mathbb{D}$. In particular, $\lim_{x\rightarrow 1^-}f'(x) = -\infty$. This means $f^{-1}\in C^1[0,f(0))$, where $f^{-1}(y)$ for $y\in[0,f(0)]$ is the inverse function of $f(x)$ for $x\in[0,1]$. As a consequence, the boundary of the domain $D:=\{(x_1,x_2):x_1\in[-1,1],x_2\in[-f(x_1),f(x_1)]\}$ is a simple closed $C^1$-curve.
\end{theorem}

\begin{proof}
One should keep in mind that $\Lambda > M > 1 > \lambda>0$.
It is a direct result of Proposition \ref{prop:implicit_existence} that $f=\mtx{R}(f)\in \mathbb{M}_1$. We thus only need to establish the bounds for $f'(x)$. Since $f$ is an even function, we may only consider $x\in[0,1]$. Recall that $f\in \mathbb{D}$ implies 
\[f(x)\geq u_\lambda(x) = \lambda (1-x^2)\geq \lambda (1-x).\] 
Also, with $W$ defined in \eqref{eqt:v_definition}, we have
\[W(t)\geq \frac{t/2}{1+\ln(1+2/t)},\quad t\geq 0,\]
which implies
\[W^{-1}(s) \leq 2s\big(1+\ln(1+1/s)\big),\quad s\geq 0.\]
Hence, 
\begin{align*}
f(x)\leq v_{\Lambda}(x) &= M\cdot W^{-1}\left(\frac{\Lambda (1-x)}{M}\right)\\
&\leq 2\Lambda (1-x)\left(1+\ln\left(1+\frac{M}{\Lambda(1-x)}\right)\right) \leq C_{\Lambda,M} (1-x)\ln\left(1+\frac{1}{1-x}\right).
\end{align*}

For $x\in[0,1/2]$, we invoke the Lemma \ref{lem:Fx1_estimate_2} and the second bound in Lemma \ref{lem:Fx1_estimate} to obtain 
\begin{align*}
-F_{x_1}(x,f(x);f) &\geq C_\lambda \frac{x}{1+f(0)^2}\ln\left(1+\frac{1}{(1-x)^2+f(x)^2}\right) \\
&\geq C_\lambda \frac{x}{1+M^2}\ln\left(1+\frac{1}{1+M^2}\right) \geq C_{\lambda,M} x,
\end{align*}
and 
\[
-F_{x_1}(x,f(x);f) \leq \frac{x}{f(x)}\leq \frac{x}{\lambda (1-x)} \leq C_{\lambda} x.
\]

For $x\in[1/2,1)$, we use the Lemma \ref{lem:Fx1_estimate_2} and the first bound in Lemma \ref{lem:Fx1_estimate} to get
\begin{align*}
-F_{x_1}(x,f(x);f) &\geq C_\lambda \frac{x}{1+f(0)^2}\ln\left(1+\frac{1}{(1-x)^2+f(x)^2}\right) \\
&\geq C_\lambda \frac{1}{1+M^2}\ln\left(1+\frac{1}{C_{\Lambda,M} (1-x)^2\big(\ln(1+1/(1-x))\big)^2}\right)\\
&\geq C_{\lambda,M} \ln\left(1+ \frac{1}{C_{\Lambda,M}(1-x)^2|\ln(1-x)|^2}\right)\\
&\geq C_{\Lambda,\lambda,M} \ln\left(1+\frac{1}{1-x}\right),
\end{align*}
and 
\begin{align*}
-F_{x_1}(x,f(x);f) &\leq C\left(1 + \ln\left(1+\frac{f(0)}{f(x)}\right)\right)\\
&\leq C\left(1 + \ln\left(1+\frac{M}{\lambda(1-x)}\right)\right) \leq C_{\lambda,M} \ln\left(1+\frac{1}{1-x}\right).
\end{align*}

Moreover, we have by Lemma \ref{lem:Fx2_estimate} and Lemma \ref{lem:Fx2_estimate_2} that, for any $x\in[0,1)$, 
\[C\geq -F_{x_2}(x,f(x);f) \geq \frac{C_{\lambda}}{\max\{1,f(x)^3\}} \geq C_{\lambda,M}.\]

Then, combining all the estimates above and using the implicit function theorem that
\[-f'(x) = \frac{F_{x_1}(x,f(x);f)}{F_{x_2}(x,f(x);f)},\]
we obtain the claimed estimates for $-f'(x)$ on $[0,1)$. It then follows that $\lim_{x\rightarrow 1^-}f'(x) = -\infty$. Note that these estimates for $f'$ immediately implies that $(f^{-1})'(y)<0$ for $y\in(0,f(0))$, 
\[-(f^{-1})'(y) \leq \frac{C_{\Lambda,\lambda,M}}{\ln\left(1+\frac{1}{1-f^{-1}(y)}\right)}\leq \frac{C_{\Lambda,\lambda,M}}{\ln(1+\lambda/y)}<+\infty,\quad y\in(0,f(1/2)], \]
\[-(f^{-1})'(y) \leq \frac{C_{\Lambda,\lambda,M}}{f^{-1}(y)} <+\infty,\quad y\in[f(1/2),f(0)),\]
and $\lim_{y\rightarrow 0^+}(f^{-1})'(y) = 0$. Hence, $f^{-1}\in C^1[0,f(0))$. The proof is thus completed.
\end{proof}

We remark that the $C^1$-regularity of $\partial D = \{(x_1,x_2):x_1\in[-1,1],x_2=\pm f(x_1)\}$ is optimal at the two points $(\pm1,0)$, since $f(x)$ behaves like $(1-|x|)\ln(1+1/(1-|x|))$ near $x=\pm1$. However, we can further show that a fixed point $f$ of $\mtx{R}$ is in fact infinitely smooth in $(-1,1)$.

\begin{theorem}\label{thm:analyticity}
Let $f\in \mathbb{D}$ be a fixed point of $\mtx{R}$. Then $f$ is locally analytic in $(-1,1)$.
\end{theorem}

\begin{proof}
Write $F(\vx) = F(\vx;f)$ and $\phi(\vx) = \phi(\vx;f)$. As before, denote $\mathbb{H} := \{(x_1,x_2): x_1\in \R,\, x_2\geq 0\}$, $\Omega := \{(x_1,x_2): x_1\in[-1,1],\, x_2\in[0,f(x_1)]\}$, and $\Gamma := \{(x_1,x_2): x_1\in[-1,1],\, x_2=f(x_1)\}$. We also denote $\Gamma^\circ = \{(x_1,x_2): x_1\in(-1,1),\, x_2=f(x_1)\}$.

Recall $\phi$ solves $-\Delta\phi = \mathbf{1}_\Omega$ in $\mathbb{H}$ and $\phi|_{x_2=0}=0$. By the classic elliptic regularity theory, $\phi\in C^{1,\alpha}_{loc}(\mathbb{H})$ for any $\alpha\in(0,1)$. Hence, $F = \phi/x_2 - c(f)\in C^{1,\alpha}_{loc}(\mathbb{H}^\circ)$. Take an arbitrary $x_0\in (-1,1)$. By the definition \eqref{eqt:D_definition} of $\mathbb{D}$, $M\geq f(x_0)\geq u_{\lambda}(x_0) \geq \lambda(1-|x_0|)$. Then, by Lemma \ref{lem:Fx2_estimate}, $|F_{x_2}|\geq C$ in a neighborhood of $(x_0,f(x_0))$ with some universal constant $C>0$. Since $f$ solves $F(x,f(x))=0$, by the implicit function theorem, $f$ is $C^{1,\alpha}$ in a neighborhood of $x_0$, with the $C^{1,\alpha}$-norm depending on $\alpha$ and $1-|x_0|$. Therefore, $f\in C^{1,\alpha}_{loc}((-1,1))$. 

Denote $\psi(\vx) := \phi(\vx) - c(f)x_2$. We may also rewrite the equation for $\phi$ in terms of $\psi$ as 
\[-\Delta\psi = \mathbf{1}_{\{\psi>0\}}.\]
This is in the form of the unstable free boundary problem \cite{10.1215/S0012-7094-07-13624-X}. Moreover, for arbitrary $(x_1,x_2) = (x,f(x))\in\Gamma^\circ$ with $x$ in a suitable neighborhood of $x_0$, one has
\[
|\nabla\psi(x_1,x_2)|\geq |\psi_{x_2}(x_1,x_2)| = |\partial_{x_2}(x_2F(x_1,x_2))| = |x_2F_{x_2}(x_1,x_2)| \geq C(1-|x_0|)>0,
\]
which implies that in the suitable neighborhood of $(x_0,f(x_0))$, points on $\Gamma$ are not singular points. Recall $\psi = 0$ on $\Gamma$. In view of \cite[Theorem 3.1']{kinderlehrer1978regularity} (see also \cite[Theorem 8.1]{10.1215/S0012-7094-07-13624-X}), in order to show the local analyticity of $f$ (namely the local analyticity of $\Gamma$), it suffices to verify that, in a neighborhood of $(x_0,f(x_0))$, $\psi$ (or $\phi$) is $C^2$ up to $\Gamma$ on either side of $\Gamma$.

By a direct calculation, one finds that the Hessian of $\phi$ for $\vx\in \mathbb{H}$ is given by 
\[\nabla^2\phi(\vx) = -\frac{1}{2}\kappa_\Omega(\vx)\cdot \Id + \frac{1}{2\pi}\mathrm{P.V.}\int_{\Omega}G(\vx,\vy)\idiff{\vy},\]
where
\[
\kappa_\Omega(\vx) = \left\{\begin{array}{ll}
1, &\vx\in \Omega^\circ,\\
1/2, & \vx\in \Gamma^\circ,\\
0,& \text{otherwise},
\end{array}\right.
\]
and
\begin{equation}\label{eqt:analyticity_step2}
G(\vx,\vy) = \sigma(\vx-\vy) - \sigma(\vx - \bar\vy),\quad \sigma(\vz) = \frac{1}{|\vz|^4}\left(\begin{array}{cc}
z_1^2-z_2^2 & 2z_1z_2\\
2z_1z_2 & z_2^2-z_1^2
\end{array}\right),
\end{equation}
with $\bar\vy = (y_1,-y_2)$. The $L^\infty$-bound of $\nabla^2\phi$ can be derived by following the argument in \cite[Proposition 1 \& Geometric Lemma]{bertozzi1993global}. Nevertheless, we shall give a self-contained proof using a refined estimate in Lemma \ref{lem:Gamma_local_estimate} below. It suffices to bound $|\nabla^2\phi|$ near $\Gamma$. Given $\vx = (x_1,x_2)\in \Omega^\circ$, suppose $\delta := \mathrm{dist}(\vx,\Gamma)\ll1$ and it is achieved at some $\vx_*\in \Gamma^\circ$. Let $(\cos\beta,\sin\beta)$ be the outward unit normal of $\Gamma$ at $\vx_*$ for some $\beta\in[0,2\pi)$. Note that $|\sigma(\vz)|\leq C|\vz|^{-2}$. Then, with $\mu\in(0,x_2]$, we can use Lemma \ref{lem:Gamma_local_estimate} to find that
\begin{align*}
|\nabla^2\phi(\vx)| &\leq C + C\int_{\Omega}|\sigma(\vx-\bar \vy)|\idiff \vy + C\int_{\Omega\backslash B_\mu(\vx)}|\sigma(\vx-\vy)|\idiff \vy + C\left|\mathrm{P.V.}\int_{\Omega\cap B_\mu(\vx)}\sigma(\vx-\vy)\idiff \vy\right|\\
&\leq C + C\int_\mu^{2f(0)+2}\frac{1}{r}\idiff r + C\left|\int_{\min\{\delta,\mu\}/\mu}^1 \sqrt{1-s^2} \idiff s\left(\begin{array}{cc}
\cos 2\beta & \sin 2\beta\\
\sin2\beta & -\cos2\beta
\end{array}\right)\right| + C\mu^\alpha\\
&\leq C(1 + |\ln\mu| + \mu^\alpha),
\end{align*}
where $C$ depends on $\alpha$ and the local $C^{1,\alpha}$-norm of $f$. Choosing $\mu=x_2$ yields the desired $L^\infty$-estimate for $\nabla^2\phi$ in $\Omega^\circ$. The estimate in $\Omega^c$ can be obtained in a similar way. This allows us to conclude via the implicit function theorem again that $f\in C^{1,1}_{loc}((-1,1))$. In particular, the local $C^{1,1}$-norm of $f$ depends on $\alpha$, $1-|x_0|$, and the local $C^{1,\alpha}$-norm of $\phi$.

To show the desired $C^2$-regularity of $\phi$, it suffices to study the continuity of $\nabla^2\phi$ near $\Gamma$. Take $\vx,\vx'\in \Omega^\circ\cap B_{r_0}((x_0,f(x_0))$ with $r_0\ll \lambda(1-|x_0|)/2$. Denote $l=|\vx-\vx'|$ and $\vz = (\vx + \vx')/2$. We may assume that $l\ll \mu:=l^{1/2}\ll r_0$, so that $B_\mu(\vx),B_\mu(\vx'), B_\mu(\vz)\subset \mathbb{H}^\circ$. We then compute that  
\begin{align*}
|\nabla^2\phi(\vx) - \nabla^2\phi(\vx')| &= \frac{1}{2\pi}\left|\mathrm{P.V.}\int_\Omega G(\vx,\vy)\idiff{\vy} - \mathrm{P.V.}\int_\Omega G(\vx',\vy)\idiff{\vy}\right|\\
&\leq \frac{1}{2\pi}\int_\Omega \left|\sigma(\vx-\bar\vy) - \sigma(\vx'-\bar\vy)\right| \idiff{\vy}\\
&\quad + \frac{1}{2\pi}\int_{\Omega\backslash B_\mu(\vz)} \left|\sigma(\vx-\vy) - \sigma(\vx'-\vy)\right| \idiff{\vy}\\
&\quad +  \frac{1}{2\pi}\int_{\Omega \cap (B_\mu(\vz)\Delta B_\mu(\vx))} \left|\sigma(\vx-\vy)\right| \idiff{\vy}  +  \frac{1}{2\pi}\int_{\Omega \cap (B_\mu(\vz)\Delta B_\mu(\vx'))} \left| \sigma(\vx'-\vy)\right| \idiff{\vy}\\
&\quad + \frac{1}{2\pi}\left|\mathrm{P.V.}\int_{\Omega\cap B_\mu(\vx)} \sigma(\vx-\vy)\idiff{\vy} - \mathrm{P.V.}\int_{\Omega\cap B_\mu(\vx')} \sigma(\vx'-\vy)\idiff{\vy}\right|\\
&=: I_1 + I_2 + I_3 + I_4 .
\end{align*}
Here $A_1\Delta A_2$ denotes the symmetric difference of the sets $A_1$ and $A_2$. Note that $|\sigma(\vz)|\leq C|\vz|^{-2}$ and $|\nabla \sigma(\vz)|\leq C|\vz|^{-3}$. Hence, we can calculate in the polar coordinate to find that 
\[I_1 + I_2 + I_3 \leq C|\vx -\vx'|\int_\mu^\infty\frac{1}{r^2}\idiff r  + C \int_{\mu-l/2}^{\mu+l/2}\frac{1}{r}\idiff r\leq C\frac{l}{\mu} \leq Cl^{1/2}.\]

It remains to control $I_4$. Denote $\delta := \mathrm{dist}(\vx,\Gamma)$ and $\delta' := \mathrm{dist}(\vx',\Gamma)$. By the way we choose $\vx,\vx'$, we have $\delta,\delta'\leq r_0$. There exist some $\vx_*, \vx_*'\in \Gamma^\circ\cap B_{2r_0}((x_0,f(x_0))$ such that $\delta = |\vx - \vx_*|$ and $\delta' = |\vx' -\vx_*'|$. It is not hard to verify that $|\delta-\delta'|\leq |\vx-\vx'| = l$. Let $(\cos\beta,\sin\beta)$ and $(\cos\beta',\sin\beta')$ be the outward normal directions of $\Gamma$ (with respect to $\Omega$) at $\vx_*$ and $\vx_*'$, respectively, for some $\beta,\beta'\in[0,2\pi)$. It follows that
\[|\beta-\beta'|\leq  C |\vx_*-\vx_*'| \leq C(\delta + \delta' + l) \leq C(\max\{\delta\,,\delta'\} + l), \]
where the constant $C$ depends on the local $C^{1,1}$-norm of $f$ in $[x_0-2r_0, x_0+2r_0]$. Recall $\mu = l^{1/2}$, and let $\tilde \delta = \min\{\delta,\mu\}$ and $\tilde \delta' = \min\{\delta',\mu\}$. We then apply Lemma \ref{lem:Gamma_local_estimate} below (with $\alpha=1$) to obtain
\begin{align*}
|I_4| &\leq C \mu + \frac{1}{\pi}\left|\int_{\tilde\delta/\mu}^1 \sqrt{1-s^2}\idiff s\left(\begin{array}{cc}
\cos 2\beta & \sin 2\beta\\
\sin2\beta & -\cos2\beta
\end{array}\right)
-
\int_{\tilde \delta'/\mu}^1 \sqrt{1-s^2} \idiff s\left(\begin{array}{cc}
\cos 2\beta' & \sin 2\beta'\\
\sin2\beta' & -\cos2\beta'
\end{array}\right)
\right|\\
&\leq C\mu + \frac{C|\tilde \delta-\tilde \delta'|}{\mu} + C|\beta-\beta'|\left(1-\frac{\max\{\tilde \delta \,,\tilde \delta'\}}{\mu}\right)\\
&\leq C\mu + \frac{C|\delta-\delta'|}{\mu} + C(\max\{\delta\,,\delta'\}+l)\left(1-\frac{\min\{\max\{\delta\,,\delta'\}\,,\mu\}}{\mu}\right)\\
&\leq C\mu + C\frac{l}{\mu} + Cl\leq Cl^{1/2}.
\end{align*}

Finally, combining the estimates above, we obtain 
\[|\nabla^2\phi(\vx) - \nabla^2\phi(\vx')|\leq I_1+I_2+I_3+I_4\leq C|\vx-\vx'|^{1/2},\]
where $C$ depends on the local $C^{1,1}$-norm of $f$. This implies the continuity of $\nabla^2\phi$ in $\Omega^\circ$ up to $\Gamma$. The continuity of $\nabla^2\phi$ on the other side of $\Gamma$ can be justified similarly. Therefore, Theorem 3.1' in \cite{kinderlehrer1978regularity} applies and the local analyticity of $\Gamma$ follows. 
\end{proof}

The following auxiliary lemma is needed in the proof above.

\begin{lemma}\label{lem:Gamma_local_estimate}
Let $\Gamma\subset \R^2$ be a simple closed $C^1$-curve and let $\Omega$ be its interior domain. Given $\vx\in \Omega^\circ$, suppose $\delta := \mathrm{dist}(\vx,\Gamma)\ll1$ and is achieved at some point $\vx_*\in \Gamma$, i.e. $|\vx-\vx_*|=\delta$. Let $\mu>0$ be sufficiently small such that, either $\mu\leq \delta$, or $\Gamma\cap B_\mu(\vx)$ divides $B_\mu(\vx)$ into two simply connected parts. Furthermore, suppose in the latter case $\Gamma$ is $C^{1,\alpha}$ in $B_\mu(\vx)$ for some $\alpha\in(0,1]$. Let $(\cos\beta,\sin\beta)$ be the outward unit normal of $\Gamma$ at $\vx_*$ for some $\beta\in[0,2\pi)$. Let $\sigma(\vz)$ be defined in \eqref{eqt:analyticity_step2}. Then, if $\mu\leq \delta$, 
\[\mathrm{P.V.}\int_{\Omega\cap B_\mu(\vx)} \sigma(\vx-\vy)\idiff{\vy} = 0;\]
if instead $\mu>\delta$,
\begin{equation}\label{eqt:Gamma_local}
\left|\mathrm{P.V.}\int_{\Omega\cap B_\mu(\vx)} \sigma(\vx-\vy)\idiff{\vy} + 2\int_{\delta/\mu}^1 \sqrt{1-s^2}\idiff s\left(\begin{array}{cc}
\cos 2\beta & \sin 2\beta\\
\sin2\beta & -\cos2\beta
\end{array}\right)\right|\leq C\mu^\alpha,
\end{equation}
where the constant $C$ depends on $\alpha$ and the local $C^{1,\alpha}$-norm of $\Gamma$. 
\end{lemma}

\begin{proof}
If $\mu\leq\delta$, then  
\[\mathrm{P.V.}\int_{\Omega\cap B_\mu(\vx)} \sigma(\vx-\vy)\idiff{\vy} = \mathrm{P.V.}\int_{B_\mu(\vx)} \sigma(\vx-\vy)\idiff{\vy} = \mathrm{P.V.}\int_{B_\mu(0)} \sigma(\vy)\idiff{\vy} = 0,\]
which is the first claim of the lemma.

Otherwise, if $\mu>\delta$, we shall compare the integral on $\Omega\cap B_\mu(\vx)$ to that on $U\cap  B_\mu(\vx)$, where $U := \{\vy\in \R: \la\vy-\vx_*\,,(\cos\beta,\sin\beta)\ra<0\}$. We can use the $C^{1,\alpha}$-regularity of $\Gamma$ to find that
\begin{equation}\label{eqt:Gamma_local_step}
\begin{split}
\left|\mathrm{P.V.}\int_{\Omega\cap B_\mu(\vx)} \sigma(\vx-\vy)\idiff{\vy} - \mathrm{P.V.}\int_{U\cap B_\mu(\vx)} \sigma(\vx-\vy)\idiff{\vy}\right| &\leq \int_{(\Omega\Delta U)\cap  B_\mu(\vx)} |\sigma(\vx-\vy)|\idiff{\vy}\\
&\leq C\int_{-\mu}^\mu \idiff z_1 \int_{\delta-Cz_1^{1+\alpha}}^{\delta+Cz_1^{1+\alpha}}\frac{1}{z_1^2+z_2^2}\idiff z_2\\
&\leq C\int_{-\mu}^\mu z_1^{\alpha-1}\idiff z_1\leq C\mu^\alpha,
\end{split}
\end{equation}
where $C$ depends on the local $C^{1,\alpha}$-norm of $\Gamma$. Moreover, thanks to the simpler shape of the domain $U\cap B_\mu(\vx)$, we can use the polar coordinate to derive that
\begin{align*}
\mathrm{P.V.}\int_{U\cap B_\mu(\vx)} \sigma(\vx-\vy)\idiff{\vy} &= \mathrm{P.V.}\int_{B_\mu(\vx)} \sigma(\vx-\vy)\idiff{\vy} - \int_{B_\mu(\vx)\backslash U} \sigma(\vx-\vy)\idiff{\vy} \\
&= - \int_{B_\mu(\vx)\backslash U} \sigma(\vx-\vy)\idiff{\vy} \\
&= -\int_\delta^\mu \frac{1}{r} \int_{\beta - \arccos(\delta/r)}^{\beta +  \arccos(\delta/r)} \left(\begin{array}{cc}
\cos 2\theta & \sin 2\theta\\
\sin2\theta & -\cos2\theta 
\end{array}\right)\idiff \theta \idiff r\\
&= -\int_\delta^\mu \frac{1}{2r} \left.\left(\begin{array}{cc}
\sin 2\theta & -\cos 2\theta\\
-\cos2\theta & -\sin 2\theta 
\end{array}\right)\right|_{\beta - \arccos(\delta/r)}^{\beta + \arccos(\delta/r)}\idiff r \\
&= -\int_\delta^\mu \frac{1}{r} \sin\big(2\arccos(\delta/r)\big)\left(\begin{array}{cc}
\cos2\beta & \sin2\beta\\
\sin2\beta & -\cos 2\beta
\end{array}\right)\idiff r \\
&= -\int_\delta^\mu \frac{2\delta}{r^2}\left(1-\frac{\delta^2}{r^2}\right)^{1/2} \idiff r \left(\begin{array}{cc}
\cos2\beta & \sin2\beta\\
\sin2\beta & -\cos 2\beta
\end{array}\right)\\
&= -2\int_{\delta/\mu}^1 \sqrt{1-s^2} \idiff s \left(\begin{array}{cc}
\cos2\beta & \sin2\beta\\
\sin2\beta & -\cos 2\beta
\end{array}\right).
\end{align*}
Combining this and \eqref{eqt:Gamma_local_step} yields \eqref{eqt:Gamma_local}.
\end{proof}

\section{Numerical computation}\label{sec:numerical}
In this final section, we explain how to numerically obtain a solution $f$ of the problem \eqref{eqt:solution_condition_2}-\eqref{eqt:c_definition_2} using a simple fixed-point iteration method.

In the previous sections, we have shown theoretically the existence of a nontrivial solution of \eqref{eqt:solution_condition_2} by proving the existence of a fixed point of the nonlinear map $\mtx{R}$. Then, a natural idea to obtain an approximate numerical solution is by performing the fixed point iteration of $\mtx{R}$ from a suitable initial guess. However, since the map $\mtx{R}$ is determined by \eqref{eqt:solution_condition_2} in an implicit way (Proposition \ref{prop:implicit_existence}), it is not easy to implement its fixed-point iteration numerically with high accuracy. Yet, this implicit fixed-point iteration was employed in \cite{childress2018eroding} to compute Sadovskii's vortex patch, where an area renormalization technique must be applied in every iteration to guarantee convergence of the algorithm.

Instead, we choose to employ a simpler explicit iteration. Let us rewrite \eqref{eqt:solution_condition_2} as
\[\frac{\phi(x,f(x);f)}{c(f)} = f(x),\quad x\in[-1,1].\]
This naturally suggests the explicit iteration scheme
\begin{equation}\label{eqt:explicit_iteration}
f^{(n+1)} = \mtx{P}(f^{(n)}),\quad f^{(0)}\in\mathbb{V}, 
\end{equation}
where
\[\mtx{P}(f) := \frac{\phi(x,f(x);f)}{c(f)}.\]
Apparently, a fixed point of this map $\mtx{P}$ is also a nontrivial solution of \eqref{eqt:solution_condition_2}. 

The explicit iteration scheme \eqref{eqt:explicit_iteration} is much easier to implement numerically, for it only involves the computation of one-dimensional integrations of absolutely integrable functions (see \eqref{eqt:phi} and \eqref{eqt:cf}). Though we have not proved the convergence of this iteration scheme, it preforms excellently well in providing an accurate approximate solution and exhibits convergence to a unique terminal solution for a wide range of initial states $f^{(0)}$. The accuracy of an approximate solution $f^{(n)}$ can be measured by its residual
\[r^{(n)} = \mtx{P}(f^{(n)}) - f^{(n)}.\]
Numerically, the problem is discretized on a suitable adaptive mesh on $[-1,1]$, and the residual is evaluated on the grid points. We omit the routine implementation details. We have performed the iteration scheme numerically for different suitable initial guesses such as $f^{(0)} = (1-x^2)\ln(1+1/(1-x^2))$ and $f^{(0)} = \cos(\pi x/2)$. We observe that the solution $f^{(n)}$ seems to converge to a unique solution in all trials, and the (grid-value) maximum residual $|r^{(n)}|_{\max}$ decreases very quickly. As illustrated in Figure \ref{fig:numerical}, in all our numerical experiments, it only takes dozens of iterations for $f^{(n)}$ to become visually indistinguishable from the same profile (indicated by a dashed curve), with $|r^{(n)}|_{\max}$ dropping from $O(1)$ to below $10^{-4}$. We remark that the solution profile and the corresponding streamlines plotted in Figure \ref{fig:dipole}(a) are generated by an accurate approximate solution $\bar f$ obtained numerically via our fixed-point iteration \eqref{eqt:explicit_iteration} with $|\mtx{P}(\bar f)-\bar f|_{\max}$ smaller than $10^{-8}$. One should note that, unlike the implicit map $\mtx{R}$, the explicit map $\mtx{P}$ does not preserve the monotonicity of $f$ on $[0,1]$ (see Figure \ref{fig:numerical}(c)), which is why we choose to study the fixed-point problem of $\mtx{R}$ in the previous sections.

Our numerical results strongly suggest the uniqueness of the nontrivial solution to \eqref{eqt:solution_condition_2} and thus the uniqueness of the solution described in Theorem \ref{thm:main_formal} under scaling normalization. Moreover, the unique terminal profile in our numerical computation seems to be a concave function, even if the initial guess is not concave. This suggests that the vorticity support $D$ (or $\Omega$) is a convex domain.

\begin{figure}[!ht]
\centering
    \begin{subfigure}[b]{0.45\textwidth}
        \includegraphics[width=1\textwidth]{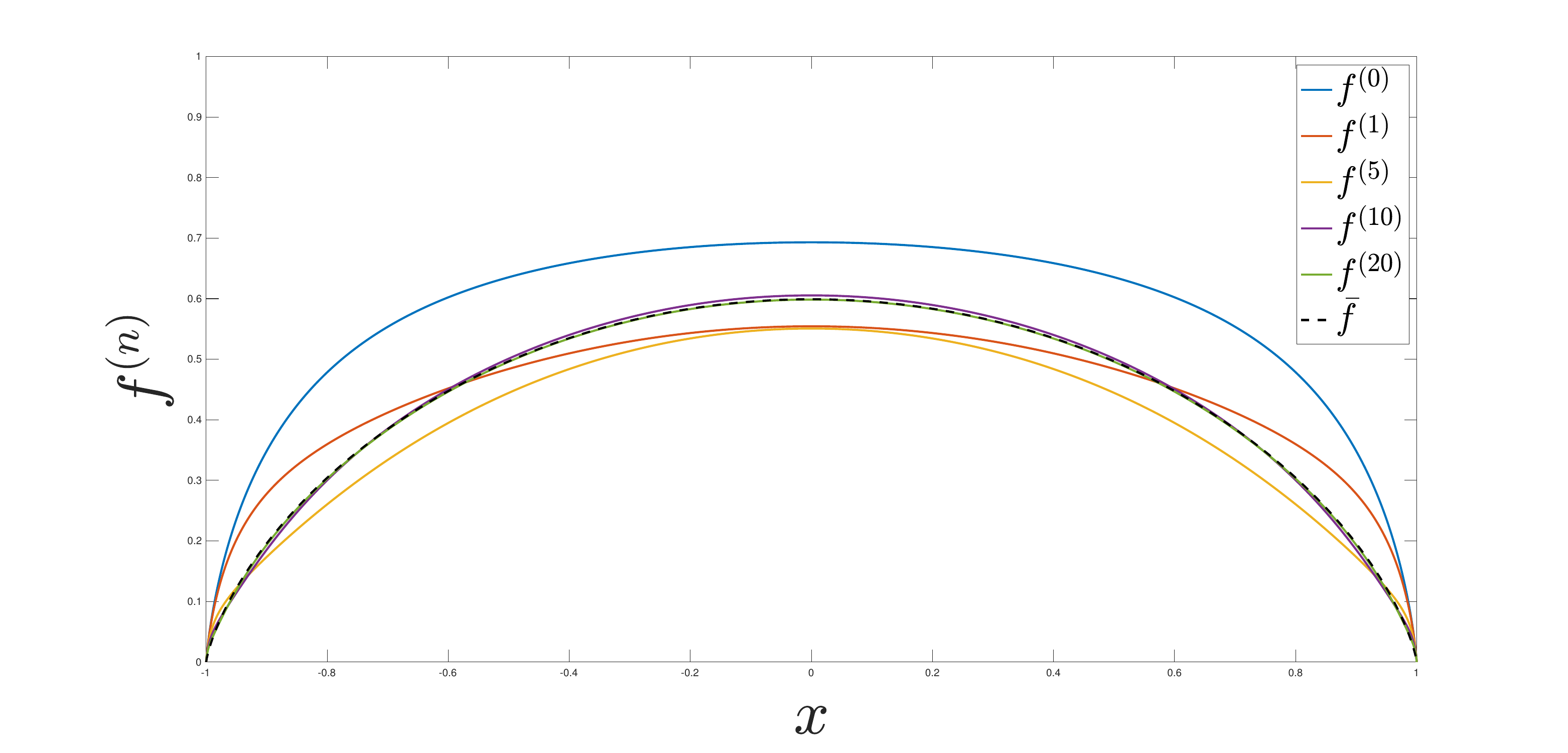}
        \caption{\small $f^{(0)}(x) = (1-x^2)\log\left(1+\frac{1}{1-x^2}\right)$}
    \end{subfigure}
    \begin{subfigure}[b]{0.45\textwidth}
        \includegraphics[width=1\textwidth]{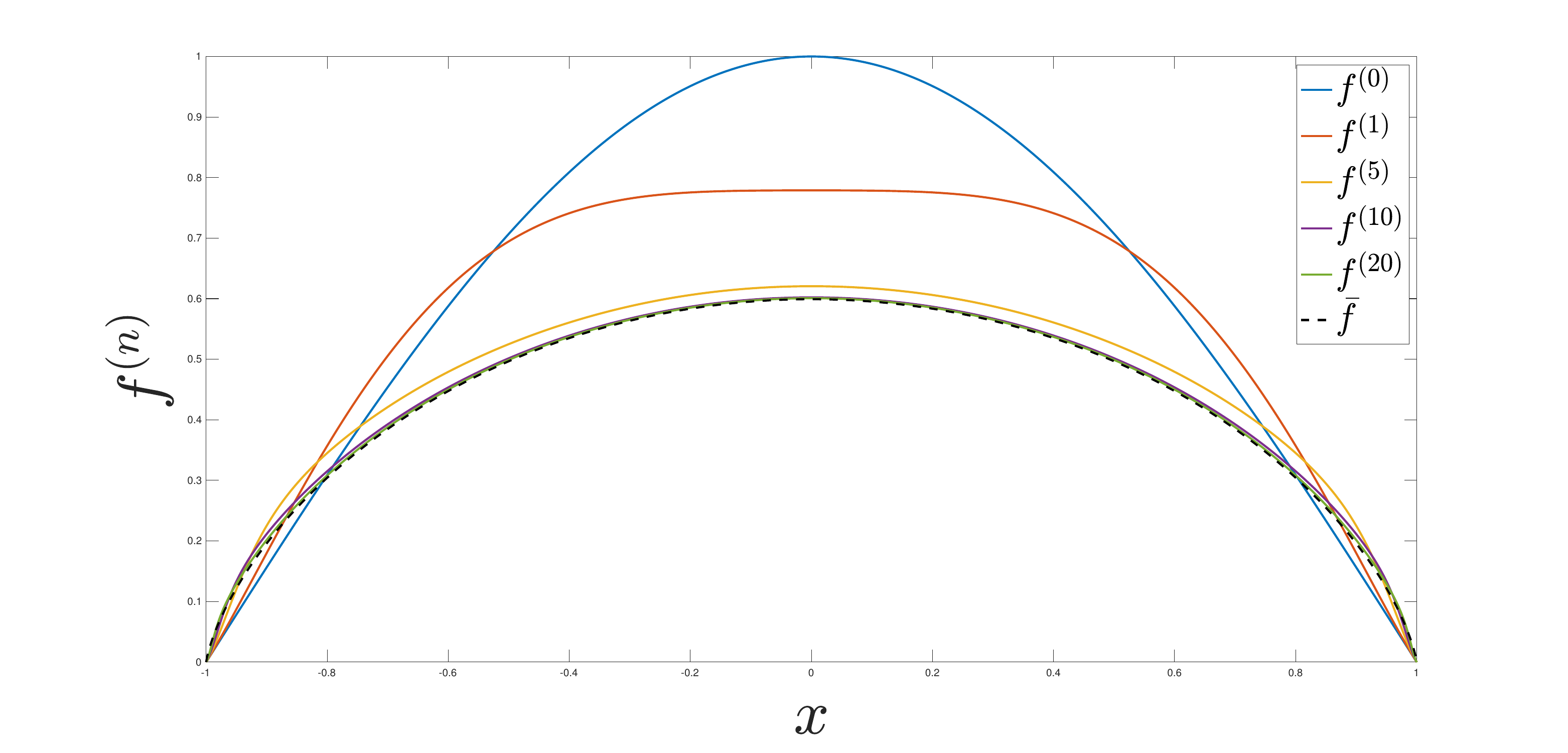}
        \caption{\small $f^{(0)}(x) = \cos(\pi x/2)$}
    \end{subfigure}
    \begin{subfigure}[b]{0.45\textwidth}
        \includegraphics[width=1\textwidth]{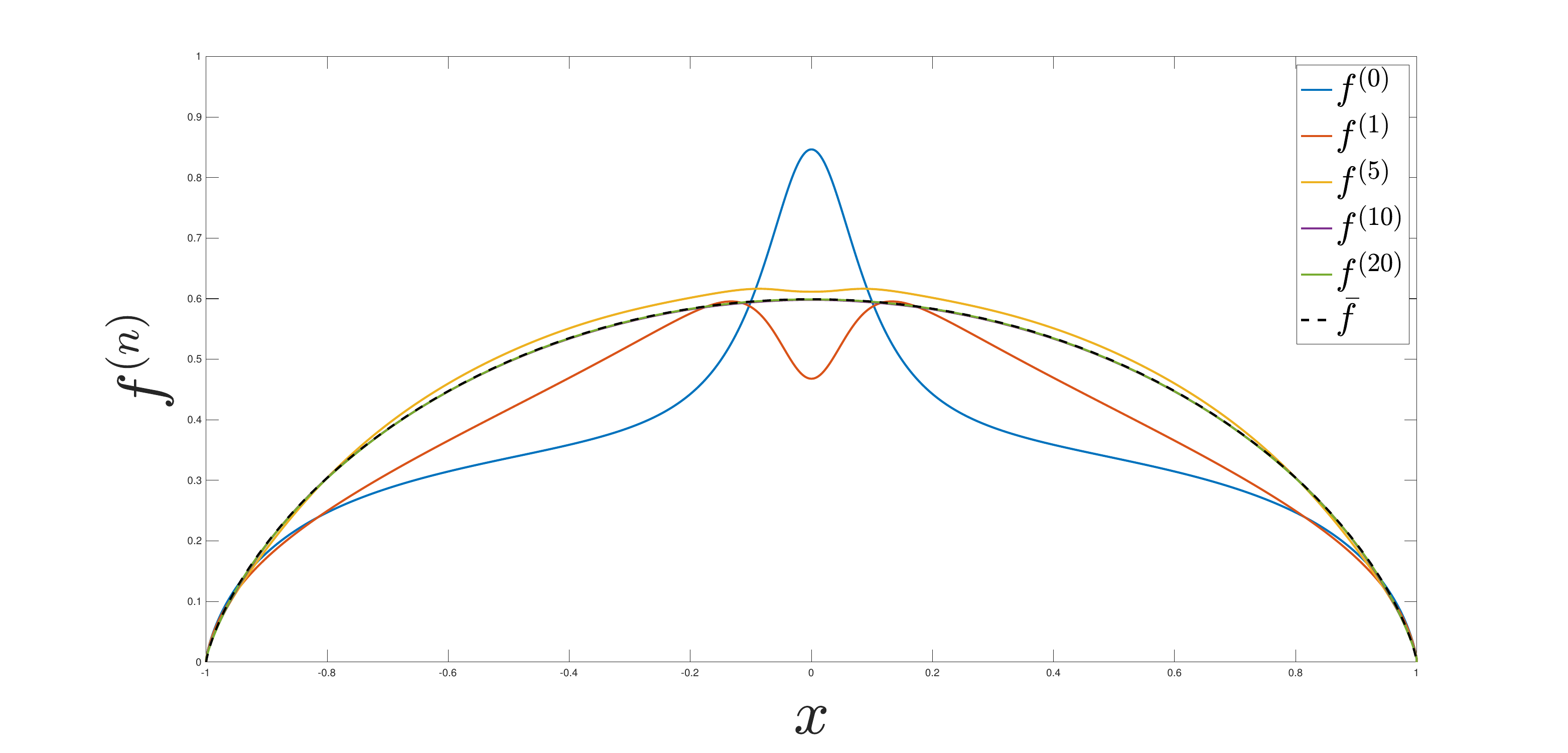}
        \caption{\small $f^{(0)}(x) = \frac{1}{2(1+100x^2)} + \frac{1-x^2}{2}\log\left(1+\frac{1}{1-x^2}\right)$}
    \end{subfigure}
    \begin{subfigure}[b]{0.45\textwidth}
        \includegraphics[width=1\textwidth]{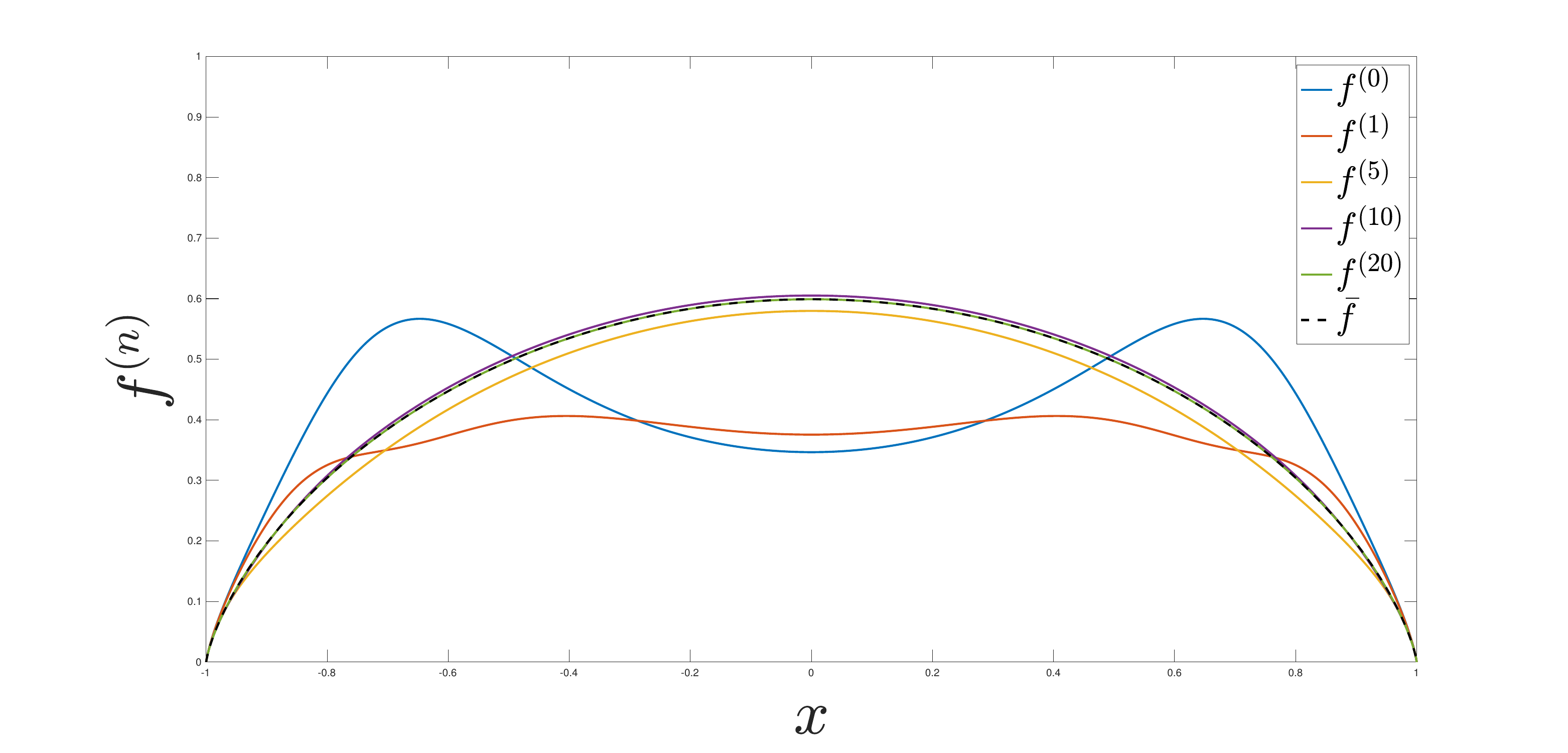}
        \caption{\small $f^{(0)}(x) = \frac{1-x^2}{1+(2x^2-1)^2}\log\left(1+\frac{1}{1-x^2}\right)$}
    \end{subfigure}
    \caption[Numerical]{Convergence of $f^{(n)}$ for different initial guesses $f^{(0)}$. The dashed curve $\bar f$ is an accurate approximate solution with its (grid-value) maximum residual $|\mtx{P}(\bar f)-\bar f|$ being smaller than $10^{-8}$.}
    \label{fig:numerical}
\end{figure}

\subsection*{Acknowledgements} The authors are supported by the National Key R\&D Program of China under the grant 2021YFA1001500. We are grateful for the stimulating discussions with Alexey Cheskidov, Tarek Elgindi, In-Jee Jeong, Peicong Song, Piotr Kokocki, Elaine Cozzi, and Farina Siddiqua hosted by the 2023 AIM workshop ``Small scale dynamics in incompressible fluid flows'', where the initial ideas of this work were formed. We would like to thank Peicong Song for performing some preliminary numerical computation that suggested the existence of a fixed-point solution. We also thank Hui Yu and Fanghua Lin for helpful discussions.

\appendix

\section{Useful formulas}\label{apx:formulas}
Given a function $f\in \mathbb{M}_0$, with $\phi(\vx) = \phi(\vx;f)$ defined in \eqref{eqt:phi_definition_2}, we have
\begin{equation}\label{eqt:phi}
\begin{split}
\phi(x_1,x_2)  &= \frac{1}{2\pi}\int_{-1}^1\idiff y_1\int_0^{f(y_1)} \frac{1}{2}\ln\frac{(x_1-y_1)^2+(x_2+y_2)^2}{(x_1-y_1)^2+(x_2-y_2)^2}\idiff y_2\\
&= \frac{1}{2\pi}\int_{-1}^1 \left\{ \frac{x_2+f(y_1)}{2}\ln\frac{(x_1-y_1)^2+(x_2+f(y_1))^2}{(x_1-y_1)^2+x_2^2} \right.\\
&\qquad\qquad\qquad +\left. \frac{x_2-f(y_1)}{2}\ln\frac{(x_1-y_1)^2+(x_2-f(y_1))^2}{(x_1-y_1)^2+x_2^2}\right.  \\
&\qquad\qquad\qquad + \left.  (x_1-y_1)\left(\arctan\left(\frac{x_2+f(y_1)}{x_1-y_1}\right) + \arctan\left(\frac{x_2-f(y_1)}{x_1-y_1}\right) \right. \right. \\
&\qquad\qquad\qquad\qquad\qquad\qquad \left. - \left. 2\arctan\left(\frac{x_2}{x_1-y_1}\right) \right)\right\} \idiff y_1,
\end{split}
\end{equation}

\begin{equation}\label{eqt:phi_x1}
\begin{split}
\phi_{x_1}(x_1,x_2) &= \frac{1}{2\pi}\int_{-1}^1\idiff y_1\int_0^{f(y_1)} \left(\frac{x_1-y_1}{(x_1-y_1)^2+(x_2+y_2)^2} -\frac{x_1-y_1}{(x_1-y_1)^2+(x_2-y_2)^2} \right)\idiff y_2\\
&= \frac{1}{2\pi}\int_{-1}^1\left\{\arctan\left(\frac{x_2+f(y_1)}{x_1-y_1}\right) + \arctan\left(\frac{x_2-f(y_1)}{x_1-y_1}\right) - 2\arctan\left(\frac{x_2}{x_1-y_1}\right)\right\}\idiff y_1,
\end{split}
\end{equation}

\begin{equation}\label{eqt:phi_x2}
\begin{split}
\phi_{x_2}(x_1,x_2) &= \frac{1}{2\pi}\int_{-1}^1\idiff y_1\int_0^{f(y_1)} \left(\frac{x_2+y_2}{(x_1-y_1)^2+(x_2+y_2)^2} -\frac{x_2-y_2}{(x_1-y_1)^2+(x_2-y_2)^2} \right)\idiff y_2\\
&= \frac{1}{2\pi}\int_{-1}^1\frac{1}{2}\ln\frac{\big((x_1-y_1)^2+(x_2+f(y_1))^2\big)\big((x_1-y_1)^2+(x_2-f(y_1))^2\big)}{\big((x_1-y_1)^2 + x_2^2\big)^2} \idiff y_1,
\end{split}
\end{equation}
and
\begin{equation}\label{eqt:cf}
c(f) = \frac{1}{2\pi}\int_{-1}^1\idiff y_1\int_0^{f(y_1)}\frac{2y_2}{(1-y_1)^2+y_2^2}\idiff y_2 = \frac{1}{2\pi}\int_{-1}^1 \ln \frac{(1-y_1)^2 + f(y_1)^2}{(1-y_1)^2}\idiff y_1.
\end{equation}

If we further assume that $f$ is strictly decreasing on $[0,1]$, so its inverse function $f^{-1}$ is well defined on $[0,f(0)]$, then we also have
\begin{equation}\label{eqt:phi_inverse}
\phi(x_1,x_2)  = \frac{1}{2\pi}\int_0^{f(0)}\idiff y_2 \int_{-f^{-1}(y_2)}^{f^{-1}(y_2)}\frac{1}{2}\ln\frac{(x_1-y_1)^2+(x_2+y_2)^2}{(x_1-y_1)^2+(x_2-y_2)^2}\idiff y_1,
\end{equation}

\begin{equation}\label{eqt:phi_x1_inverse}
\begin{split}
\phi_{x_1}(x_1,x_2) &= \frac{1}{2\pi}\int_0^{f(0)}\idiff y_2 \int_{-f^{-1}(y_2)}^{f^{-1}(y_2)}  \left(\frac{x_1-y_1}{(x_1-y_1)^2+(x_2+y_2)^2} -\frac{x_1-y_1}{(x_1-y_1)^2+(x_2-y_2)^2} \right)\idiff y_1\\
&= \frac{1}{2\pi}\int_0^{f(0)}\frac{1}{2}\ln\frac{\big((x_1-f^{-1}(y_2))^2+(x_2-y_2)^2\big)\big((x_1+f^{-1}(y_2))^2+(x_2+y_2)^2\big)}{\big((x_1-f^{-1}(y_2))^2+(x_2+y_2)^2\big)\big((x_1+f^{-1}(y_2))^2+(x_2-y_2)^2\big)}\idiff y_2,
\end{split}
\end{equation}

\begin{equation}\label{eqt:phi_x2_inverse}
\begin{split}
\phi_{x_2}(x_1,x_2) &= \frac{1}{2\pi}\int_0^{f(0)}\idiff y_2 \int_{-f^{-1}(y_2)}^{f^{-1}(y_2)} \left(\frac{x_2+y_2}{(x_1-y_1)^2+(x_2+y_2)^2} -\frac{x_2-y_2}{(x_1-y_1)^2+(x_2-y_2)^2} \right)\idiff y_1\\
&= \frac{1}{2\pi}\int_0^{f(0)} \left\{ \arctan\left(\frac{x_1-f^{-1}(y_2)}{x_2-y_2}\right) - \arctan\left(\frac{x_1-f^{-1}(y_2)}{x_2+y_2}\right)\right.\\
&\qquad\qquad\qquad\quad \left. - \arctan\left(\frac{x_1+f^{-1}(y_2)}{x_2-y_2}\right) + \arctan\left(\frac{x_1+f^{-1}(y_2)}{x_2+y_2}\right)\right\}\idiff y_2,
\end{split}
\end{equation}
and
\begin{equation}\label{eqt:cf_inverse}
\begin{split}
c(f) &= \frac{1}{2\pi}\int_0^{f(0)}\idiff y_2 \int_{-f^{-1}(y_2)}^{f^{-1}(y_2)}\frac{2y_2}{(1-y_1)^2+y_2^2}\idiff y_1 \\
&= \frac{1}{\pi}\int_0^{f(0)}\left\{\arctan\left(\frac{f^{-1}(y_2)-1}{y_2}\right) + \arctan\left(\frac{f^{-1}(y_2)+1}{y_2}\right)\right\}\idiff y_2.
\end{split}
\end{equation}

\section{Proof of Lemma \ref{lem:F_lower_bound}}\label{apx:proofs}

\begin{proof}[Proof of Lemma \ref{lem:F_lower_bound}]
For notation simplicity, write $x_1=x$, $x_2=f(x)$, $F = F(\cdot\,,\cdot\,;f)$, and $\phi=\phi(\cdot\,,\cdot\,;f)$. Since $f\in \mathbb{M}_1$, $f^{-1}(y)$ is well-defined and decreasing in $y$ for $y\in[0,f(0)]$. We invoke \eqref{eqt:phi_inverse} to obtain
\begin{align*}
\phi(x_1,x_2) &= \frac{1}{2\pi}\int_0^{f(0)}\idiff y_2 \int_{-f^{-1}(y_2)}^{f^{-1}(y_2)}\frac{1}{2}\ln\frac{(x_1-y_1)^2+(x_2+y_2)^2}{(x_1-y_1)^2+(x_2-y_2)^2}\idiff y_1\\
&= \frac{1}{2\pi}\int_0^{f(0)}\idiff y_2 \int_{-f^{-1}(y_2)}^{f^{-1}(y_2)}\idiff y_1\int_{x_2-y_2}^{x_2+y_2}\frac{z}{(x_1-y_1)^2+z^2}\idiff z\\
&= \frac{1}{2\pi}\int_0^{f(0)}\idiff y_2 \int_{x_2-y_2}^{x_2+y_2}\idiff z\int_{-f^{-1}(y_2)}^{f^{-1}(y_2)}\frac{z}{(x_1-y_1)^2+z^2}\idiff y_1\\
&= \frac{1}{2\pi}\int_0^{f(0)}\idiff y_2 \int_{x_2-y_2}^{x_2+y_2} \left\{\arctan\left(\frac{f^{-1}(y_2)-x_1}{z}\right)+\arctan\left(\frac{f^{-1}(y_2)+x_1}{z}\right)\right\} \idiff z.
\end{align*}
Combining this and \eqref{eqt:cf_inverse} yields 
\begin{align*}
F(x_1,x_2) &= \frac{\phi(x_1,x_2)}{x_2} - c(f)\\
&= \frac{1}{\pi }\int_0^{f(0)}\idiff y_2 \int_{x_2-y_2}^{x_2+y_2} \frac{1}{2x_2}\left\{\arctan\left(\frac{f^{-1}(y_2)-x_1}{z}\right)+\arctan\left(\frac{f^{-1}(y_2)+x_1}{z}\right)\right\} \idiff z\\
&\quad - \frac{1}{\pi}\int_0^{f(0)}\left\{\arctan\left(\frac{f^{-1}(y_2)-1}{y_2}\right) + \arctan\left(\frac{f^{-1}(y_2)+1}{y_2}\right)\right\}\idiff y_2\\
&=: \frac{1}{\pi}\left(J_1 + J_2\right).
\end{align*}
The last line above is a decomposition of $F(x_1,x_2)$ by splitting the integrals into two parts: $J_1$ corresponds to the integral for $y_2\in[0,x_2]$, and $J_2$ for $y_2\in[x_2,f(0)]$. We first control
\begin{align*}
J_1 &:= \int_0^{x_2}\idiff y_2\int_{x_2-y_2}^{x_2+y_2}  \frac{1}{2x_2}\left\{\arctan\left(\frac{f^{-1}(y_2)-x_1}{z}\right)+\arctan\left(\frac{f^{-1}(y_2)+x_1}{z}\right)\right\} \idiff z\\
&\quad\ - \int_0^{x_2}\left\{\arctan\left(\frac{f^{-1}(y_2)-1}{y_2}\right) + \arctan\left(\frac{f^{-1}(y_2)+1}{y_2}\right)\right\}\idiff y_2.
\end{align*}
Note that $z\mapsto \arctan(z^{-1})$ is convex on $[0,+\infty)$. Then, since $f^{-1}(y_2)\in [x_1,1]$ for $y_2\in[0,x_2]$,
\begin{align*}
J_1 &\geq \int_0^{x_2} \frac{y_2}{x_2}\left\{\arctan\left(\frac{f^{-1}(y_2)-x_1}{x_2}\right)+\arctan\left(\frac{f^{-1}(y_2)+x_1}{x_2}\right)\right\}\idiff y_2 \\
&\quad - \int_0^{x_2}\left\{\arctan\left(\frac{f^{-1}(y_2)-1}{y_2}\right) + \arctan\left(\frac{f^{-1}(y_2)+1}{y_2}\right)\right\}\idiff y_2\\
&= \int_0^{x_2} \frac{y_2}{x_2}\left\{\arctan\left(\frac{f^{-1}(y_2)-x_1}{x_2}\right)+\arctan\left(\frac{1-f^{-1}(y_2)}{y_2}\right)+\arctan\left(\frac{f^{-1}(y_2)+x_1}{x_2}\right)\right\}\idiff y_2 \\
&\quad + \int_0^{x_2}\left\{\left(1-\frac{y_2}{x_2}\right)\arctan\left(\frac{1-f^{-1}(y_2)}{y_2}\right) - \arctan\left(\frac{f^{-1}(y_2)+1}{y_2}\right)\right\}\idiff y_2\\
&\geq \int_0^{x_2} \frac{y_2}{x_2}\left\{\arctan\left(\frac{f^{-1}(y_2)-x_1}{x_2}\right)+\arctan\left(\frac{1-f^{-1}(y_2)}{x_2}\right)+\arctan\left(\frac{f^{-1}(y_2)+x_1}{x_2}\right)\right\}\idiff y_2 \\
&\quad + \int_0^{x_2}\left\{\left(1-\frac{y_2}{x_2}\right)\arctan\left(\frac{1-f^{-1}(y_2)}{y_2}\right) - \frac{\pi}{2}\right\}\idiff y_2.
\end{align*}
Furthermore, we notice that
\[\arctan\left(\frac{f^{-1}(y_2)-x_1}{x_2}\right) + \arctan\left(\frac{1- f^{-1}(y_2)}{x_2}\right)\geq \arctan\left(\frac{1-x_1}{x_2}\right),\]
and 
\[\arctan\left(\frac{1}{z}\right)\leq \frac{C}{1+z},\quad z>0,\]
for some absolute constant $C>0$. Hence, 
\begin{align*}
J_1 &\geq \int_0^{x_2} \frac{y_2}{x_2}\left\{\pi - \arctan\left(\frac{x_2}{1-x_1}\right) - \arctan\left(\frac{x_2}{f^{-1}(y_2) + x_1}\right)\right\}\idiff y_2 \\
&\quad + \int_0^{x_2}\left\{\left(1-\frac{y_2}{x_2}\right)\arctan\left(\frac{1-f^{-1}(y_2)}{y_2}\right) - \frac{\pi}{2}\right\}\idiff y_2\\
&\geq - \int_0^{x_2} \frac{y_2}{x_2}\left\{  \frac{Cx_2}{1-x_1+x_2} + \frac{Cx_2}{f^{-1}(y_2)+x_1+x_2}\right\}\idiff y_2 +\int_0^{x_2}\left(1-\frac{y_2}{x_2}\right)\arctan\left(\frac{1-f^{-1}(y_2)}{y_2}\right) \idiff y_2\\
&\geq \int_0^{x_2}\left(1-\frac{y_2}{x_2}\right)\arctan\left(\frac{1-f^{-1}(y_2)}{y_2}\right) \idiff y_2 - \frac{Cx_2^2}{\max\{1-x_1,x_2\}} -C\int_0^{x_2}\frac{x_2}{f^{-1}(y_2)+x_1+x_2}\idiff y_2.
\end{align*}

Secondly, we bound $J_2$ from below as
\begin{align*}
J_2 &:= \int_{x_2}^{f(0)}\idiff y_2\int_{x_2-y_2}^{x_2+y_2}  \frac{1}{2x_2}\left\{\arctan\left(\frac{f^{-1}(y_2)-x_1}{z}\right)+\arctan\left(\frac{f^{-1}(y_2)+x_1}{z}\right)\right\} \idiff z\\
&\quad\ - \int_{x_2}^{f(0)}\left\{\arctan\left(\frac{f^{-1}(y_2)-1}{y_2}\right) + \arctan\left(\frac{f^{-1}(y_2)+1}{y_2}\right)\right\}\idiff y_2\\
&= \int_{x_2}^{f(0)}\idiff y_2\int_{y_2-x_2}^{y_2+x_2}  \frac{1}{2x_2}\left\{\arctan\left(\frac{f^{-1}(y_2)-x_1}{z}\right)+\arctan\left(\frac{f^{-1}(y_2)+x_1}{z}\right)\right\} \idiff z\\
&\quad - \int_{x_2}^{f(0)}\left\{\arctan\left(\frac{f^{-1}(y_2)-1}{y_2}\right) + \arctan\left(\frac{f^{-1}(y_2)+1}{y_2}\right)\right\}\idiff y_2\\
& = \int_{x_2}^{f(0)}\idiff y_2\int_{y_2-x_2}^{y_2+x_2}  \frac{1}{2x_2}\left\{\arctan\left(\frac{f^{-1}(y_2)+x_1}{z}\right)-\arctan\left(\frac{x_1-f^{-1}(y_2)}{z}\right) \right. \\
&\qquad\qquad\qquad\qquad\qquad\qquad \left. +\arctan\left(\frac{1-f^{-1}(y_2)}{z}\right)-\arctan\left(\frac{f^{-1}(y_2)+1}{z}\right) \right\} \idiff z\\
&\quad + \left[\int_{x_2}^{f(0)}\idiff y_2\int_{y_2-x_2}^{y_2+x_2}  \frac{1}{2x_2}\left\{\arctan\left(\frac{f^{-1}(y_2)+1}{z}\right)- \arctan\left(\frac{1-f^{-1}(y_2)}{z}\right)\right\} \idiff z\right.\\
&\qquad\quad\left. - \int_{x_2}^{f(0)}\left\{ \arctan\left(\frac{f^{-1}(y_2)+1}{y_2}\right) - \arctan\left(\frac{1-f^{-1}(y_2)}{y_2}\right)\right\}\idiff y_2\right]\\
&=: J_{2,1} + J_{2,2}.
\end{align*}
We control $J_{2,1}$ as
\begin{align*}
J_{2,1} &= \int_{x_2}^{f(0)}\idiff y_2\int_{y_2-x_2}^{y_2+x_2}  \frac{1}{2x_2}\left\{\arctan\left(\frac{f^{-1}(y_2)+x_1}{z}\right)-\arctan\left(\frac{x_1-f^{-1}(y_2)}{z}\right) \right. \\
&\qquad\qquad\qquad\qquad\qquad\qquad \left. +\arctan\left(\frac{1-f^{-1}(y_2)}{z}\right)-\arctan\left(\frac{f^{-1}(y_2)+1}{z}\right) \right\} \idiff z\\
&= \int_{x_2}^{f(0)}\idiff y_2\int_{y_2-x_2}^{y_2+x_2}  \frac{1}{2x_2}\left\{\arctan\left(\frac{2f^{-1}(y_2)z}{x_1^2-f^{-1}(y_2)^2+z^2}\right)-\arctan\left(\frac{2f^{-1}(y_2)z}{1-f^{-1}(y_2)^2+z^2}\right) \right\} \idiff z\\
&= \int_{x_2}^{f(0)}\idiff y_2\int_{y_2-x_2}^{y_2+x_2} \frac{1}{2x_2}\arctan\left(\frac{2(1-x_1^2)f^{-1}(y_2)z}{(x_1^2-f^{-1}(y_2)^2+z^2)(1-f^{-1}(y_2)^2+z^2) + 4f^{-1}(y_2)^2z^2}\right)\idiff z\\
&\geq \int_{x_2}^{f(0)}\idiff y_2\int_{y_2-x_2}^{y_2+x_2} \frac{1}{2x_2}\arctan\left(\frac{\tilde C(1-x_1)f^{-1}(y_2)z}{(1-f^{-1}(y_2))^2 + y_2^2}\right)\idiff z,
\end{align*}
where we have used $f^{-1}(y_2)\leq x_1\leq 1$ and $0\leq z\leq y_2+x_2\leq 2y_2\leq 2M$ for $y_2\in[x_2,f(0)]\subset[0,M]$. Since
\[\frac{(1-x_1)f^{-1}(y_2)z}{(1-f^{-1}(y_2))^2 + y_2^2}\leq \frac{2(1-x_1)y_2}{(1-x_1)^2 + y_2^2}\leq 1,\]
it follows that
\[
J_{2,1} \geq \int_{x_2}^{f(0)}\idiff y_2\int_{y_2-x_2}^{y_2+x_2} \frac{1}{2x_2}\cdot \frac{\tilde C(1-x_1)f^{-1}(y_2)z}{(1-f^{-1}(y_2))^2 + y_2^2}\idiff z \geq \tilde C \int_{x_2}^{f(0)}\frac{(1-x_1)f^{-1}(y_2)y_2}{(1-f^{-1}(y_2))^2 + y_2^2}\idiff y_2.
\]
Next, we use the convexity of $z\mapsto \arctan(z^{-1})$ on $[0,+\infty)$ again to bound $J_{2,2}$ as 
\begin{align*}
J_{2,2} &= \int_{x_2}^{f(0)}\idiff y_2\int_{y_2-x_2}^{y_2+x_2}  \frac{1}{2x_2}\left\{\arctan\left(\frac{1+f^{-1}(y_2)}{z}\right)- \arctan\left(\frac{1-f^{-1}(y_2)}{z}\right)\right\} \idiff z\\
&\quad - \int_{x_2}^{f(0)}\left\{ \arctan\left(\frac{1+f^{-1}(y_2)}{y_2}\right) - \arctan\left(\frac{1-f^{-1}(y_2)}{y_2}\right)\right\}\idiff y_2\\
&= \int_{x_2}^{f(0)}\idiff y_2\int_0^{x_2}  \frac{1}{2x_2}\left\{\arctan\left(\frac{1+f^{-1}(y_2)}{y_2+z}\right)+\arctan\left(\frac{1+f^{-1}(y_2)}{y_2-z}\right) \right. \\
&\qquad\qquad\qquad\qquad\qquad\quad - \left.\arctan\left(\frac{1-f^{-1}(y_2)}{y_2+z}\right)  -\arctan\left(\frac{1-f^{-1}(y_2)}{y_2-z}\right)\right\} \idiff z\\
&\quad - \int_{x_2}^{f(0)}\left\{ \arctan\left(\frac{1+f^{-1}(y_2)}{y_2}\right) - \arctan\left(\frac{1-f^{-1}(y_2)}{y_2}\right)\right\}\idiff y_2\\
&\geq -\int_{x_2}^{f(0)}\frac{1}{2}\left\{\arctan\left(\frac{1-f^{-1}(y_2)}{y_2+x_2}\right)  +\arctan\left(\frac{1-f^{-1}(y_2)}{y_2-x_2}\right) - 2\arctan\left(\frac{1-f^{-1}(y_2)}{y_2}\right) \right\}\idiff y_2.
\end{align*}
With $a=1-f^{-1}(y_2)>0$, we compute that
\begin{align*}
&\arctan\left(\frac{a}{y_2+x_2}\right) + \arctan\left(\frac{a}{y_2-x_2}\right) -2\arctan\left(\frac{a}{y_2}\right)\\
&= \arctan\left(\frac{y_2}{a}\right) -\arctan\left(\frac{y_2-x_2}{a}\right) - \left(\arctan\left(\frac{y_2+x_2}{a}\right)-\arctan\left(\frac{y_2}{a}\right)\right)\\
&= \arctan\left(\frac{ax_2}{a^2 + y_2(y_2-x_2)}\right) - \arctan\left(\frac{ax_2}{a^2 + y_2(y_2+x_2)}\right)\\
&= \arctan\left(\frac{a^2 + y_2(y_2+x_2)}{ax_2}\right) - \arctan\left(\frac{a^2 + y_2(y_2-x_2)}{ax_2}\right)\\
&= \arctan\left(\frac{2ax_2^2y_2}{a^2x_2^2 + (a^2 + y_2(y_2+x_2))(a^2 + y_2(y_2-x_2))}\right)\\
&= \arctan\left(\frac{2ax_2^2y_2}{a^2x_2^2 + a^4+2a^2y_2^2+ y_2^2(y_2^2-x_2^2)}\right)\\
&\leq \left\{
\begin{array}{ll}
\displaystyle \arctan\left(\frac{Cx_2}{a}\right)\leq C\min\left\{\frac{x_2}{a}\,,\,1\right\}, & y_2\in[x_2,2x_2],\\\\
\displaystyle \arctan\left(\frac{Cx_2^2}{a^2 + y_2^2}\right)\leq \frac{Cx_2^2}{a^2 + y_2^2}, & y_2\geq 2x_2.
\end{array}
\right.
\end{align*}
Hence,
\begin{align*}
J_{2,2} &\geq -C\int_{x_2}^{\min\{2x_2,f(0)\}}\min\left\{\frac{x_2}{1-f^{-1}(y_2)}\,,\,1\right\}\idiff y_2 -C\int_{\min\{2x_2,f(0)\}}^{f(0)}\frac{x_2^2}{(1-f^{-1}(y_2))^2 + y_2^2}\idiff y_2\\
&\geq -C\int_{x_2}^{2x_2}\min\left\{\frac{x_2}{1-x_1}\,,\,1\right\}\idiff y_2 - C\int_{x_2}^{f(0)}\frac{x_2^2}{(1-x_1)^2 + y_2^2}\idiff y_2\\
&\geq -C\frac{x_2^2}{\max\{1-x_1,x_2\}}-C\frac{x_2^2}{1-x_1}\left(\arctan\left(\frac{f(0)}{1-x_1}\right) - \arctan\left(\frac{x_2}{1-x_1}\right)\right)\\
&= -C\frac{x_2^2}{\max\{1-x_1,x_2\}}-C\frac{x_2^2}{1-x_1}\left(\arctan\left(\frac{1-x_1}{x_2}\right)-\arctan\left(\frac{1-x_1}{f(0)}\right)\right)\\
&\geq -C\frac{x_2^2}{\max\{1-x_1,x_2\}}-C\frac{x_2^2}{1-x_1}\arctan\left(\frac{1-x_1}{x_2}\right)\\
&\geq -C\frac{x_2^2}{\max\{1-x_1,x_2\}}.
\end{align*}

Finally, combining all the preceding estimates, we obtain
\begin{align*}
F(x_1,x_2) &= \frac{1}{\pi}(J_1 + J_{2,1} + J_{2,2})\\
&\geq \tilde C\int_0^{x_2}\left(1-\frac{y_2}{x_2}\right)\arctan\left(\frac{1-f^{-1}(y_2)}{y_2}\right)\idiff y_2 + \tilde C\int_{x_2}^{f(0)}\frac{(1-x_1)f^{-1}(y_2)y_2}{(1-f^{-1}(y_2))^2+y_2^2}\idiff y_2\\
&\quad - C \frac{x_2^2}{\max\{1-x_1,x_2\}} - C \int_0^{x_2}\frac{x_2}{x_1+x_2+f^{-1}(y_2)}\idiff y_2,
\end{align*}
as desired.
\end{proof}

\bibliographystyle{myalpha}
\bibliography{reference}

\end{document}

%% file: macro-3.tex
\usepackage[usenames,dvipsnames]{xcolor}
\usepackage{fancyhdr}
\usepackage{amsmath,amsfonts,amsbsy,amsgen,amscd,amssymb,amsthm}
\usepackage{url}

\usepackage{mathtools}

\usepackage[font=small,margin=0.25in,labelfont={bf},labelsep={colon}]{caption}

\usepackage{subcaption}

\usepackage{tikz}
\usepackage{microtype}
\usepackage{enumitem}

\definecolor{dark-gray}{gray}{0.3}
\definecolor{dkgray}{rgb}{.4,.4,.4}
\definecolor{dkblue}{rgb}{0,0,.5}
\definecolor{medblue}{rgb}{0,0,.75}
\definecolor{rust}{rgb}{0.5,0.1,0.1}

\usepackage[colorlinks=true]{hyperref}

\hypersetup{urlcolor=rust}
\hypersetup{citecolor=dkblue}
\hypersetup{linkcolor=dkblue}

\usepackage{setspace}

\usepackage{graphicx}
\usepackage{booktabs,longtable,tabu} 
\usepackage{multirow} 

\usepackage{float}

\usepackage[full]{textcomp}

\usepackage[scaled=.98,sups]{XCharter}
\usepackage[scaled=1.04,varqu,varl]{inconsolata}
\usepackage[type1]{cabin}
\usepackage[charter,vvarbb,scaled=1.07]{newtxmath}
\usepackage[cal=boondoxo]{mathalfa}
\usepackage[T1]{fontenc}

\usepackage{bm} 

\graphicspath{{figures/}}

\newtheorem{theorem}{Theorem}[section]
\newtheorem{lemma}[theorem]{Lemma}

\newtheorem{proposition}[theorem]{Proposition}

\newtheorem{corollary}[theorem]{Corollary}

\theoremstyle{definition}

\numberwithin{equation}{section} 
\numberwithin{figure}{section}
\numberwithin{table}{section}

\floatstyle{plaintop}
\newfloat{recipe}{thp}{lor}
\floatname{recipe}{Recipe}
\numberwithin{recipe}{section}

\providecommand{\mathbold}[1]{\bm{#1}}  
                                
\newcommand{\econst}{\mathrm{e}}

\newcommand{\Id}{\mathbf{I}}

\providecommand{\mathbbm}{\mathbb} 

\newcommand{\R}{\mathbbm{R}}

\newcommand{\diff}[1]{\mathrm{d}{#1}}
\newcommand{\idiff}[1]{\, \diff{#1}}

\newcommand{\vct}[1]{\mathbold{#1}}
\newcommand{\mtx}[1]{\mathbold{#1}}

\newcommand{\triplenorm}[1]{{\left\vert\kern-0.25ex\left\vert\kern-0.25ex\left\vert #1
    \right\vert\kern-0.25ex\right\vert\kern-0.25ex\right\vert}}

